\documentclass[reqno,a4paper,10pt]{amsart}
\usepackage[T1]{fontenc}
\usepackage{libertine}
\usepackage[libertine]{newtxmath}
\usepackage[utf8]{inputenc}
\usepackage[english]{babel}
\usepackage{url}
\usepackage[colorlinks=true, linkcolor = blue, urlcolor=black, citecolor=blue, anchorcolor=blue]{hyperref}
\usepackage{comment}
\usepackage{xcolor,color}
\usepackage{geometry}
\usepackage[all]{xy}
\usepackage{enumitem}
\usepackage{booktabs}
\usepackage{slashed}
\usepackage{bbm}
\usepackage{cancel}
\usepackage{makebox}

\usepackage{cleveref}
\usepackage{tikz}
\usepackage{tikz-cd}
\usepackage{pgfplots}
\pgfplotsset{compat=1.17}

\usepackage{soul}

\usepackage{amsmath}
\usepackage{graphicx}
\usepackage{amsthm}
\usepackage{latexsym}
\usepackage{amsfonts}
\usepackage{bbm}
\usepackage{mathrsfs}
\usepackage{mdframed}

\numberwithin{equation}{section}

\theoremstyle{plain}
\newtheorem{lemma}{Lemma}[section]
\newtheorem{proposition}[lemma]{Proposition}
\newtheorem{proposition/definition}[lemma]{Proposition/Definition}
\newtheorem{theorem}[lemma]{Theorem}

\newtheorem{corollary}[lemma]{Corollary}

\theoremstyle{definition}
\newtheorem{definition}[lemma]{Definition}
\newtheorem{example}[lemma]{Example}

\theoremstyle{remark}  
\newtheorem{remark}[lemma]{Remark}

\DeclareMathOperator{\id}{id}
\DeclareMathOperator{\im}{im}
\DeclareMathOperator{\rank}{rank}

\DeclareMathOperator{\gr}{Gr}

\DeclareMathAlphabet{\mathpzc}{OT1}{pzc}{m}{it}

\newcommand{\good}{\text{$\calF$}}
\newcommand{\Fol}{\text{$\mathpzc{Fol}$}}
\newcommand{\SymplFol}{\text{$\mathpzc{SymplFol}$}}
\newcommand{\RegPoiss}{\text{$\mathpzc{RegPoiss}$}}

\newcommand{\calA}{\mathcal{A}}

\newcommand{\calF}{\mathcal{F}}

\newcommand{\calI}{\mathcal{I}}

\newcommand{\calL}{\mathcal{L}}

\newcommand{\calQ}{\mathcal{Q}}
\newcommand{\calR}{\mathcal{R}}

\newcommand{\bbN}{\mathbb{N}}

\newcommand{\bbR}{\mathbb{R}}

\newcommand{\bbT}{\mathbb{T}}

\newcommand{\frakX}{\mathfrak{X}}

\newcommand{\frakl}{\mathfrak{l}}

\newcommand{\frakn}{\mathfrak{n}}

\newcommand{\frakv}{\mathfrak{v}}

\newcommand{\rmd}{\mathrm{d}}

\renewcommand{\phi}{\varphi}

\newcommand{\ldsb}{[\![}
\newcommand{\rdsb}{]\!]}
\newcommand{\ldab}{\langle\!\langle}
\newcommand{\rdab}{\rangle\!\rangle}
\renewcommand{\theta}{\vartheta}

\newcommand{\pr}{\operatorname{pr}}

\allowdisplaybreaks

\title[]{Deformations of Symplectic Foliations}

\author{Stephane Geudens}
\address{Geometry Section, Department of Mathematics, KU Leuven, Celestijnenlaan 200B - 3001 Leuven, Belgium
\newline
Current address: Max-Planck-Institut für Mathematik, Vivatsgasse 7, 53111 Bonn, Germany}
\email{\href{mailto:stephane\_geudens@hotmail.com}{\underline{\smash{stephane\_geudens@hotmail.com}}}}

\author{Alfonso G.~Tortorella}
\address{Geometry Section, Department of Mathematics, KU Leuven, Celestijnenlaan 200B - 3001 Leuven, Belgium
\newline
Current address: Centro de Matemática da Universidade do Porto, Rua do Campo Alegre 687, 4169-007 Porto, Portugal
}
\email{\href{mailto:alfonso.tortorella@fc.up.pt}{\underline{\smash{alfonso.tortorella@fc.up.pt}}}}

\author{Marco Zambon}
\address{Geometry Section, Department of Mathematics, KU Leuven, Celestijnenlaan 200B - 3001 Leuven, Belgium}
\email{\href{mailto:marco.zambon@kuleuven.be}{\underline{\smash{marco.zambon@kuleuven.be}}}}



\begin{document}

\begin{abstract}
	We develop the deformation theory of symplectic foliations, i.e.~regular foliations equipped with a leafwise symplectic form.
	The main result of this paper is that each symplectic foliation has an attached    $L_\infty$-algebra controlling its deformation problem.
	Indeed, viewing symplectic foliations as regular Poisson structures, we establish a one-to-one correspondence between the small deformations of a given symplectic foliation and the Maurer-Cartan elements of the associated $L_\infty$-algebra. 
	Using this, we show that infinitesimal deformations of symplectic foliations can be obstructed. 
	Further, we relate symplectic foliations with foliations on one side  and  with (arbitrary) Poisson structures on the other, showing that obstructed  infinitesimal deformations of the former may give rise  to unobstructed deformations of the latter. 
\end{abstract}

\maketitle

\setcounter{tocdepth}{1}

\tableofcontents

\section*{Introduction}

This paper addresses the deformation theory of symplectic foliations, i.e. foliations equipped with a leafwise two-form that is closed and non-degenerate. Equivalently, a symplectic foliation can be viewed as a regular Poisson structure; this is a constant rank bivector field $\Pi$ whose Schouten-Nijenhuis bracket $[\Pi,\Pi]_{\sf SN}$ vanishes.

These objects have received and still receive a lot of attention. An important question concerning them is that of existence, and a key result in this respect is the h-principle for symplectic foliations due to Fernandes-Frejlich~\cite{hsymplfol}. It guarantees their existence on a given manifold under certain assumptions; more specifically, it shows that a regular bivector field $\Pi$ on an open manifold $M$ is homotopic, {through} regular bivector fields, to a regular Poisson structure iff the distribution $\im\Pi^{\sharp}$ is homotopic to an involutive distribution. Other interesting results about symplectic foliations concern those of codimension-one (e.g.~the existence problem on $S^{5}$~\cite{Mit}) and normal form statements~\cite{CM}, to name a few.
  
\bigskip
\paragraph{\bf {A local parametrization of  symplectic foliations}}

Taking the Poisson geometry point of view, our aim is to describe small deformations of a given regular Poisson structure $\Pi$. Following the common strategy in deformation theory, this is done by constructing a suitable {$L_{\infty}[1]$-algebra} whose Maurer-Cartan elements parametrize small deformations of $\Pi$. {(Recall  that $L_{\infty}[1]$-algebras are higher generalizations of Lie algebras, up to a degree shift~\cite{LadaMarkl}.)
}
Clearly, this requires us to take care of the regularity condition and the Poisson integrability condition simultaneously, hence the main difficulty lies in finding a local parametrization for the space {of} bivector fields
 which behaves well with respect to both conditions. 

\bigskip
 
{
We now describe such a parametrization. Recall first the following fact from linear algebra, which underlies the correspondence between Poisson structures and symplectic {foliations}:
given a  vector space $V$,  the data of a bivector $\pi\in \wedge^2V$  is equivalent to a pair $(W,\omega)$, where $W\subset V$ is a subspace and $\omega\in \wedge^2W^*$ a non-degenerate skew-symmetric bilinear form on it.
For any sufficiently small $\gamma\in \wedge^2W^*$, one defines the gauge-transformed  bivector $\pi^{\gamma}$ to be the bivector corresponding to the pair $(W,\omega+\gamma)$. Now let $\Pi$ be a regular Poisson structure of rank $2k$. Pick a complement $G$ to the characteristic distribution $T\calF$ of $\Pi$ and let $\gamma\in\Omega^{2}(M)$ be the two-form defined by extending the leafwise symplectic form $\omega\in\Omega^{2}(\calF)$ by zero on $G$. 
We define  the \emph{Dirac exponential map}
\begin{equation*}
\exp_G:(\text{neighborhood of zero section in}\ \wedge^{2}TM)\rightarrow\wedge^{2}TM,\;\;\;\;\;\exp_G(Z)=\Pi+Z^{\gamma}.
\end{equation*} 
The key point is that this parametrization is compatible both with the constant rank condition and the integrability condition of Poisson structures (see Thm.~\ref{theor:nearby_regular_bivectors} and Prop.~\ref{prop:Koszul_algebra:MC_elements}):
\begin{enumerate}
    \item It linearizes the constant rank condition. Namely, for small enough $Z\in\mathfrak{X}^{2}(M)$, the bivector field $\exp_G(Z)$ has constant rank $2k$ iff $Z$ belongs to a distinguished \emph{linear} subspace $\frakX_\good^2(M)\subset\frakX^{2}(M)$.
    \item It turns the integrability condition of Poisson structures into a certain cubic equation. Namely, for any small enough bivector field $Z\in\mathfrak{X}^{2}(M)$, one has that $\exp_G(Z)$ self-commutes exactly when
    \[
    \frakl^G_1(Z)+\frac{1}{2}\frakl^G_2(Z,Z)+\frac{1}{6}{\frakl}^G_3(Z,Z,Z)=0.
    \]
Here $\frakl^G_1$ is the Poisson differential $[\Pi,-]_{SN}$, while  $\frakl^G_2$ and $\frakl^G_3$ depend also on $\gamma$ (see  Prop.~\ref{prop:Koszul_algebra}).
\end{enumerate}
Algebraically, the above equation is the Maurer-Cartan equation of an $L_{\infty}[1]$-algebra $(\frakX^\bullet(M)[2],\{\frakl^G_k\})$, as we explain now.
}

\bigskip
{
The above parametrization of regular Poisson structures by means of $\exp_G$ arises naturally when, instead of working in the Poisson category, one considers the larger category of Dirac structures.
A Dirac structure on a manifold $M$ is a maximally isotropic, involutive subbundle of the Courant algebroid $TM\oplus T^{*}M$, as we recall in Appendix~\ref{app:Courant_algebroids}. In particular, the graph of the rank $2k$ Poisson structure $\Pi$ is a Dirac structure $\gr\Pi\subset TM\oplus T^{*}M$, and deforming $\gr\Pi$ as a Dirac structure is tantamount to deforming $\Pi$ as a Poisson structure. Hence we can rely on results about the deformation theory of Dirac structures, which we recall in Appendix~\ref{app:Deformations_Dirac_Structures}.
}

Given a Dirac structure $A\subset TM\oplus T^{*}M$, a choice of complementary almost Dirac structure $B$ endows the graded vector space $\Omega^{\bullet}(A)[2]$ with an $L_{\infty}[1]$-algebra structure whose Maurer-Cartan elements parametrize the Dirac structures transverse to $B$. Moreover, different choices for $B$ produce isomorphic $L_{\infty}[1]$-algebras~\cite{GMS}. With this in mind, it seems natural to deform the Dirac structure $\gr\Pi$ making use of the complement $TM$, but this approach is not well-suited to single out regular deformations of $\Pi$. Indeed, it parametrizes the Poisson structures on $M$ by means of Maurer-Cartan elements of the familiar Koszul dgLa, via the correspondence
\begin{align*}
&\Big\{Z\in\mathfrak{X}^{2}(M):\ d_{\Pi}Z+\frac{1}{2}[Z,Z]_{\sf SN}=0\Big\}\longrightarrow\{P\in\mathfrak{X}^{2}(M):\ [P,P]_{\sf SN}=0\}:Z\mapsto \Pi+Z.
\end{align*}
The inconvenience of this approach is that the space of bivector fields $Z$ for which $\Pi+Z$ is of constant rank $2k$ does not have a vector space structure. Instead, we proceed as follows: if $G$ is a complement to the characteristic distribution $T\calF$ of $\Pi$, then the almost Dirac structure $G\oplus G^{0}$ is complementary to $\gr\Pi$. 
{Upon the canonical identification $\gr\Pi\cong T^*M$, this yields the}  
 $L_{\infty}[1]$-algebra $(\frakX^\bullet(M)[2],\{\frakl^G_k\})$ mentioned above.

At last, we combine step (1) and (2) above. We notice that the space of \emph{good multivector fields}
\[
\frakX^\bullet_\good(M)=\{W\in\frakX^\bullet(M):\iota_\alpha\iota_\beta W=0\ \text{for all}\ \alpha,\beta\in\Gamma(T^\circ\calF)\}
\]
is closed under the multibrackets $\frakl^G_k$, so  $\frakX^\bullet_\good(M)[2]\subset(\frakX^\bullet(M)[2],\{\frakl^G_k\})$ is an $L_{\infty}[1]$-subalgebra. Its Maurer-Cartan elements parametrize bivector fields near $\Pi$ that are both regular and Poisson (see Thm.~\ref{theor:deformation_theory regular_Poisson}).
\bigskip

\noindent
\textbf{Main Theorem.} \emph{There is a bijection
\begin{align*}
\big\{\text{small Maurer-Cartan elements of}\ &(\frakX^\bullet_\good(M)[2],\{\frakl^G_k\})\big\}\rightarrow \big\{\widetilde{\Pi}\in\RegPoiss^{2k}(M):\ \im\widetilde{\Pi}^{\sharp}\pitchfork G\big\}:\\
&\hspace{2.5cm}Z\mapsto\exp_{G}(Z).
\end{align*}
}

\bigskip
\paragraph{\bf {Relation to deformations of Poisson structures and of foliations}}
When dealing with 
regular Poisson structures, one has two obvious forgetful maps
\begin{equation}\label{eq:diag}
\xymatrix{
& \text{Regular Poisson structures}\ar[rd]\ar[ld]&  \\
 \text{Poisson structures}&  &  \text{Foliations}}
\end{equation}
{We ask whether these maps can be lifted to morphisms between the
algebraic structures governing the respective deformations. More precisely, we ask whether there are
$L_{\infty}[1]$-algebra morphisms whose induced maps of Maurer-Cartan elements are exactly those appearing in the above diagram. For both maps, the answer is positive, as we now explain.}

\begin{itemize}
\item[i)] For the arrow on the left in diagram~\eqref{eq:diag}, 
{this follows from} a result in~\cite{GMS}. Briefly, {as seen in the Main Theorem,} the deformation problem of a regular Poisson structure $\Pi$ is governed by the $L_{\infty}[1]$-subalgebra $(\frakX^\bullet_\good(M)[2],\{\frakl^G_k\})$ of $(\frakX^\bullet(M)[2],\{\frakl^G_k\})$, and the latter is essentially obtained deforming the Dirac structure $\gr\Pi$ using $G\oplus G^{0}$ as a complement. On the other hand, deforming $\Pi$ as a Poisson structure using the Koszul dgLa amounts to choosing $TM$ as a complement to $\gr\Pi$. By~\cite{GMS}, the $L_{\infty}[1]$-algebras obtained choosing different complements to $\gr\Pi$ are isomorphic; hence there is an $L_{\infty}[1]$-morphism relating $(\frakX^\bullet_\good(M)[2],\{\frakl^G_k\})$ with the $L_{\infty}[1]$-algebra obtained shifting degrees in the Koszul dgLa.

\item[ii)]
For the arrow on the right in diagram~\eqref{eq:diag}, we provide 
{the answer} 
ourselves. Denote by $\calF$   the foliation underlying the regular Poisson structure $\Pi$. Using deformation theory of Dirac structures, we first reconstruct an $L_{\infty}[1]$-algebra $(\Omega^\bullet(\calF;N\calF)[1],\{\frakv_k\})$ governing the deformations of $\calF$ (see Prop.~\ref{lem:G_infty_algebra:foliation}). This recovers a result already obtained 
{in~\cite{Huebsch},~\cite{Ji} and~\cite{vitagliano2014LieRinehart}}
 by different means, {and in \cite{GMS} by similar means}. We then proceed by constructing a strict $L_{\infty}[1]$-morphism 
\begin{equation}\label{eq:mor}
\phi:(\frakX_\good^\bullet(M)[2],\{\frakl_k^{G}\})\longrightarrow(\Omega^\bullet(\calF;N\calF)[1],\{\frakv_k\})
\end{equation}
which corresponds with the right arrow in diagram~\eqref{eq:diag} (see Prop.~\ref{prop:strict_L_infty-algebra_morphism} and Prop.~\ref{prop:strict_morphism:MC_elements}). 
\end{itemize}

The morphism~\eqref{eq:mor} fits in a short exact sequence of $L_{\infty}[1]$-algebras and strict morphisms
\begin{equation*}
 \{0\} \to
  (\frakX^\bullet(\calF)[2],\rmd_\Pi) \to
  (\frakX_\good^\bullet(M)[2],\{\frakl^G_k\})\overset{\phi}{\to}(\Omega^\bullet(\calF;N\calF)[1],\{\frakv_k\})
 \to  \{0\},
\end{equation*}
reflecting the fact that one obtains deformations of regular Poisson structures by deforming both the leafwise symplectic form (see Lemma~\ref{lem:gaugeB}) and the underlying foliation.

\bigskip
\paragraph{\bf {Geometric consequences: infinitesimal deformations}}
At last,  we  draw some geometric consequences. Most results can be stated without making reference to the full algebraic formalism developed above (which however is needed in the proofs).  We address infinitesimal deformations of a regular Poisson structure $\Pi$; these are the $2$-cocycles of the complex $(\frakX_\good^\bullet(M),d_{\Pi})$. Recall that an infinitesimal deformation is called \emph{unobstructed} if it is tangent to a path of deformations, otherwise it is called \emph{obstructed}. We show that the deformation problem of a regular Poisson structure is obstructed in general. We do so applying  the classical Kuranishi criterion to a simple example (see Ex.~\ref{ex:obstructed_example}).

{Refining this}, we also study obstructedness in relation to the diagram~\eqref{eq:diag}. {We ask whether a regular Poisson structure $\Pi$ can have an infinitesimal deformation $Z$ which is obstructed, but such that the corresponding infinitesimal deformation as a Poisson structure is unobstructed,  or the corresponding infinitesimal deformation of the foliation is  unobstructed. In both cases the answer is positive. In other words, even though $Z$ can not be prolonged to a   path of regular Poisson structures, it may be prolonged to a path of  Poisson structures, and it may be prolonged to a path {of regular bivector} fields spanning an integrable distribution. In detail:}

\begin{itemize}
    \item [i)]
First, we investigate whether a regular Poisson structure $\Pi$ can admit obstructed infinitesimal deformations that are unobstructed when deforming $\Pi$ just as a Poisson structure {(without conditions on the rank)}. We display some examples, showing that the answer is positive (see Ex.~\ref{ex:StephaneobstructedReg} and Ex.~\ref{ex:genexStephane}). Our examples involve regular Poisson structures whose foliated symplectic form $\omega$ has a lot of leafwise variation; this is due to the fact that the primary obstructions to extending an infinitesimal deformation of $\Pi$, either to a path of regular Poisson structures or just to a path of Poisson structures, are equivalent when $\omega$ admits a closed extension (see Cor.~\ref{cor:Kuranisheq}).

\item[ii)]
Second, we relate obstructedness of infinitesimal deformations of $\Pi$ with features of the underlying foliation $\calF$. Using the morphism~\eqref{eq:mor}, an infinitesimal deformation of $\Pi$ gives rise to an infinitesimal deformation of $\calF$; when the latter is obstructed, so is the former. There are however obstructed infinitesimal deformations of $\Pi$ that do not arise in this way (see Ex.~\ref{projobs}).
In fact, it is even possible that $\Pi$ has obstructed infinitesimal deformations, while the deformation problem of the foliation $\calF$ is completely unobstructed (see Ex.~\ref{ex:obs-unobs}). We also give some conditions on the underlying foliation $\calF$ which do imply unobstructedness of infinitesimal deformations of $\Pi$, {see Prop.~\ref{dim2} and Prop.~\ref{prop:infrig}.}
\end{itemize}

We finish this introduction by mentioning that in a companion paper~\cite{DefSFalg}, we will present some results which  are not needed for this paper but which complete the theory presented here. Most notably, we will show that the $L_{\infty}[1]$-algebra $(\frakX^\bullet_\good(M)[2],\{\frakl^G_k\})$ constructed here does not depend on auxiliary data (up to $L_{\infty}[1]$-isomorphism), and we will prove that its gauge equivalence relation corresponds with the geometric notion of equivalence given by isotopies. We also discuss in more detail its relation with the Koszul dgLa.

\bigskip
\paragraph{\bf Relation to the literature.}
{Our approach to the deformations of regular Poisson structures, by means of Dirac geometry, is analogous to the one adopted in~\cite{SZDirac,SZPre}  in the setting of presymplectic forms, i.e. closed 2-forms with kernel of constant rank. We remark that in that setting, it was not investigated whether the obstructedness of infinitesimal deformation{s} is exclusively due to the obstructedness of the underlying foliation, whereas here we address this question. 
}

{
There is a relation between our approach to deformations and
previous results about horizontally non-degenerate Poisson and Dirac structures, which can be used to recover a semi-local version of our main result. We explain this  in 
\S\ref{rem:Rui}.
}

\bigskip
\paragraph{\bf Structure of the paper.}
In Section~\ref{sec:one}, we set up the stage by introducing symplectic foliations and regular Poisson structures. We describe the infinitesimal deformations of regular Poisson structures.

In Section~\ref{sec:regular_bivector_fields}, we deal with the constant rank condition. We parametrize regular bivector fields close to a given regular Poisson structure by means of the Dirac exponential map.

In Section~\ref{sec:L_infty-algebra_regular_Poisson}, we also take the Poisson integrability condition into account. This section contains our main result. We construct the $L_{\infty}[1]$-algebra $(\frakX_\good^\bullet(M)[2],\{\frakl_k^{G}\})$ introduced above, and we show that its Maurer-Cartan elements correspond with small deformations of $\Pi$ under the Dirac exponential map.

In Section~\ref{sec:four}, we explore the relation between deformations of the regular Poisson structure $\Pi$ and those of its underlying foliation $\calF$. We reconstruct the $L_{\infty}[1]$-algebra governing the deformation problem of $\calF$, and we relate it with $(\frakX_\good^\bullet(M)[2],\{\frakl_k^{G}\})$ by means of a strict $L_{\infty}[1]$-morphism.

In Section~\ref{sec:obstructedness}, we show that infinitesimal deformations of a regular Poisson structure can be obstructed. 

{In Section~\ref{sec:relating} we} also relate obstructedness in the realm of regular Poisson structures with obstructedness in the realm of Poisson structures, and we show how properties of the underlying foliation allow us to draw conclusions about (un)obstructedness of infinitesimal deformations of the regular Poisson structure.

Appendices \ref{app:Courant_algebroids} and 
\ref{app:Deformations_Dirac_Structures} are a recollection of Dirac structures and their deformations, while Appendix
\ref{app:proofstrict} is devoted to a proof.

 \bigskip
  \paragraph{\bf Acknowledgements.} 

S.G.~and M.Z.~acknowledge partial support by
the FWO and FNRS under EOS project G0H4518N. {S.G. would also like to thank the Max Planck Institute for Mathematics in Bonn for its hospitality and financial support.}
M.Z.~acknowledges partial support by the long term structural funding -- Methusalem grant of the Flemish Government,  and the FWO research project G083118N (Belgium).
A.T.~has been supported by the FWO postdoctoral fellowship 1204019N during the preparation of this paper.
Further he is member of the National Group for Algebraic and Geometric Structures, and their Applications (GNSAGA – INdAM) and is currently partially supported by CMUP, which is financed by national funds through FCT – Fundação para a Ciência e a Tecnologia, I.P., under the project with reference UIDB/00144/2020.

\section{\textsf{Symplectic Foliations}}\label{sec:one}
{In this section, we set up the stage for the deformation problem of a symplectic foliation. We introduce the objects under consideration, and we describe their infinitesimal deformations.}

\subsection{Basic definitions}

\begin{definition}
	\label{def:symplectic_foliation}
	A \emph{symplectic foliation} on a manifold $M$ is a (regular) foliation $\calF$ of $M$ endowed with a \emph{leafwise symplectic structure} $\omega$, i.e.~a non-degenerate $2$-cocycle $\omega$ in the leafwise de Rham complex $(\Omega^\bullet(\calF),\rmd_\calF)$.
\end{definition}

In the following, we will denote by $\SymplFol(M)$ the space of all symplectic foliations on $M$ and by $\SymplFol^{2k}(M)\subset\SymplFol(M)$ the subspace of those having rank $2k$, so that
\begin{equation*}
\SymplFol(M)=\sqcup_k\SymplFol^{2k}(M).
\end{equation*}

Actually, a symplectic foliation is the same thing as a regular Poisson structure; let us briefly recall this well-known fact.

\begin{definition}
	A Poisson structure $\Pi\in\frakX^2(M)$ on a manifold $M$ is \emph{regular} if the associated VB-morphism $\Pi^\sharp:T^\ast M\to TM,\ \eta\mapsto\iota_\eta\Pi$, has constant rank.
	In this case, we will refer to the rank of $\Pi^\sharp$ as the \emph{rank} of $\Pi$.
\end{definition}
Above, ``VB'' stands for vector bundle.
In the following, we will denote by $\RegPoiss(M)$ the space of all regular Poisson structures on $M$ and by $\RegPoiss^{2k}(M)\subset\RegPoiss(M)$ the subspace of those having rank $2k$, so that
\begin{equation*}
\RegPoiss(M)=\sqcup_k\RegPoiss^{2k}(M).
\end{equation*}

Let $\calF$ be a foliation on a manifold $M$, {and denote $\mathfrak{X}^{\bullet}(\calF):=\Gamma(\wedge^{\bullet}T\calF)$}.
Notice that if $\Pi$ is a regular Poisson structure on $M$ with $T\calF=\im\Pi^\sharp$, then $T^\circ\calF=\ker\Pi^\sharp$ and so $\Pi\in\frakX^2(\calF)$.
In other words, a regular Poisson structure with characteristic foliation $\calF$ is the same thing as a non-degenerate Maurer-Cartan (MC) element of $(\frakX^\bullet(\calF)[1],[-,-]_{\sf SN})$, where $[-,-]_{\sf SN}$ denotes the Schouten-Nijenhuis bracket.
Further, the relation $\Pi=-\omega^{-1}$ establishes a one-to-one correspondence between non-degenerate MC elements $\Pi$ of $(\frakX^\bullet(\calF)[1],[-,-]_{\sf SN})$ and non-degenerate $2$-cocycles $\omega$ in $(\Omega^\bullet(\calF),\rmd_\calF)$.
This proves the following.

\begin{proposition}
	\label{prop:regular_symplectic_foliation}
	For any manifold $M$, there is a canonical rank-preserving bijection
	\begin{equation*}
	\SymplFol^\bullet(M)\longrightarrow\RegPoiss^\bullet(M),\ \ (\calF,\omega)\longmapsto {-}\omega^{-1}.
	\end{equation*}
\end{proposition}

\subsection{The formal tangent space}
We now identify the formal tangent space to $\RegPoiss^{2k}(M)$ at a point $\Pi$. To this end, for any foliation $\calF$ on $M$, we introduce a suitable space of multivector fields on $M$, which we call \emph{good multivector fields} w.r.t. $\calF$.

\begin{definition}
	\label{def:good_multi-vector_fields}
	For any foliation $\calF$ on a manifold $M$, the graded space
	 of \emph{good multivector fields} is defined by
	\begin{equation*}
	\label{eq:good_k-vectorfields}
	\frakX^\bullet_\good(M)=\{W\in\frakX^\bullet(M):\iota_\alpha\iota_\beta W=0\ \text{for all}\ \alpha,\beta\in\Gamma(T^\circ\calF)\}.
	\end{equation*}
	Here $T^\circ\calF$ denotes the annihilator of $T\calF$.
\end{definition}

\begin{remark}\label{rem:good}
	Notice that $\frakX^0_\good(M)\!=\!C^\infty(M)$ and $\frakX^1_\good(M)\!=\!\frakX(M)$.
	Further, $\frakX^\bullet_\calF(M)\!=\!\frakX^\bullet(M)$, if $\calF$ has corank $1$.	
Using a choice of distribution $G$ complementary to $T\calF$, we have $\frakX^\bullet_\good(M)=\Gamma(\wedge^\bullet T\calF)\oplus \Gamma(\wedge^{\bullet-1} T\calF\otimes G)$.

\end{remark}

\begin{lemma}
	\label{lem:good_multivector_fields}
	For any foliation $\calF$ on a manifold $M$, the graded space $\frakX^\bullet_\good(M)$ is the graded $\frakX^\bullet(\calF)$-submodule of $\frakX^\bullet(M)$ generated by $C^\infty(M)\oplus\frakX(M)$.
	
	Further, if $T\calF=\im\Pi^\sharp$  for some $\Pi\in\RegPoiss(M)$, then $\Pi$ is a good bivector field, i.e.~$\Pi\in\frakX_\good^2(M)$, and $\frakX_\good^\bullet(M)$ is a subcomplex of $(\frakX^\bullet(M),\rmd_\Pi)$, i.e.
	\begin{equation}
	\label{eq:lem:good_multivector_fields}
	\rmd_\Pi\frakX_\good^\bullet(M)\subset\frakX_\good^{\bullet+1}(M).
	\end{equation}
\end{lemma}

{Above, $d_{\Pi}:=[\Pi,-]_{SN}$ denotes the Lichnerowicz differential on the space of multivector fields $\frakX^\bullet(M)$.}
\begin{proof}
Since the rest is clear, we only prove Equation~\eqref{eq:lem:good_multivector_fields}.
The hypothesis $T\calF=\im\Pi^\sharp$ implies that $\Pi\in\frakX^2(\calF)$ and $\rmd_\Pi\frakX^\bullet(\calF)\subset\frakX^{\bullet+1}(\calF)$.
Since $\rmd_\Pi$ is a graded algebra derivation and $C^\infty(M)\oplus\frakX(M)$ generates $\frakX_\calF^\bullet(M)$ over $\frakX^\bullet(\calF)$,
it remains to prove that $d_\Pi\mathfrak{X}(M)\subset\mathfrak{X}^{2}_\good(M)$.
For any $X\in\frakX(M)$ and $\alpha,\beta\in \Gamma(T^0\calF)$, one gets
\[
(d_{\Pi}X)(\alpha,\beta)=-(\calL_{X}\Pi)(\alpha,\beta)=\Pi(\calL_{X}\alpha,\beta)+\Pi(\alpha,\calL_{X}\beta)-\calL_{X}(\Pi(\alpha,\beta))=0,
\]
{using that $T^0\calF=\ker\Pi^\sharp$}.
It follows that $d_{\Pi}X\in\mathfrak{X}^{2}_\good(M)$, and this completes the proof.
\end{proof}

\begin{lemma}
	\label{lem:formal_tangent_space}
	Let $\Pi$ be a rank $2k$ regular Poisson structure on $M$, with $\im\Pi^\sharp=T\calF$.
	For any smooth path $(\Pi_t)_{t\in(-\epsilon,\epsilon)}$ in $\RegPoiss^{2k}(M)$ with $\Pi_0=\Pi$, one gets that $\dot\Pi_0:=\left.\frac{\rmd}{\rmd t}\right|_{t=0}\Pi_t$ is a $2$-cocycle in $(\frakX^\bullet_\good(M),\rmd_\Pi)$, i.e.:
		\begin{equation*}
		\dot\Pi_0\in\frakX^2_\good(M)\quad\text{and}\quad\rmd_\Pi\dot\Pi_0=0.
		\end{equation*}
\end{lemma}

\begin{proof}
	Differentiating the identity $[\Pi_t,\Pi_t]_{\sf SN}=0$, at time $t=0$, one gets that $\rmd_\Pi\dot\Pi_0=0$ as follows:
	\begin{equation*}
	0=\left.\tfrac{\rmd}{\rmd t}\right|_{t=0}[\Pi_t,\Pi_t]_{\sf SN}=[\dot\Pi_0,\Pi]_{\sf SN}+[\Pi,\dot\Pi_0]_{\sf SN}=2[\Pi,\dot\Pi_0]_{\sf SN}=2\rmd_\Pi\dot\Pi_0.
	\end{equation*}
	{Next, we claim that if $\alpha\in\Gamma(T^\circ\calF)$, then one can construct a smooth path $\alpha_t$ in $\Omega^1(M)$ with $\alpha_0=\alpha$ and $\iota_{\alpha_t}\Pi_t=0$. To see this, note that the product manifold $M\times(-\epsilon,\epsilon)$ has a smooth distribution $D$ defined by
	\[
	D_{(p,t)}=\im(\Pi_t^{\sharp})_{p}\oplus\langle\partial_t\rangle.
	\]
	Denote by $D^{0}$ its annihilator, which is a vector bundle over $M\times (-\epsilon,\epsilon)$.
	Since $\alpha\in\Gamma(D^{0}|_{M\times\{0\}})$ and the submanifold $M\times\{0\}\subset M\times(-\epsilon,\epsilon)$ is properly embedded, one can extend $\alpha$ to a global section $\widetilde{\alpha}\in\Gamma(D^{0})$. Setting $\alpha_t:=\widetilde{\alpha}|_{M\times\{t\}}$ proves the claim.}
	
Now we can show that $\dot\Pi_0\in\frakX^2_\good(M)$. Take $\alpha,\beta\in\Gamma(T^\circ\calF)$ and let $\alpha_t,\beta_t$ be paths as constructed above.
	Differentiating the identity $\Pi_t(\alpha_t,\beta_t)=0$ one gets:
	\begin{equation*}
	0=\tfrac{\rmd}{\rmd t}|_{t=0}(\Pi_t(\alpha_t,\beta_t))
	=\dot\Pi_0(\alpha,\beta)+\Pi(\dot\alpha_0,\beta)+ \Pi(\alpha,\dot\beta_0)=\dot\Pi_0(\alpha,\beta).\qedhere
	\end{equation*}
\end{proof}

	Lemma~\ref{lem:formal_tangent_space} leads to the following description for the formal tangent space to $\RegPoiss^{2k}(M)$ at a point $\Pi$: 
	\begin{equation*}
	T_\Pi\big(\RegPoiss^{2k}(M)\big)=Z^2\big(\frakX^\bullet_\good(M),\rmd_\Pi\big):=\ker\big\{\rmd_\Pi:\frakX^2_\good(M)\to\frakX^3_\good(M)\big\}.
	\end{equation*}

\section{\textsf{Parametrizing Nearby Regular Bivector Fields}}
\label{sec:regular_bivector_fields}

In this section, using tools from Dirac geometry, we discuss the constant rank condition on bivector fields close to a given regular Poisson structure.
Since we postpone the discussion of integrability to \S~\ref{sec:L_infty-algebra_regular_Poisson}, everything boils down to Dirac linear algebra. {We {freely use} notions and notations from Dirac geometry, for which we refer the reader to Appendix~\ref{app:Courant_algebroids}.}
We first introduce in \S\ref{subsec:nearbybivf} a local parametrization of bivector fields close to the given regular Poisson structure, and then in \S\ref{subsec:nearbyregbivf} we take into account the constant rank condition.

The first idea for parametrizing small deformations of a rank $2k$ regular Poisson structure $\Pi$, with characteristic foliation $\calF$, would be to use those small bivector fields $Z$ such that $\Pi+Z$ is still a rank $2k$ regular Poisson structure.
From the perspective of deformation theory of Dirac structures (see~Lemma~\ref{lem:2.6second}), this approach amounts to deforming the Dirac structure $\gr\Pi$ (the graph of $\Pi$) using $TM$ as a complementary Dirac structure.
However, the inconvience of this approach is that it does not translate the constant rank condition into a linear condition.

In this section and in \S~\ref{sec:L_infty-algebra_regular_Poisson} we use a better way of parametrizing small deformations of the regular Poisson structure $\Pi$, still based on the deformation theory of Dirac structures, which works as follows.
First, one fixes a distribution $G$ on $M$ complementary to $T\calF$, and then one uses the complementary almost Dirac structure
$G\oplus T^\ast\calF$ to parametrize the small deformations of the Dirac structure $\gr\Pi$.
As we will find out below (see Lemma~\ref{lem:nearby_regular_2-vectors}), the constant rank condition turns out to be a linear condition once we adopt this parametrization.

{In the whole body of this paper, unless stated explicitly, we assume to following set-up:
\bigskip
\noindent
\begin{mdframed}
\begin{itemize}
\item 
$(\calF,\omega)$ is a rank $2k$ symplectic foliation on a manifold $M$. 
\item $\Pi$ is the corresponding rank $2k$ regular Poisson structure. 
\item $G$ is a  distribution  on $M$ such that $TM=G\oplus T\calF$. 
\item $\gamma$ is the unique $2$-form on $M$ defined by
\begin{equation*}
\gamma^\flat|_{T\calF}=\omega^\flat\qquad\text{and}\qquad\gamma^\flat|_G=0.
\end{equation*}
\end{itemize}
\end{mdframed}
}
\bigskip{}
\noindent  We denote by $\pr_{T\calF}:TM\to T\calF$ and $pr_G:TM\to G$ the projections induced by the above splitting of $TM$. Dually we have  $T^{*}M=T^{*}\calF\oplus G^{*}$, identifying $T^{*}\calF\cong G^{0}$ and $G^{*}\cong T\calF^{0}$. 

\subsection{Parametrizing Nearby Bivector Fields}\label{subsec:nearbybivf}

{Assume the set-up outlined in the box above.} In this subsection, we will parametrize bivector fields close to $\Pi$ by means of the so-called \emph{Dirac exponential map} associated with $G$ and $\Pi$. We now introduce this map, which is constructed in two steps. 

\subsubsection{{\underline{\smash{The gauge transformation by $\gamma$}}}}
The two-form $\gamma\in\Omega^{2}(M)$ determines an orthogonal transformation of the generalized tangent bundle $(\bbT M,\ldab-,-\rdab)$ given by
\begin{equation*}
\calR_\gamma:\bbT M\to\bbT M,\ X+\alpha\mapsto X+\alpha+\iota_X\gamma.
\end{equation*}
For any bivector field $Z\in\frakX^2(M)$, if the almost Dirac structure $\calR_\gamma\gr(Z)$ is still transverse to $TM$, then it is the graph of a (unique) bivector field, denoted by $Z^\gamma$ and called~\cite[\S3]{SW} the \emph{gauge transform of $Z$ by $\gamma$}. 
In order to describe when $Z^\gamma$ exists, let us introduce the open neighborhood $\calI_\gamma$ of the zero section in $\wedge^2TM$ given by
\begin{equation}
\label{eq:def:I_gamma}
\calI_\gamma:=\sqcup_{x\in M}\{Z_x\in\wedge^2T_xM\mid \id_{T^\ast_xM}+\gamma^\flat_x\circ Z_x^\sharp:T_x^\ast M\to T_x^\ast M\ \text{is invertible}\}.
\end{equation}
Then one can easily check that the open neighborhood $\Gamma(\calI_\gamma)$ of $0$ in $\frakX^2(M)$ w.r.t. the $C^0$-topology consists exactly of those bivector fields whose gauge transform by $\gamma$ is a well-defined bivector field.

\begin{lemma}
	\label{lem:existence_gauge_trasform}
	For any $Z\in\frakX^2(M)$, the following conditions are equivalent:
	\begin{enumerate}
		\item there is a (unique) bivector field $Z^\gamma$ such that $\calR_\gamma\gr(Z)=\gr(Z^\gamma)$,
		\item the almost Dirac structure $\gr(Z)$ is transverse to $\gr(-\gamma)$,
		\item the bivector field takes values in $\calI_\gamma$, i.e.~$Z\in\Gamma(\calI_\gamma)$.
	\end{enumerate} 
\end{lemma}
\begin{proof}
{
Clearly $(1)$ and $(2)$ are equivalent, since $\calR_\gamma\gr(Z)\pitchfork TM$ iff $\gr(Z)\pitchfork\calR_{-\gamma}TM$, and  $\calR_{-\gamma}TM=\gr(-\gamma)$. We now check that $(1)$ and $(3)$ are equivalent. As $\calR_\gamma\gr(Z)=\{Z^{\sharp}(\alpha)+\alpha+\gamma^{\flat}(Z^{\sharp}(\alpha)):\ \alpha\in T^{*}M\}$, we have
\[
\calR_\gamma\gr(Z)\cap TM=\big\{Z^{\sharp}(\alpha):\ \alpha\in\ker\big(\id+\gamma^{\flat}\circ Z^{\sharp}\big)\big\}.
\]
It follows that $\calR_\gamma\gr(Z)\pitchfork TM$ iff $\ker(\id+\gamma^{\flat}\circ Z^{\sharp})\subset\ker(Z^{\sharp})$, which in turn is equivalent with $\ker(\id+\gamma^{\flat}\circ Z^{\sharp})$ being trivial. This shows that $(1)$ and $(3)$ are equivalent.
}
\end{proof}

The orthogonal transformation $\calR_\gamma:\bbT M\to\bbT M$ induces a bijection
\begin{equation}
\label{eq:exp_G:first_component}
\begin{tikzcd}
\{\calL\ \text{almost Dirac}\mid\calL\pitchfork TM\ \&\ \calL\pitchfork\gr(-\gamma)\}\arrow[r, "\calR_\gamma"]&\{\calL\ \text{almost Dirac}\mid\calL\pitchfork\gr(\gamma)\ \&\ \calL\pitchfork TM\}.
\end{tikzcd}
\end{equation}
By Lemma \ref{lem:existence_gauge_trasform}, identifying bivector fields with their graphs, this bijection comes from a fiber preserving diffeomorphism
\begin{equation*}
F_\gamma:\label{eq:fiber-preserving_diffeo:1}
\calI_\gamma\overset{\sim}{\longrightarrow}\calI_{-\gamma},\ Z\longmapsto Z^\gamma.
\end{equation*}

\noindent
Clearly, we have that $F_\gamma(0)=0$ and $F_\gamma^{-1}=F_{-\gamma}$. The bivector field $Z^\gamma$ is explicitly characterized by 
\begin{equation}
\label{eq:Z^gamma}
(Z^\gamma)^\sharp=Z^\sharp\circ(\id_{T^\ast M}+{\gamma^\flat}{{}\circ{}}{Z^\sharp})^{-1}.
\end{equation}

\subsubsection{{\underline{\smash{The Dirac exponential map}}}}

The bivector field $\Pi\in\mathfrak{X}(M)$ also determines an orthogonal transformation of  $(\bbT M,\ldab-,-\rdab)$, given by
\begin{equation*}
\calR_\Pi:\bbT M\longrightarrow\bbT M,\ X+\alpha\longmapsto X+\iota_\alpha\Pi+\alpha.
\end{equation*}
Since $\im\gamma^\flat=T^\ast\calF$ and $\id_{TM}+\Pi^\sharp\circ\gamma^\flat=\pr_{G}$, it is clear that
\begin{equation}
\label{eq:preliminary:exp_G:second_component}
\calR_\Pi TM=TM\qquad\text{and}\qquad\calR_\Pi\gr(\gamma)=G\oplus T^\ast\calF.
\end{equation}
Consequently, the orthogonal transformation $\calR_\Pi:\bbT M\to\bbT M$ induces a bijection
\begin{equation}
\label{eq:exp_G:second_component}
\begin{tikzcd}
\{\calL\ \text{almost Dirac}\mid\calL\pitchfork\gr(\gamma)\ \&\ \calL\pitchfork TM\}\arrow[r, "\calR_\Pi"]&\{\calL\ \text{almost Dirac}\mid\calL\pitchfork G\oplus T^\ast\calF\ \&\ \calL\pitchfork TM\}
\end{tikzcd}
\end{equation}
which, identifying bivector fields with their graphs, corresponds with a fiber preserving diffeomorphism
\begin{equation*}
\label{eq:fiber-preserving_diffeo:2}
\calI_{-\gamma}\longrightarrow \Pi+\calI_{-\gamma},\ Z_x\longmapsto\Pi_x+Z_x.
\end{equation*}
Composing the maps~\eqref{eq:exp_G:first_component} and~\eqref{eq:exp_G:second_component}, we get that the orthogonal transformation $\calR_\Pi\calR_\gamma$ induces a bijection
\begin{equation*}
\label{eq:exp_G:almost_Dirac}
\begin{tikzcd}
\{\calL\ \text{almost Dirac}\mid\calL\pitchfork TM\ \&\ \calL\pitchfork\gr(-\gamma)\}\arrow[r, "\calR_\Pi\calR_\gamma"]&\{\calL\ \text{almost Dirac}\mid\calL\pitchfork G\oplus T^\ast\calF\ \&\ \calL\pitchfork TM\}.
\end{tikzcd}
\end{equation*}

{After identifying  bivector fields with  their graphs, we obtain 
the desired parametrization of bivector fields close to $\Pi$. }

\begin{definition}
The \emph{Dirac exponential map} $\exp_G$ associated with $G$ and $\Pi$ is defined by
\begin{equation}\label{dir-exp}
	\exp_G\colon\calI_\gamma\longrightarrow\Pi+\calI_{-\gamma},\ \ Z\longmapsto\exp_G(Z):=\Pi+Z^\gamma,
\end{equation}
Its action on sections $Z\in\Gamma(\calI_\gamma)$ is given by 
\begin{equation}\label{eq:grexp}
 \gr(\exp_G(Z))=\calR_\Pi\calR_\gamma\gr(Z), 
\end{equation}
 or equivalently,
\[(\exp_G(Z))^\sharp=\Pi^\sharp+Z^\sharp\circ(\id_{T^\ast M}+\gamma^\flat\circ Z^\sharp)^{-1}. 
\]
\end{definition}

\subsection{Parametrizing Nearby Regular Bivector Fields}
\label{subsec:nearbyregbivf}

The following key lemma  proves that parametrizing bivector fields close to $\Pi$ by means of the Dirac exponential map $\exp_G$ turns the constant rank condition into a linear condition. Namely, the parameter is required to belong to the subspace of good bivector fields (see Definition~\ref{def:good_multi-vector_fields}).
\begin{lemma}
	\label{lem:nearby_regular_2-vectors}
	For any bivector field $Z\in\frakX^2(M)$, the following two conditions are equivalent:
	\begin{enumerate}
		\item $Z\in\Gamma(\calI_\gamma)$ is a good bivector field,
		\item $\exp_G(Z)\in\Gamma(\Pi+\calI_{-\gamma})$ is a rank $2k$ regular bivector field.
	\end{enumerate}
\end{lemma}

\begin{proof}
	First note that $\calR_{-\gamma}\calR_{-\Pi}T^\ast M=T\calF\oplus G^\ast$, as follows
	from $\im\Pi^\sharp=T\calF$ and $\id_{T^\ast M}+\gamma^\flat\circ\Pi^\sharp=\pr_{G^\ast}$. 
	So for each $Z\in\Gamma(\calI_\gamma)$, one can easily compute
	\begin{equation*}
	\ker(\exp_G(Z))^{\sharp}=T^\ast M\cap\gr(\exp_G(Z))=\calR_\Pi\calR_\gamma((T\calF\oplus G^\ast)\cap\gr(Z)).
	\end{equation*}
	This tells us that $\exp_G(Z)$ has constant rank $2k$ iff the fibers of the subbundle $(T\calF\oplus G^\ast)\cap\gr(Z)\subset\bbT M$ have constant rank equal to the one of $\ker\Pi^{\sharp}=T^\circ\calF\simeq G^\ast$.
Since
	\begin{equation*}
	(T\calF\oplus G^\ast)\cap \gr(Z)=\{\iota_\alpha Z+\alpha\mid\alpha\in G^\ast\ \&\  \iota_\alpha Z\in T\calF \},
	\end{equation*}
	this happens iff $Z^\sharp(G^\ast)\subset T\calF$, i.e.~iff $Z$ is a good bivector field.
\end{proof}

{When restricted to good bivector fields, the Dirac exponential map parametrizes regular bivector fields that are close to $\Pi$, in the sense that they are still transverse to the complement $G$. This is shown in the next lemma.}

\begin{lemma}
	\label{lem:nearby_regular_2-vectors:bis}
	For any rank $2k$ regular bivector field $W\in\frakX^2(M)$, the following conditions are equivalent:
	\begin{enumerate}[label=(\arabic*)]
		\item
		\label{enumitem:lem:nearby_regular_2-vectors:bis:1}
		the bivector field $W$ takes values in $\Pi+\calI_{-\gamma}$, i.e.~$W\in\Gamma(\Pi+\calI_{-\gamma})$,
		\item
		\label{enumitem:lem:nearby_regular_2-vectors:bis:2}
		the almost Dirac structure $\gr(W)\subset\bbT M$ is transverse to $G\oplus T^\ast\calF$,
		\item
		\label{enumitem:lem:nearby_regular_2-vectors:bis:3}
		the distribution $\im W^\sharp\subset TM$ is transverse to $G$.
	\end{enumerate}
\end{lemma}

\begin{proof}
	First notice that  conditions~\ref{enumitem:lem:nearby_regular_2-vectors:bis:1} and~\ref{enumitem:lem:nearby_regular_2-vectors:bis:2} are equivalent for an arbitrary bivector field $W\in\frakX^2(M)$.
	Indeed, using equation~\eqref{eq:preliminary:exp_G:second_component} {and Lemma~\ref{lem:existence_gauge_trasform}}, one gets that
	\begin{equation*}
	\gr W\pitchfork G\oplus T^\ast\calF\Longleftrightarrow \gr W\pitchfork\calR_\Pi\gr(\gamma)\Longleftrightarrow\gr(W-\Pi)\pitchfork\gr(\gamma)\Longleftrightarrow W-\Pi\in\Gamma(\calI_{-\gamma}).
	\end{equation*}
	
	Let us continue by proving that  conditions~\ref{enumitem:lem:nearby_regular_2-vectors:bis:2} and~\ref{enumitem:lem:nearby_regular_2-vectors:bis:3} are equivalent for any rank $2k$ regular bivector field $W\in\frakX^2(M)$.
	Since $TM=G\oplus T\calF$ and $W$ is regular with $\rank W=2k=\rank T\calF$, one gets immediately that
	\begin{equation*}
	G\pitchfork\im W^\sharp\Longleftrightarrow T^\ast\calF\pitchfork\ker W^\sharp.
	\end{equation*}
	Assume now that condition~\ref{enumitem:lem:nearby_regular_2-vectors:bis:2} holds.
	Then one can compute
	\begin{equation*}
	T^\ast\calF\cap\ker W^\sharp\subset\{W^\sharp\alpha+\alpha\mid\alpha\in T^\ast\calF\ \&\ W^\sharp\alpha\in G\}=(G\oplus T^\ast\calF)\cap\gr W=0,
	\end{equation*}
	so that $T^\ast\calF\pitchfork\ker W^\sharp$, and this shows that condition~\ref{enumitem:lem:nearby_regular_2-vectors:bis:3} holds.
	Conversely, assume that condition~\ref{enumitem:lem:nearby_regular_2-vectors:bis:3} holds.
	Then one can easily see that
	\begin{equation*}
	(G\oplus T^\ast\calF)\cap\gr W=\{W^\sharp\alpha+\alpha\mid\alpha\in T^\ast\calF\ \&\ W^\sharp\alpha\in G\}\subset\ker W^\sharp\cap T^\ast\calF=0,
	\end{equation*}
	and therefore condition
	~\ref{enumitem:lem:nearby_regular_2-vectors:bis:2} holds.
	This concludes the proof.
\end{proof}

\begin{remark}
	\label{rem:nearby_regular_bivectors}
	For any $W\in\frakX^2(M)$, it is straightforward that $W\in\frakX^2_\good(M)$ iff $-W\in\frakX^2_\good(M)$. Furthermore, 
	\begin{equation*}
	W\in\Gamma(\calI_{-\gamma})\Longleftrightarrow -W\in\Gamma(\calI_\gamma)\hspace{0.5cm}\text{and}\hspace{0.5cm} (-W)^\gamma=-(W^{-\gamma}).
	\end{equation*}
	{ Notice that $(-W)^{\gamma}\neq -(W^{\gamma})$,
reflecting the fact that  $W\mapsto W^{\gamma}$ is not a linear operation.}
\end{remark}

{Combining Lemmas~\ref{lem:nearby_regular_2-vectors} and~\ref{lem:nearby_regular_2-vectors:bis}, we obtain the desired parametrization of regular bivector fields close to $\Pi$.}

\begin{theorem}
	\label{theor:nearby_regular_bivectors}
	The fiber preserving diffeomorphism given by the
	  Dirac exponential map as in eq.~\eqref{dir-exp}  
	\begin{equation*}
	\exp_G\colon\calI_{\gamma}\overset{\sim}{\longrightarrow}\Pi+\calI_{-\gamma},\ \ Z\longmapsto\exp_G(Z):=\Pi+Z^\gamma,
	\end{equation*}
induces the following bijection at the level of sections:
	\begin{equation*}
	\Gamma(\calI_{\gamma})\cap\frakX_\good^2(M)\overset{\sim}{\longrightarrow}\{W\in\frakX^2_{\text{reg-2k}}(M)\mid \im W^\sharp\pitchfork G\},\ \ Z\longmapsto\exp_G(Z).
	\end{equation*}
\end{theorem}

{Here $\frakX^2_{\text{reg-2k}}(M)$ denotes the space of bivector fields on $M$ that are regular of rank $2k$.}

\begin{remark}\label{rem:expmap}
{We can view $\exp_G$ (at the level of sections) as a submanifold chart for $\frakX^2_{\text{reg-2k}}(M)$ nearby $\Pi$, as depicted in Figure \ref{fig:chart} below. The name ``exponential map'' is justified by the fact that the derivative $d_0\exp_G$ is the identity on ${\frakX_\good^2(M)}$. Indeed, given a smooth curve $Z_t$ in $\frakX_\good^2(M)$ with $Z_0=0$, we have 
$$\left.\frac{\rmd}{\rmd t}\right|_{t=0}\exp_G(Z_t)=\left.\frac{\rmd}{\rmd t}\right|_{t=0}Z_t,$$
as follows immediately applying Eq.~\eqref{eq:Z^gamma} to each $Z_t$ and taking the time derivative.
}
\end{remark}

  \begin{center}
\begin{figure}[h!]
\begin{tikzpicture}[scale=1.2]
\draw(-5,2.6)-- (5,2.6) -- (5,-1.8) -- (-5,-1.8)--(-5,2.6);
\node [above] at (-4.5,2.6) {$\mathfrak{X}^2(M)$};

\draw[dashed]  (-2,0)  -- (2,2) ;
\node  at (-0,1) {$\bullet$}; \node [above] at (-0,1) {$\Pi$};
\node [left] at (-2,0) {$\mathfrak{X}^2_{\good}(M)$};

  \draw[blue] (-0,1) to [out=30,in=165] (2.5,0.9);
  \draw[blue] (-0,1) to [out=30+180,in=30] (-1.5,-1);  
  \node [left] at (-1.5,-1) {\textcolor{blue}{$\mathfrak{X}^2_{\text{reg-2k}}(M)$}};
  
 \draw[thick,magenta,-stealth]  (-0,1)  -- (1.5,1.75); 
 \node [left] at (1.5,2) {\textcolor{magenta}{$Z$}};
 \node  at (1.7,1.11) {\textcolor{magenta}{$\bullet$}}; \node [right] at (1.8,1.2){\textcolor{magenta}{$\Pi+ Z^{\gamma}$}};
\end{tikzpicture}
\caption{A submanifold chart for $\frakX^2_{\text{reg-2k}}(M)$. 
}
\label{fig:chart}
\end{figure}
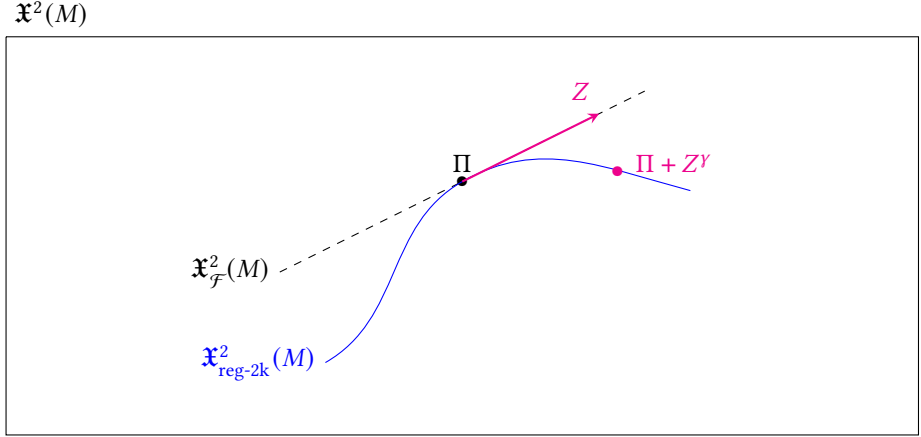
\end{center}

 \begin{remark}[An alternative characterization of the Dirac exponential map]
Recall that any regular bivector field can be equivalently described in terms of a pair, consisting of a  distribution $D$ (the image of its sharp-map) together  
with a section of $\wedge^2D^*$ which is non-degenerate at every point. For instance, in the notation above, the regular Poisson structure $\Pi$ corresponds to the distibution $T\mathcal{F}$ and the leaf-wise symplectic form $\omega$.

We provide an alterative characterization of $\exp_G(Z)$, for $Z$ a   good bivector field lying in $\Gamma(\calI_{\gamma})$.
Decompose $Z$ as \[Z=Z_1+Z_2\in\Gamma(\wedge^{2}T\mathcal{F})\oplus\Gamma(T\mathcal{F}\otimes G).\]  
Then the regular bivector field $\exp_G(Z)$ corresponds to the following pair:
\begin{itemize}
    \item $D:=\text{Gr}(-Z_2^{\sharp}\circ \omega^{\flat}\colon T\mathcal{F}\to G)$,
    \item $pr^*(\omega-(\wedge^2\omega^{\flat})Z_1)$, 
\end{itemize}
where $pr\colon D \to T\mathcal{F}$ is the restriction of the projection $TM\to T\mathcal{F}$ with kernel $G$.
This can be checked as in the proof of the later Proposition \ref{prop:strict_morphism:MC_elements}.
Hence the deformation $\exp_G(Z)$ of $\Pi$ can be described as follows: the component $Z_2$ deforms the  distribution $T\mathcal{F}$, while $Z_1$ deforms of the foliated symplectic form $\omega$.
\end{remark}

\section{\textsf{The Local Deformation Space of a Symplectic Foliation}}
\label{sec:L_infty-algebra_regular_Poisson}

This section  describes the local deformation space of a rank $2k$ symplectic foliation $(\calF,\omega)$
via an associated $L_\infty[1]$-algebra (see Proposition~\ref{prop:good_L_infty_algebra}).
This $L_\infty[1]$-algebra structure lives on the shifted space $\frakX^\bullet_\good(M)[2]$ of good multivector fields, and it controls the deformations of the rank $2k$ regular Poisson structure $\Pi$ corresponding with $(\calF,\omega)$.
Indeed, as stated in the main result of this section (see Theorem~\ref{theor:deformation_theory regular_Poisson}),  \emph{small} Maurer--Cartan (MC) elements of the $L_\infty[1]$-algebra $\frakX_\good^\bullet(M)[2]$ parametrize a neighborhood of $\Pi$ in $\RegPoiss^{2k}(M)$.
The construction of this $L_\infty[1]$-algebra structure  relies on standard results from deformation theory of Dirac structures (see Appendix~\ref{app:Deformations_Dirac_Structures}). {We use freely the notation from Appendix~\ref{app:Courant_algebroids}.}

In \S\ref{sec:Koszul_algebra} we display an $L_\infty[1]$-algebra whose MC elements parametrize arbitrary Dirac structures nearby $\Pi$ (see Proposition \ref{prop:Koszul_algebra:MC_elements}). In \S\ref{sec:deformations_symplectic_foliations} we restrict this  $L_\infty[1]$-algebra  to obtain Theorem~\ref{theor:deformation_theory regular_Poisson}.

\subsection{Deformation Theory of the Dirac Structure $\mathbf{\text{Gr}}(\boldsymbol{\Pi})$}
\label{sec:Koszul_algebra}

\subsubsection{{\underline{\smash{The general Poisson case}}}}
Let $\Pi$ be a Poisson structure on a manifold $M$.
	The generalized tangent bundle $\bbT M=TM\oplus T^\ast M$ admits the following direct sum decomposition
	\begin{equation*}
	\bbT M=TM\oplus \gr\Pi,
	\end{equation*}
	where both $TM$ and $\gr\Pi=\{\iota_\alpha\Pi+\alpha\mid\alpha\in T^\ast M\}\subset\bbT M$ are Dirac structures.
	Applying Lemma~\ref{lem:2.6first} and Remark~\ref{rem:decalage} 2) to the current situation, where $E=\bbT M$ is the standard Courant algebroid, $A=\gr \Pi$ and $B=TM$, one gets that the associated $L_\infty$-algebra $\big(\Omega^\bullet(\gr\Pi)[1],\{m_k^{TM}\}\big)$ reduces to a dgLa.
	The latter admits an equivalent description that is {better-known}, as we now show.

\begin{lemma}\label{ex:deformation_theory_Poisson_structures}

Denote by $\calR_\Pi$ the orthogonal transformation of $(\bbT M,\ldab-,-\rdab)$ determined by $\Pi$. 

a) This map gives an isomorphism of dgLa's  
 \begin{equation}\label{eq:isoGrPi}
\wedge^\bullet(\calR_\Pi|_{T^\ast M})^\ast:(\Omega^\bullet(\gr\Pi)[1],m_1^{TM},m_2^{TM})\longrightarrow(\frakX^\bullet(M)[1],\rmd_\Pi, [-,-]_{SN}).
\end{equation}

b) The associated map between MC elements recovers the fact that under the relation
\[
P=\Pi+Z,
\]
Poisson structures $P\in\mathfrak{X}^{2}(M)$ correspond with MC elements $Z$ of the dgLa $(\frakX^\bullet(M)[1],\rmd_\Pi,[-,-]_{SN})$.
\end{lemma}
\begin{proof}
{a)}	There exist unique Dorfman bracket $\ldsb-,-\rdsb_\Pi$ {and anchor $\rho$} on $\bbT M$,
	such that the following is a Courant algebroid isomorphism
	\begin{equation*}
	\calR_\Pi:(\bbT M,\ldab-,-\rdab,\ldsb-,-\rdsb_\Pi,{\rho})\longrightarrow(\bbT M,\ldab-,-\rdab,\ldsb-,-\rdsb,\pr_{TM}):X+\alpha\mapsto X+\iota_{\alpha}\Pi+\alpha,
	\end{equation*}
	where on the RHS, $\bbT M$ is equipped with its standard Courant algebroid structure.
	Furthermore, the orthogonal transformation $\calR_\Pi$ induces the identity map $\calR_\Pi|_{TM}=\id_{TM}:TM\to TM$ and the Lie algebroid isomorphism:
	\begin{equation*}
	\calR_\Pi|_{T^\ast M}:(T^\ast M,[-,-]_\Pi,\Pi^\sharp)\longrightarrow(\gr\Pi,\ldsb-,-\rdsb,\pr_{TM}),
	\end{equation*}
	where $T^\ast M$ carries the Lie algebroid structure associated with $\Pi$.
	Hence, {Lemma~\ref{lem:2.6first} and Remark~\ref{rem:decalage}} 2)
imply that  {\eqref{eq:isoGrPi}} is a dgLa isomorphism.

{b)}	
  For any $\xi\in\Omega^{2}(\gr\Pi)$, the graph
  $\gr(\xi)$ of the map $\gr\Pi\to (\gr\Pi)^*\cong TM$ induced by $\xi$
  is transverse to $TM$. Hence  $\gr(\xi)$ corresponds with a bivector field $P\in\mathfrak{X}^{2}(M)$. One checks that it is determined by 
	\[
	\wedge^2(\calR_\Pi|_{T^\ast M})^\ast(\xi)=P-\Pi.
	\]
Using Lemma~\ref{lem:2.6second}, we can summarize the situation in the following diagram:
\begin{equation*}
\begin{tikzcd}
MC(\Omega^\bullet(\gr\Pi)[1],m_1^{TM},m_2^{TM})\arrow[d,"\xi\mapsto\gr(\xi)",swap]\arrow[rr, "\wedge^2(\calR_\Pi|_{T^\ast M})^\ast"] &  &MC(\frakX^\bullet(M)[1],\rmd_\Pi,[-,-]_{SN})\\
\{P\in\mathfrak{X}^{2}(M):\ [P,P]_{SN}=0\}& &
\end{tikzcd}.
\end{equation*}

From this diagram, we indeed read off the relation
    \begin{equation}\label{eq:para}
    P=\Pi+Z
	\end{equation}
between bivectors $P\in\frakX^2(M)$ s.t.~$[P,P]_{SN}=0$, and MC elements $Z$ of the dgLa $(\frakX^\bullet(M)[1],\rmd_\Pi,[-,-]_{SN})$.
\end{proof}

\subsubsection{{\underline{\smash{The regular Poisson case: deforming using the complement $G\oplus T^\ast\calF$}}}}
For any Poisson structure $\Pi$ on $M$, Lemma~\ref{ex:deformation_theory_Poisson_structures} shows that, choosing $TM$ as Dirac structure complementary to $\gr\Pi$, 
Lemma~\ref{lem:2.6second} reduces to the fact that the dgLa $(\frakX^\bullet(M)[1],\rmd_\Pi,[-,-]_{\sf SN})$ controls the deformation problem of $\Pi$.
Assuming that $\Pi$ is regular of rank $2k$, {the parametrization~\eqref{eq:para} is not well-suited to single out regular deformations of $\Pi$, since the space of $Z\in\mathfrak{X}^{2}(M)$ for which $\Pi+Z$ has rank $2k$ is not a vector space.}

{We have seen in Section~\ref{sec:regular_bivector_fields} that there is a better choice of almost Dirac structure complementary to $\gr\Pi$, which is compatible with the constant rank condition. Using this complement, we now construct via Lemma~\ref{lem:2.6first} a different $L_{\infty}[1]$-algebra that, by Lemma~\ref{lem:2.6second},  still controls the deformation problem of the Dirac structure $\gr\Pi$. We will see in \S\ref{sec:deformations_symplectic_foliations}  that this $L_{\infty}[1]$-algebra is relevant to the deformation problem of the regular Poisson structure $\Pi$. Below, we assume the setup summarized in the box at the beginning of \S\ref{sec:regular_bivector_fields}.}

\begin{lemma}
	\label{lem:Koszul_algebra:splitting}	
	The almost Dirac structure $G\oplus T^\ast\calF\simeq G\oplus G^\circ$ is complementary to $\gr\Pi$.
\end{lemma}
\begin{proof}
{
For any $\alpha\in T^{*}M$, we have $\Pi^{\sharp}(\alpha)\in T\calF$, so requiring that $\Pi^{\sharp}(\alpha)$ lies in $G$ implies that $\Pi^{\sharp}(\alpha)=0$. This means that $\alpha\in T\calF^{0}$, hence also requiring that $\alpha$ lies in $G^{0}$ yields $\alpha=0$.
}
\end{proof}

At this point, it is natural to apply Lemmas~\ref{lem:2.6first} and~\ref{lem:2.6second} to the setting where $E=\bbT M$, $A=\gr\Pi$ and $B=G\oplus T^\ast\calF$, yielding an $L_\infty[1]$-algebra $(\Omega^\bullet(\gr\Pi)[2],\{\mu^{G\oplus T^\ast\calF}_k\})$ which controls the deformation problem of the Dirac structure $\gr\Pi$. Instead of doing so directly, to simplify the computations, we will follow an indirect approach.
{
We first simplify the situation, by transporting the Courant algebroid structure along a map which transforms the splitting $\bbT M=\gr\Pi\oplus(G\oplus T^\ast\calF)$ into $\bbT M=T^{*}M\oplus TM$. Then we apply Lemmas~\ref{lem:2.6first} and~\ref{lem:2.6second}, which yields 
 an $L_\infty[1]$-algebra structure $\{\frakl_k^{G}\}$ on $\frakX^\bullet(M)[2]$. This one is more directly related to the geometry of the symplectic foliation $(\calF,\omega)$ and the chosen splitting $TM=G\oplus T\calF$, 
it is strictly isomorphic to $(\Omega^\bullet(\gr\Pi)[2],\{\mu^{B}_k\})$ (see Proposition~\ref{prop:Koszul_algebra:isomorphism}) and its MC elements encode the Dirac structures close to $\gr\Pi$ w.r.t. $G\oplus T^\ast\calF$ (see Proposition~\ref{prop:Koszul_algebra:MC_elements}).} 
\begin{lemma}
	\label{lem:Koszul_algebra:Courant_iso}
	There exists a unique Courant algebroid structure $(\ldsb-,-\rdsb_G,\rho_G)$ on $(\bbT M,\ldab-,-\rdab)$,  
	such that the orthogonal transformation
	\begin{align}
	\label{eq:lem:Koszul_algebra:Courant_iso}
	&\calR_\Pi\calR_\gamma:(\bbT M,\ldab-,-\rdab,\ldsb-,-\rdsb_G,\rho_G)\overset{\sim}{\longrightarrow}(\bbT M,\ldab-,-\rdab,\ldsb-,-\rdsb,\pr_{TM}),\nonumber\\
	&\hspace{4.05cm}X+\alpha\longmapsto (\pr_GX+\iota_\alpha\Pi)+(\alpha+\iota_X\gamma),
	\end{align}
	is a Courant algebroid isomorphism.
	In particular, the latter induces: 
\begin{itemize}
\item 
	the Lie algebroid isomorphism
	\begin{equation}
	\label{eq:lem:Koszul_algebra:Lie_algbd_iso}
	(\calR_\Pi\calR_\gamma)|_{T^\ast M}:T^\ast M\overset{\sim}{\longrightarrow}\gr\Pi,\ \alpha\longmapsto\iota_\alpha\Pi+\alpha,
	\end{equation}
	where $T^\ast M$ carries the Lie algebroid structure associated with $\Pi$.
\item  the almost Lie algebroid isomorphism
	\begin{equation}
	\label{eq:lem:Koszul_algebra:almost_Lie_algbd_iso}
	(\calR_\Pi\calR_\gamma)|_{TM}:TM\overset{\sim}{\longrightarrow}G\oplus T^\ast\calF,\ X\longmapsto\pr_GX+\iota_X\gamma,
	\end{equation}
	where $TM$ carries the almost Lie algebroid structure $([-,-]_\gamma,\rho_\gamma)$ defined by
	\begin{equation}
	\label{eq:lem:Koszul_algebra:almost_Lie_algbd_structure}
	[X,Y]_\gamma=[\pr_GX,\pr_GY]-\Pi^\sharp(\calL_{\pr_GX}\iota_Y\gamma-\calL_{\pr_GY}\iota_X\gamma),\qquad\rho_\gamma(X)=\pr_GX.
	\end{equation}
	\end{itemize}
\end{lemma}

The proof is a straightforward computation, and so we omit it.
  
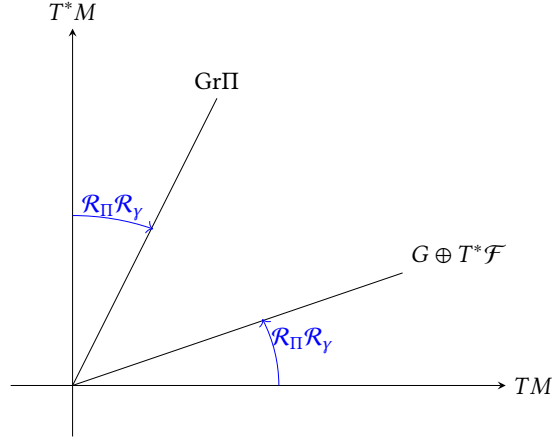
\begin{center}
\begin{figure}[h]
\begin{tikzpicture}[scale=0.95]
		\begin{axis}[axis lines=middle,ticks=none, xlabel style={right},
		ylabel style={above}, xmin=-0.3, xmax=2.1,ymin=-0.3,ymax=2.1,
		xlabel=$TM$,ylabel=$T^\ast M$, samples =2, disabledatascaling, no marks ]
		\addplot[domain=0:2*0.35] {2.414*x} node[above,text=black]{$\text{Gr}\Pi$};
		\addplot[domain=0:2*0.8] {0.414*x} node[above right,text=black] {$G\oplus T^*\calF$};
		\draw [blue, ->] (axis cs:1,0) arc [radius=1,start angle=0,end angle=23] node[below right]{$\calR_\Pi \calR_\gamma$};
		\draw [blue, ->] (axis cs:0,1) arc [radius=1,start angle=90,end angle=67] node[above left]{$\calR_\Pi \calR_\gamma$};
		\end{axis}
		\end{tikzpicture}
\end{figure}
\end{center}

\begin{remark}\label{rem:Courant_tensor_twisted_TM}
  Denote by $\Upsilon_{TM}^G\in\Omega^3(M)$  the Courant tensor (see Remark~\ref{rem:Courant_tensor}) of the almost Dirac structure $TM$ in $(\bbT M,\ldab-,-\rdab,\ldsb-,-\rdsb_G,\rho_G)$, then Lemma~\ref{lem:Koszul_algebra:Courant_iso} implies that $\Upsilon_{TM}^G$ is the pullback along isomorphism~\eqref{eq:lem:Koszul_algebra:almost_Lie_algbd_iso} of the Courant tensor of $G\oplus T^\ast\calF$.
So it is easy to see that, for all $X,Y,Z\in\frakX(M)$,
	\begin{equation*}
	\Upsilon^G_{TM}(X,Y,Z)=\gamma(X,[\pr_G Y,\pr_G Z])+\gamma(Y,[\pr_G Z,\pr_G X])+\gamma(Z,[\pr_G X,\pr_G Y]).
	\end{equation*}	
Additionally, notice that $\Upsilon^G_{TM}\in\Gamma(T^\ast\calF\otimes\wedge^2G^\ast)\subset\Omega^3(M)$ and it vanishes iff $G\subset TM$ is involutive.
\end{remark}

The following remark relates the almost Lie algebroid structure $(TM,[-,-]_\gamma,\rho_\gamma)$  with the existing literature.
\begin{remark}
\begin{enumerate}
\item Notice that $\gamma$ is a constant rank 2-form on $M$, with kernel $G$, whose restriction to the leaves of $T\calF $ is closed. Hence the condition $d\gamma=0$ is equivalent to
\begin{itemize}
\item $(d\gamma)|_{T\calF \wedge G\wedge G}=0$, meaning that $G$ is involutive
\item $(d\gamma)|_{T\calF \wedge T\calF \wedge G}=0$, meaning that $\calL_{X}\gamma=0$ for all $X\in \Gamma(G)$.
\end{itemize}
\item  In full generality, the bracket in \eqref{eq:lem:Koszul_algebra:almost_Lie_algbd_structure} can be rewritten as
\begin{align}\label{eq:brgamma1}
  [X,Y]_\gamma&=\big([\pr_GX,\pr_GY]+\pr_{T\calF} [\pr_GX,\pr_{T\calF} Y]-
\pr_{T\calF} [\pr_GY,\pr_{T\calF} X]\big)\nonumber\\
&\hspace{1cm}
- \Pi^\sharp\left(\iota_{\pr_{T\calF} Y}\iota_{\pr_GX}d\gamma-\iota_{\pr_{T\calF} X}\iota_{\pr_GY}d\gamma\right),
\end{align}
by applying Cartan identities to the term $\calL_{\pr_GX}\iota_Y\gamma$ in Eq.~\eqref{eq:lem:Koszul_algebra:almost_Lie_algbd_structure} and using that $\Pi^\sharp\gamma^{\flat}=-\pr_{T\calF}$.
 
Recall~\cite{MagriYvettePN} (see also~\cite{LieBialgebroidPN})  that an endomorphism $N\colon TM\to TM$ is called \emph{Nijenhuis} if the  tensor $T_N$ vanishes, where $T_N(X,Y)=
[N X, N Y ] - N ([N X, Y ] + [X, N Y ]) + N^2 [X, Y ].$
In that case,  $N$ gives rise to a Lie algebroid structure on $TM$, with anchor $N$ itself and Lie bracket
\[[X,Y]_N :=[NX,Y]+[X,NY]-N[X,Y].\] 
Assume now that $G$ is involutive. Then one can check that $\pr_G\colon TM\to TM$ is a Nijenhuis endomorphism, and that the corresponding Lie algebroid bracket $[-,-]_{\pr_G}$ is the first term in round brackets in Eq.~\eqref{eq:brgamma1}.
Using this, and the first bullet point in item (1) together with the fact that $\Pi\in \wedge^2T\calF $, we see that 
\begin{equation*}
  [X,Y]_\gamma=[X,Y]_{\pr_G}
-\Pi^\sharp\left(\iota_{Y}\iota_{X}d\gamma\right).
\end{equation*}
In other words, when $G$ is involutive, the Lie bracket $[-,-]_\gamma$
equals the Lie bracket associated to the Nijenhuis endomorphism $\pr_G$ 
plus an additional term involving $\Pi$.
 
 \item When the distribution $G$ is involutive, we have $\Upsilon^G_{TM}\equiv 0$ (see Remark~\ref{rem:Courant_tensor_twisted_TM}).
 Hence $(TM,[-,-]_{\gamma})$ is a Lie algebroid, and together with the Lie algebroid $(T^*M,
 [-,-]_{\Pi})$ associated to the Poisson structure it forms a Lie bialgebroid.
The stronger condition $d\gamma=0$ implies that the Nijenhuis endomorphism $\pr_G$ and the Poisson structure $\Pi$ are compatible, i.e. form a Poisson-Nijenhuis structure. In that case the    Lie bialgebroid of the  Poisson-Nijenhuis structure~\cite[\S 3.2]{LieBialgebroidPN} is exactly the Lie bialgebroid mentioned just above.
\end{enumerate}
\end{remark}

We apply Lemma~\ref{lem:2.6first}  to the case where $E=(\bbT M,\ldab-,-\rdab,\ldsb-,-\rdsb_G,\rho_G)$ is the Courant algebroid mentioned in Lemma~\ref{lem:Koszul_algebra:Courant_iso}, with Dirac structure $A=T^\ast M$ and almost Dirac structure $B=TM$.
So the graded space $\frakX^\bullet(M)[2]$ inherits an $L_\infty[1]$-algebra structure $\{\frakl_k^G\}$, described in the following.

\begin{proposition}
	\label{prop:Koszul_algebra}
	The $L_\infty[1]$-algebra associated with the Dirac structure $T^\ast M$ via the complementary almost Dirac structure $TM$ in the Courant algebroid $(\bbT M,\ldab-,-\rdab,\ldsb-,-\rdsb_G,\rho_G)$ 
	consists of the graded vector space $\frakX^\bullet(M)[2]$ with the $L_\infty[1]$-algebra structure $\{\frakl_k^G\}$ whose only non-trivial multibrackets, $\frakl_1^G,\frakl_2^G,\frakl_3^G$, are given by:
	\begin{itemize}
		\item the unary bracket $\frakl_1^G$ is the Poisson differential $\rmd_\Pi$, i.e. for all $P\in\frakX^\bullet(M)$:
		\begin{equation}
		\label{eq:prop:Koszul_algebra:unary_bracket}
		\frakl_1^G(P)=[\Pi,P]_{\sf SN},
		\end{equation}
		\item the binary bracket $\frakl_2^G$ acts as follows on homogeneous $P,Q\in\frakX^\bullet(M)$: 
		\begin{equation}
		\label{eq:prop:Koszul_algebra:binary_bracket}
		\frakl_2^G(P,Q)=(-)^{|P|} [P,Q]_\gamma,
		\end{equation}
	where $[-,-]_\gamma$ denotes the
	extension to an	almost Gerstenhaber bracket 
	(cf.~Remark~\ref{rem:almost_Dirac_structure})
	of the bracket in eq. \eqref{eq:lem:Koszul_algebra:almost_Lie_algbd_structure}.
		\item the ternary bracket $\frakl_3^G$ acts as follows on  homogeneous $P,Q,R\in\frakX^\bullet(M)$:
		\begin{equation}
		\label{eq:prop:Koszul_algebra:ternary_bracket}
		\frakl_3^G(P,Q,R)=(-1)^{|Q|}\big(P^\sharp\wedge Q^\sharp \wedge R^\sharp\big)\Upsilon^G_{TM},
		\end{equation}
		{where $\Upsilon^G_{TM}$  is defined in Remark~\ref{rem:Courant_tensor_twisted_TM}.}
	\end{itemize}
\end{proposition}

\begin{proof}
	It is a straightforward consequence of Lemmas~\ref{lem:2.6first} and~\ref{lem:Koszul_algebra:Courant_iso}.
\end{proof}

\begin{remark}\label{rem:Koszul_algebra:dgLA}
We list some remarks concerning the $L_\infty[1]$-algebra $(\frakX^\bullet(M)[2],\{\frakl^G_k\})$ introduced above.
\begin{enumerate}
\item In view of Remarks~\ref{rem:decalage} (2)  and~\ref{rem:Courant_tensor_twisted_TM}, the $L_\infty[1]$-algebra $(\frakX^\bullet(M)[2],\{\frakl^G_k\})$ reduces to a dgL[1]a 
if and only if $G\subset TM$ is involutive.
\item As already pointed out in Remark~\ref{rem:decalage} (3),  
the $L_\infty[1]$-algebra $(\frakX^\bullet(M)[2],\{\frakl_k^G\})$  
is actually a \emph{$G_\infty[1]$-algebra}
{(also called  \emph{$P_\infty[1]$-algebra})},
i.e.~its multibrackets are graded algebra derivations of $\frakX^\bullet(M)$. 
\end{enumerate}
\end{remark}

The $L_\infty[1]$-algebra $(\frakX^\bullet(M)[2],\{\frakl_k^G\})$ is canonically isomorphic to $\big(\Omega^\bullet(\gr\Pi)[2],\big\{\mu_k^{G\oplus T^\ast\calF}\big\}\big)$, as stated below. The proof is a straightforward consequence of Lemma~\ref{lem:Koszul_algebra:Courant_iso}. 

\begin{proposition}
	\label{prop:Koszul_algebra:isomorphism}
	The VB isomorphism~\eqref{eq:lem:Koszul_algebra:Lie_algbd_iso} induces a strict isomorphism of $L_\infty[1]$-algebras
	\begin{equation*}
	\wedge^\bullet(\calR_\Pi\calR_\gamma)|_{T^\ast M}^\ast:\big(\Omega^\bullet(\gr\Pi)[2],\big\{\mu_k^{G\oplus T^\ast\calF}\big\}\big)\overset{\sim}{\longrightarrow}(\frakX^\bullet(M)[2],\{\frakl_k^G\}),
	\end{equation*}
	with inverse  $\wedge^\bullet (\pr_{T^\ast M}|_{\gr\Pi}^\ast):(\frakX^\bullet(M)[2],\{\frakl_k^G\})\overset{\sim}{\longrightarrow}\big(\Omega^\bullet(\gr\Pi)[2],\big\{\mu_k^{G\oplus T^\ast\calF}\big\}\big)$.
\end{proposition}

Turning to MC elements, we now show that the $L_\infty[1]$-algebra $(\frakX^\bullet(M)[2],\{\frakl^G_k\})$ encodes Dirac structures that are close to $\gr\Pi$ w.r.t. $G\oplus T^{*}\mathcal{F}$.

\begin{proposition}
	\label{prop:Koszul_algebra:MC_elements}
	Let $(\calF,\omega)$ be a symplectic foliation on $M$, with corresponding regular Poisson structure $\Pi$.
	For any splitting $TM=T\calF\oplus G$, the relation 
	\begin{equation*}
	L=\calR_\Pi\calR_\gamma\gr(Z)
	\end{equation*}
	establishes a canonical one-to-one correspondence between:
	\begin{itemize}
		\item MC elements $Z$ of the $L_\infty[1]$-algebra $(\frakX^\bullet(M)[2],\{\frakl^G_k\})$,
		\item Dirac structures $L\subset\bbT M$   
		transverse to $G\oplus T^\ast\calF$.
	\end{itemize}
	\end{proposition}

\begin{proof}[Proof of Proposition~\ref{prop:Koszul_algebra:MC_elements}]
We know that $(\frakX^\bullet(M)[2],\{\frakl_k^G\})$ is the $L_\infty[1]$-algebra associated with the Dirac structure $T^\ast M$ and the complementary almost Dirac structure $TM$ in $(\bbT M,\ldab-,-\rdab,\ldsb-,-\rdsb_G,\rho_G)$, by Proposition~\ref{prop:Koszul_algebra}. Lemma~\ref{lem:2.6second} provides the following bijection:
	\begin{equation*}
	\begin{aligned}
	\operatorname{MC}(\frakX^\bullet(M)[2],\{\frakl_k^G\})&\overset{\sim}{\longrightarrow}\{L\subset\bbT M\ \text{Dirac structure wrt}\ (\ldsb-,-\rdsb_G,\rho_G)\mid L\pitchfork TM\},\\
	Z&\longmapsto\gr(Z).
	\end{aligned}
	\end{equation*}
	
	Since $\calR_\Pi\calR_\gamma:(\bbT M,\ldab-,-\rdab,\ldsb-,-\rdsb_G,\rho_G)\overset{\sim}{\longrightarrow}(\bbT M,\ldab-,-\rdab,\ldsb-,-\rdsb,\pr_{TM})$ is a Courant algebroid isomorphism (see Lemma~\ref{lem:Koszul_algebra:Courant_iso}), with $\calR_\Pi\calR_\gamma(TM)=G\oplus T^\ast\calF$, it induces the following bijection:
	\begin{equation*}
	\label{eq:prop:Koszul_algebra:MC_elements:Dirac_structures_1-1}
	\begin{aligned}
	\{L\subset\bbT M\ \text{Dirac structure w.r.t.}\ (\ldsb-,-\rdsb_G,\rho_G)\mid L\pitchfork TM\}&\overset{\sim}{\longrightarrow}\{L\subset\bbT M\ \text{Dirac structure}\mid L\pitchfork G\oplus T^\ast\calF\},\\
	L&\longmapsto\calR_\Pi\calR_\gamma(L).
	\end{aligned}
	\end{equation*}	
Composing the two bijections, the proposition is proven.
\end{proof}

	In summary, {using also Proposition~\ref{prop:Koszul_algebra:isomorphism},} we can draw the following commutative square of bijective maps: 
	\begin{equation*}
	\begin{tikzcd}
	\{L\subset\bbT M\ \text{Dirac structure wrt}\ (\ldsb-,-\rdsb_G,\rho_G)\mid L\pitchfork TM\}\arrow[rr, hook, two heads, "\calR_\Pi\calR_\gamma"]&&\{L\subset\bbT M\ \text{Dirac structure}\mid L\pitchfork G\oplus T^\ast\calF\}\\
	\operatorname{MC}(\frakX^\bullet(M)[2],\{\frakl_k^G\})\arrow[rr, hook, two heads, swap, "\wedge^2(\pr_{T^\ast M}|_{\gr\Pi}^\ast)"]\arrow[u, "\gr(\ast)"]&&\operatorname{MC}\big(\Omega^\bullet(\gr\Pi)[2],\big\{\mu_k^{G\oplus T^\ast\calF}\big\}\big)\arrow[u, swap,  "\gr(\ast)"]
	\end{tikzcd}
	\end{equation*}

\subsubsection{Comparison with the literature: horizontally non-degenerate Dirac structures}\label{rem:Rui}

As pointed   out to us by Rui Loja Fernandes, 
Proposition~\ref{prop:Koszul_algebra:MC_elements} is reminiscent of a correspondence that appeared in~\cite{Vorobjev}\cite[\S5]{CrFeStab} for the Poisson case, and in~\cite{BrahicFernPoissonFibrations}~\cite{WadeCouplingDirac} for the Dirac case. 
 First we recall this correspondence, following the interpretation given by M\u{a}rcu\c{t} \cite[\S4.2]{MarcutThesis}. Then we describe the relation to Proposition~\ref{prop:Koszul_algebra:MC_elements} and Theorem~\ref{theor:deformation_theory regular_Poisson} below, summarizing our conclusions in Corollaries \ref{cor:concl1} and \ref{cor:1''}.
 
Fix a surjective submersion $p\colon E\to S$ with connected fibers, and denote by $V:=\ker(p_*)\subset TE$ the vertical bundle. A Dirac structure on $E$ is called \emph{horizontally non-degenerate} if it is transverse to $V\oplus V^{\circ}$. We denote
{
\begin{itemize}
    \item  by
$\mathfrak{X}^{\bullet}_V(E):=\Gamma(\wedge^{\bullet} V)$ the space of vertical multivector fields, \item  by $\mathfrak{X}_P(E)$ the space of vector fields on $E$ which are $p$-projectable,
\item by $\widetilde{\Omega}_E$ the space of differential forms on $S$ with values in  $\mathfrak{X}^{\bullet}_V(E)+\mathfrak{X}_P(E)$,
 \item   by  $\Omega_E:=\Gamma(\wedge^{\bullet}(V\oplus V^{\circ}))$ the subspace of differential forms on $S$ with values in
 $\mathfrak{X}^{\bullet}_V(E)$. 
\end{itemize}
Notice that $\widetilde{\Omega}_E$ carries a natural bigrading, and that Ehresmann connections on $E$ form an affine subspace of $\widetilde{\Omega}^{(1,1)}_E$.
It turns out that $\widetilde{\Omega}_E[1]$
 has a graded Lie bracket   $[-,-]_{\ltimes}$, making it into a graded Lie algebra, and that 
 $\Omega_E[1]$ is a graded Lie subalgebra.
 }

A Maurer Cartan (MC) element of $({\widetilde{\Omega}_E[1]},[-,-]_{\ltimes})$ is an
    element $\beta\in \widetilde{\Omega}^{(0,2)+(1,1)+(2,0)}_E$ such that
    $[\beta,\beta]_{\ltimes}=0$.
It turns out that  there is a bijection~\cite[Prop. 4.2.5]{MarcutThesis} between
\begin{itemize}
    \item[(1)] horizontally non-degenerate  Dirac structures on $E$
    \item[(2)] MC  elements of $(\widetilde{\Omega}_E[1],[-,-]_{\ltimes})$ whose $(1,1)$-component is an Ehresmann connection.
\end{itemize}
This bijection restricts~\cite[Prop. 4.2.10]{MarcutThesis} to a bijection between
\begin{itemize}
    \item[(1'\,)]   horizontally non-degenerate  Poisson structures 
\item[(2'\,)] MC elements of $(\widetilde{\Omega}_E[1],[-,-]_{\ltimes})$ for which additionally the $(2,0)$-component\footnote{This is a 2-form on $S$ with values in $C^{\infty}(E)$.} is non-degenerate.
\end{itemize}
 Further it restricts to a bijection between  
 \begin{itemize}
     \item[({1'}'\,)] 
 horizontally non-degenerate Poisson structures of constant rank equal to $\dim(S)$
 \item[({2'}'\,)] MC  elements of $(\widetilde{\Omega}_E[1],[-,-]_{\ltimes})$ for which, in addition to the previous two requirements, the $(0,2)$-component\footnote{This is a vertical bivector-field on $E$.} vanishes.
  \end{itemize}

We now relate the bijection (1)-(2) with our Proposition~\ref{prop:Koszul_algebra:MC_elements},  and the last bijection with Theorem~\ref{theor:deformation_theory regular_Poisson} below.
Take
a Dirac structure on $M$, and  fix an embedded leaf $S$. A choice of tubular neighborhood embedding provides a submersion $p\colon E\to S$, defined on a neighborhood of $S$ in $M$, {and we can assume that  the Dirac structure is horizontally-nondegenerate on $E$ (shrinking $E$ if necessary)}.
Denote by $\beta\in \widetilde{\Omega}^{(0,2)+(1,1)+(2,0)}_E$ the Maurer Cartan element of $\widetilde{\Omega}_E$ corresponding to the Dirac structure.
The twist by $\beta$ of the graded Lie algebra $(\widetilde{\Omega}_E[1],[-,-]_{\ltimes})$ is a dgLa with
 the same binary bracket and with
differential $d_{\beta}:=[\beta,-]_{\ltimes}$. It is immediate to check that 
an element $\beta'$ is a MC element of the  graded Lie algebra $(\widetilde{\Omega}_E[1],[-,-]_{\ltimes})$ if{f} $\beta'-\beta$ is a MC element of the above dgLa. Therefore the assignment $\beta'\mapsto \beta'-\beta$  yields a bijection between (2) above and 
\begin{itemize}
 \item[(3)] MC  elements of $(\Omega_E[1],d_{\beta} ,[-,-]_{\ltimes})$,
  \end{itemize}
using the fact that the difference of two Ehresmann connections is a 1-form on $S$ with values in $\Gamma(V)$, {i.e. an element of $\Omega_E^{(1,1)}$}. {Now assume that $\beta$ corresponds to a regular Poisson structure. We will show that the dgLa $(\Omega_E[1],d_{\beta} ,[-,-]_{\ltimes})$ is isomorphic with the dgLa obtained shifting degrees in $(\frakX^\bullet(E)[2],\{\frakl^G_k\})$, where $G$ is the involutive distribution $V$. Note that the latter dgLa is given by $(\frakX^\bullet(E)[1],d_{\Pi},[-,-]_{\gamma})$.}

\begin{lemma}\label{lem:corr}
Suppose that $\beta$ corresponds to a regular Poisson structure. The map
\[
f:TE\rightarrow V\oplus V^{0}:X\mapsto\pr_{V}X+\iota_{X}\gamma
\]
induces an isomorphism of dgLa's
$$
\wedge^{\bullet}f:(\frakX^\bullet(E)[1],d_{\Pi},[-,-]_{\gamma})\overset{\sim}{\rightarrow}(\Omega_E[1],d_{\beta} ,[-,-]_{\ltimes}).
$$
\end{lemma} 
\begin{proof}
{The fact that $\wedge^{\bullet}f$ matches the differentials $d_{\Pi}$ and $d_{\beta}$ is proved in~\cite[Prop. 4.2.11]{MarcutThesis}.
As for the graded Lie brackets, first recall that the bracket $[-,-]_{\ltimes}$ on $\Omega_E=\Gamma(\wedge^{\bullet}(V\oplus V^{0}))$ is defined extending the Lie bracket of the Lie algebroid $V\oplus V^{0}$ (see~\cite[p. 119]{MarcutThesis}). On the other hand, eq. \eqref{eq:lem:Koszul_algebra:almost_Lie_algbd_iso} in Lemma 
\ref{lem:Koszul_algebra:Courant_iso} shows that the bracket $[-,-]_{\gamma}$ is defined by transporting the Lie bracket of $V\oplus V^{0}$ under the map $f$, and then extending to all multivector fields. This implies that $\wedge^{\bullet}f$ intertwines the graded Lie brackets $[-,-]_{\gamma}$ and $[-,-]_{\ltimes}$.}
\end{proof} 
Consequently we conclude:
\begin{corollary}\label{cor:concl1}
{The bijection between (1) and (3) above recovers Proposition~\ref{prop:Koszul_algebra:MC_elements} in a tubular neighborhood of the leaf $S$ for
 the choice of complement $G=V$.} 
\end{corollary}
 
{ This specialises to a similar statement about Theorem~\ref{theor:deformation_theory regular_Poisson} below:
\begin{corollary}\label{cor:1''}
By restricting the bijection (1)-(3) to ({1'}'\,), we recover Theorem~\ref{theor:deformation_theory regular_Poisson}  in a tubular neighborhood of $S$ for the choice of complement $G=V$.
\end{corollary}
}

{Corollary \ref{cor:1''} is a consequence of the following result, which uses the dgLa isomorphism $\wedge^{\bullet}f$ from the previous lemma.
\begin{lemma}
A MC element $\beta'$ of the graded Lie algebra $(\widetilde{\Omega}_E[1],[-,-]_{\ltimes})$ corresponds with the MC element $\wedge^{2}f^{-1}(\beta'-\beta)$ of the dgLa $(\frakX^\bullet(E)[1],d_{\Pi},[-,-]_{\beta})$. Under this correspondence, we have:
\begin{itemize}
    \item the $(0,2)$-component of $\beta'$ vanishes iff $\wedge^{2}f^{-1}(\beta'-\beta)$ belongs to $\frakX_\good^{2}(E)$,
    \item the $(2,0)$-component of $\beta'$ is non-degenerate iff $\wedge^{2}f^{-1}(\beta'-\beta)$ belongs to the neighborhood $\calI_{\gamma}$.
\end{itemize}
\end{lemma}
\begin{proof}
We first remark that the inverse of $f$ is given by
\[
f^{-1}:V\oplus V^{0}\rightarrow TE:X+\alpha\mapsto X-\Pi^{\sharp}\alpha.
\]
In particular, $f^{-1}$ restricts to the identity map on $V$ and takes $V^{0}$ to $T\calF$. We now prove the two statements.
\begin{itemize}
\item The bivector field $\wedge^{2}f^{-1}(\beta'-\beta)$ belonging to $\frakX_\good^{2}(E)$ means that it has no component in $\Gamma(\wedge^{2}V)$. By the previous observation about $f^{-1}$, this is equivalent with $\beta'-\beta$ having no component in $\Gamma(\wedge^{2}V)$, i.e. its $(0,2)$-component being zero. Since $\beta$ corresponds with a regular Poisson structure, its $(0,2)$-component is zero; hence the last statement is equivalent with the $(0,2)$-component of $\beta'$ being zero. \item Recall that a bivector field $Z\in\mathfrak{X}^{2}(E)$ belongs to $\calI_{\gamma}$ iff the map $\id+\gamma^{\flat}\circ Z^{\sharp}:T^{*}E\rightarrow T^{*}E$ is invertible. Writing this map with respect to the decomposition $T^{*}E=T\calF^{0}\oplus V^{0}$ yields a triangular block matrix for which a diagonal block is the identity; as a consequence, it is equivalent to require that the restriction $\id_{V^{0}}+\gamma^{\flat}\circ Z^{\sharp}|_{V^{0}}:V^{0}\rightarrow V^{0}$ is invertible.
We expand
\begin{align}\label{eq:expand}
id_{V^{0}}+\gamma^{\flat}\circ\big(\wedge^{2}f^{-1}(\beta'-\beta)\big)^{\sharp}|_{V^{0}}&=id_{V^{0}}+\gamma^{\flat}\circ\big(\wedge^{2}f^{-1}(\beta'-\beta)^{(2,0)}\big)^{\sharp}|_{V^{0}}\nonumber\\
&=id_{V^{0}}-\gamma^{\flat}\circ\Pi^{\sharp}\circ\big((\beta'-\beta)^{(2,0)}\big)^{\flat}\circ\Pi^{\sharp}|_{V^{0}}.
\end{align}
In the second equality, we use that $f^{-1}$ agrees with $-\Pi^{\sharp}$ on $V^{0}$, which implies that 
\[
\big(\wedge^{2}f^{-1}(\beta'-\beta)^{(2,0)}\big)^{\sharp}=\big(\wedge^{2}\Pi^{\sharp}(\beta'-\beta)^{(2,0)}\big)^{\sharp}=-\Pi^{\sharp}\circ\big((\beta'-\beta)^{(2,0)}\big)^{\flat}\circ\Pi^{\sharp}.
\]
Further using that $\gamma^{\flat}\circ\Pi^{\sharp}=-\pr_{V^{0}}$, the right hand side of~\eqref{eq:expand} becomes
\[
id_{V^{0}}+\big((\beta'-\beta)^{(2,0)}\big)^{\flat}\circ\Pi^{\sharp}|_{V^{0}}.
\]
Since $\beta$ corresponds with the Poisson structure $\Pi$, its $(2,0)$-component $\beta^{(2,0)}$ is given by the inverse of $\Pi^{\sharp}:V^{0}\rightarrow\Pi^{\sharp}(V^{0})$ (see~\cite[p. 130]{MarcutThesis}). Therefore, we conclude that the above expression  equals
$$id_{V^{0}}+\big((\beta')^{(2,0)}\big)^{\flat}\circ\Pi^{\sharp}|_{V^{0}}-\big((\beta)^{(2,0)}\big)^{\flat}\circ\Pi^{\sharp}|_{V^{0}}
=\big((\beta')^{(2,0)}\big)^{\flat}\circ\Pi^{\sharp}|_{V^{0}}.
$$
Since $\Pi^{\sharp}|_{V^{0}}$ is injective, this shows that $\wedge^{2}f^{-1}(\beta'-\beta)$ belongs to $\calI_{\gamma}$ iff $(\beta')^{(2,0)}$ is non-degenerate.
\end{itemize}
\end{proof}
}

\subsection{Deformation Theory of a Symplectic Foliation}
\label{sec:deformations_symplectic_foliations}

{We remain in the setup described by the box at the beginning of \S\ref{sec:regular_bivector_fields}.}
In \S\ref{sec:Koszul_algebra} we constructed the $L_\infty[1]$-algebra $(\frakX^\bullet(M)[2],\{\frakl_k^G\})$, which encodes as its MC elements those Dirac structures that are close to $\gr\Pi$ w.r.t. $G$, in the sense that they are still transverse to $G\oplus T^\ast\calF$.
We now single out an $L_\infty[1]$-subalgebra (see Proposition~\ref{prop:good_L_infty_algebra}) which actually controls the deformation problem of the symplectic foliation $(\calF,\omega)$ (see Theorem~\ref{theor:deformation_theory regular_Poisson}).

\subsubsection{{\underline{\smash{An algebraic tool}}}}
We first need to introduce the notion of \emph{strongly homotopy Lie--Rinehart algebras}, or $LR_\infty[1]$-algebras.
We will essentially adopt the same terminology as in~\cite{vitagliano2014LieRinehart} (cf.~also~\cite{Kjeseth2001Rinehart} for another version of this notion).

\begin{definition}[{\cite[Definition~7]{vitagliano2014LieRinehart}}]
	\label{def:LR_infinity_algebra}
	Let $\calA$ be a graded commutative algebra.
	An \emph{$LR_\infty[1]$-algebra over $\calA$} consists of a graded $\calA$-module $\calQ$ equipped with an $L_\infty[1]$-algebra structure $\{\mu_k\}$ and a family of \emph{anchor maps}
	\begin{equation*}
	\rho_k:\calQ^{\times(k-1)}\times\calA\to\calA,\ (q_1,\ldots,q_{k-1},a)\mapsto\rho_k(q_1,\ldots,q_{k-1}|a)
	\end{equation*}
	of degree $1$ which are $\calA$-linear in the first $k-1$ entries and a derivation in the last one, such that
	\begin{itemize}
		\item for all $k\in\bbN$ and all homogeneous $a\in\calA$, $q_1,\ldots,q_k\in\calQ$, the following Leibniz-like rule holds:
		\begin{equation}
		\label{eq:def:LR_infinity_algebra:Leibniz}
		\mu_k(q_1,\ldots,aq_k)=\rho_k(q_1,\ldots,q_{k-1}|a)q_k+(-)^{|a|(1+|q_1|+\ldots+|q_{k-1}|)}a\mu_k(q_1,\ldots,q_k),
		\end{equation}
		\item for all $k\in\bbN$ and all homogeneous $v_1,\ldots,v_k\in\calQ\oplus\calA$, the following extended Jacobi identity holds:
		\begin{equation}
		\label{eq:def:LR_infinity_algebra:Jacobi}
		{0=\sum_{\substack{i+j=k+1\\i,j\geq 1}}\sum_{\sigma\in S_{(i,k-i)}}\epsilon(\sigma,\mathbf{v})\{\{v_{\sigma(1)},\ldots,v_{\sigma(i)}\}_{i},v_{\sigma(i+1)},\ldots,v_{\sigma(k)}\}_{j}.}
		\end{equation}
		Here $\epsilon(\sigma;\mathbf{v})$ denotes the symmetrical Koszul sign determined by $\sigma$ and $\mathbf{v}=(v_1,\ldots,v_k)$ (cf.~\cite{fiorenza2007cones}) and, for each $i\in\bbN$, the bracket $\{-,\ldots,-\}_i\colon(\calQ\oplus\calA)^{\times i}\to\calQ\oplus\calA$ is the unique graded symmetric multi $\bbR$-linear map which extends $\mu_i$ and $\rho_i$ and which vanishes if more than one entry belongs to $\calA$.
	\end{itemize}
\end{definition}

\begin{remark}
	If $(\calQ,\{\mu_k\},\{\rho_k\})$ is an $LR_\infty[1]$-algebra over $\calA$, then in particular, $\rho_1$ is a degree $1$ graded algebra derivation of $\calA$, and $\mu_1$ is a degree $1$ graded module derivation of $\calQ$ with symbol given by $\rho_1$.
\end{remark}

Notice that the splitting $TM=G\oplus T\calF$ endows the graded algebra $\frakX^\bullet(M)$ with an additional $\bbN\times\bbN$ grading defined as follows:
	\begin{equation}\label{bigrading}
	\frakX^\bullet(M)=\bigoplus_{q,p\in\bbN}\frakX^{(p,q)}(M),\quad\text{where}\quad\frakX^{(p,q)}(M):=\Gamma(\wedge^pT\calF\oplus\wedge^q G).
	\end{equation}
In particular, $\frakX^n(M)=\bigoplus_{p+q=n}\frakX^{(p,q)}(M)$,
and we additionally get
	\begin{equation*}
	\frakX^\bullet(\calF)=\bigoplus_{p\in\bbN}\frakX^{(p,0)}(M)\qquad\text{and}\qquad\frakX^\bullet_\good(M)=\bigoplus_{\genfrac{}{}{0pt}{}{p\in\bbN}{q=0,1}}\frakX^{(p,q)}(M).
	\end{equation*}

\begin{lemma}
	\label{lem:multibrackets:NxN_grading}
	The non-trivial multibrackets of $(\frakX^\bullet(M)[2],\{\frakl_k^G\})$ behave as follows w.r.t. the $\bbN\times\bbN$ grading:
	\begin{enumerate}[label=\arabic*)]
		\item the unary bracket $\frakl_1^G=\rmd_\Pi$ has two components of bi-degree $(1,0)$ and $(2,-1)$, namely
		\begin{equation*}
		\frakl_1^{G}(\frakX^{(p,q)}(M))\subset\frakX^{(p+1,q)}(M)\oplus\frakX^{(p+2,q-1)}(M),
		\end{equation*}
		\item the binary bracket $\frakl_2^G$ has two components of bi-degree $(0,-1)$ and $(1,-2)$, namely
		\begin{equation*}
		\frakl_2^G(\frakX^{(p_1,q_1)}(M)\times\frakX^{(p_2,q_2)}(M))\subset\frakX^{(\sum_i p_i,\sum_iq_i-1)}(M)\oplus\frakX^{(\sum_ip_i+1,\sum_iq_i-2)}(M),
		\end{equation*}
		\item the ternary bracket $\frakl_3^G$ has bi-degree $(-1,-2)$, namely
		\begin{equation*}
		\frakl_3^G(\frakX^{(p_1,q_1)}(M)\times\frakX^{(p_2,q_2)}(M)\times\frakX^{(p_3,q_3)}(M))\subset\frakX^{(\sum_ip_i-1,\sum_iq_i-2)}(M).
		\end{equation*}
	\end{enumerate}
\end{lemma}

\begin{proof}
	Equations~\eqref{eq:prop:Koszul_algebra:unary_bracket}--\eqref{eq:prop:Koszul_algebra:ternary_bracket} (see also Equation~\eqref{eq:lem:Koszul_algebra:almost_Lie_algbd_structure} defining the almost Lie bracket $[-,-]_\gamma$ on $TM$) impose the following constraints on the action of $\frakl^G_1,\ \frakl^G_2,\ \frakl^G_3$ on the generators of the graded algebra $\frakX^\bullet(M)$:
	\begin{equation}
	\label{eq:proof:lem:multibrackets:NxN_grading}
	\begin{gathered}
	\frakl_1^G(C^\infty(M))\subset\Gamma(T\calF),\quad\frakl_1^G(\Gamma(T\calF))\subset\Gamma(\wedge^2T\calF),\quad\frakl_1^G(\Gamma(G))\subset\Gamma(\wedge^2T\calF)\oplus\Gamma(G\otimes T\calF),\\
	\frakl_2^G(C^\infty(M),\Gamma(T\calF))=\frakl_2^G(\Gamma(T\calF),\Gamma(T\calF))=0,\quad\frakl_2^G(\Gamma(T\calF),\Gamma(G))\subset\Gamma(T\calF),\\
	\frakl_3^G(\Gamma(T\calF),\Gamma(T\calF),\Gamma(T\calF))=\frakl_3^G(\Gamma(T\calF),\Gamma(T\calF),\Gamma(G))=\frakl_3^G(\Gamma(G),\Gamma(G),\Gamma(G))=0.
	\end{gathered}
	\end{equation}
	Since the multibrackets are multiderivations of the graded algebra $\frakX^\bullet(M)$ (see Remark~\ref{rem:Koszul_algebra:dgLA} (2)) and the $\bbN\times\bbN$ grading is compatible with the algebra structure, i.e.~$\frakX^{(h,k)}(M)\cdot\frakX^{(p,q)}(M)\subset\frakX^{(h+p,k+q)}(M)$,  the statement follows immediately from the relations~\eqref{eq:proof:lem:multibrackets:NxN_grading}.
\end{proof}

\subsubsection{{\underline{\smash{Parametrizing
deformations of a symplectic foliation}}}}

{We now single out a suitable $L_\infty[1]$-subalgebra of the $L_\infty[1]$-algebra $(\frakX^\bullet(M)[2],\{\frakl_k^G\})$ constructed in Proposition~\ref{prop:Koszul_algebra}.} Recall that the space of good multivector fields $\frakX^\bullet_\good(M)$ was introduced in Definition \ref{def:good_multi-vector_fields}.

\begin{proposition}
	\label{prop:good_L_infty_algebra}
	The multibrackets $\frakl_k^G$ map $\frakX_\good^\bullet(M)[2]\subset\frakX^\bullet(M)[2]$ to itself.
	So $\frakX_\good^\bullet(M)[2]$ inherits an $L_\infty[1]$-algebra structure, still denoted by $\{\frakl_k^G\}$, which we call the \emph{$L_\infty[1]$-algebra associated with $(\calF,\omega)$ and $G$}.
	Actually, $(\frakX_\good^\bullet(M)[2],\{\frakl_k^G\})$ turns out to be an $LR_\infty[1]$-algebra over the graded algebra $\frakX^\bullet(\calF)$.
\end{proposition}

\begin{proof}
As proved in Lemma~\ref{lem:multibrackets:NxN_grading}, the multibrackets $\frakl_k^G$ behave as follows w.r.t. the $\bbN\times\bbN$ bi-grading of $\frakX^\bullet(M)$:
\begin{enumerate}[label=\arabic*)]
	\item the unary bracket $\frakl_1^G=\rmd_\Pi$ has two components of bi-degree $(1,0)$ and $(2,-1)$,
	\item the binary bracket $\frakl_2^G$ has two components of bi-degree $(0,-1)$ and $(1,-2)$,
	\item the ternary bracket $\frakl_3^G$ has bi-degree $(-1,-2)$.
\end{enumerate}
Since $\frakX^\bullet(\calF)=\bigoplus_{p\geq 0}\frakX^{(p,0)}(M)$ and $\frakX_\good^\bullet(M)=\bigoplus_{\genfrac{}{}{0pt}{}{p\geq 0}{q=0,1}} \frakX^{(p,q)}(M)$, one can check that for each $k\in\bbN$:
\begin{itemize}
	\item the graded $\frakX^\bullet(\calF)$-submodule $\frakX_\good^\bullet(M)$ of $\frakX^\bullet(M)$ is closed under $\frakl^G_k$,
	\item the multibracket $\frakl^G_k:\frakX^\bullet(M)^{\times k}\to\frakX^\bullet(M)$ induces a map
	\begin{align*}
	\rho_k:\frakX^\bullet_\good(M)^{\times(k-1)}\times\frakX^\bullet(\calF)\longrightarrow\frakX^\bullet(\calF),\ (u_1,\ldots,u_{k-1},\alpha)\longmapsto\frakl^G_k(u_1,\ldots,u_{k-1},\alpha),
	\end{align*}
	which is graded $\frakX^\bullet(\calF)$-linear in the first $k-1$ entries and a graded algebra derivation in the last entry.
\end{itemize}
Further, the Leibniz rule~\eqref{eq:def:LR_infinity_algebra:Leibniz} and Jacobi identity~\eqref{eq:def:LR_infinity_algebra:Jacobi} for an $LR_\infty[1]$-algebra hold because of the Leibniz rule and the Jacobi identity for the $G_\infty[1]$-algebra structure $\{\frakl_k^G\}$ on $\frakX^\bullet(M)[2]$.
This concludes the proof.
\end{proof}

{Restricting to good multivector fields, we obtain an $L_\infty[1]$-algebra that governs the deformation problem of the symplectic foliation $(\calF,\omega)$. The proof relies on results shown in Section~\ref{sec:regular_bivector_fields}, and we use the same notation as established there.} In particular, recall that the neighborhood $I_{\gamma}$   and the Dirac exponential map $\exp_G$ were introduced in \S\ref{subsec:nearbybivf}.

\begin{theorem}{\bf (Main theorem)}
	\label{theor:deformation_theory regular_Poisson}
	Let $(\calF,\omega)$ be a rank $2k$ symplectic foliation on a manifold $M$, with corresponding rank $2k$ regular Poisson structure $\Pi$.
	Fix a distribution $G$ complementary to $T\calF$ on $M$.
	Then the relation 
	\begin{equation*}
	{-}\widetilde\omega^{-1}=\exp_G(Z)
	\end{equation*}
	establishes a canonical one-to-one correspondence between
	\begin{enumerate}
		\item MC elements $Z$ of the $L_\infty[1]$-algebra $(\frakX_\good^\bullet(M)[2],\{\frakl^G_k\})$ 
		such that $Z\in\calI_{\gamma}$,
		\item rank $2k$ symplectic foliations $(\widetilde\calF,\widetilde\omega)$ on $M$ 
		such that $T\widetilde\calF\pitchfork G$. 
	\end{enumerate}
\end{theorem}

\begin{proof}
Theorem~\ref{theor:nearby_regular_bivectors} establishes a bijection
\begin{equation}
	\label{eq:proof:theor:deformation_theory regular_Poisson:bijection2}
	\begin{aligned}
	\Gamma(\calI_{\gamma})\cap\frakX^2_\good(M)\overset{\sim}{\longrightarrow}\{W\in\frakX^2_{\text{reg-$2k$}}(M)\mid\im W^\sharp\pitchfork G\},\ \ Z\longmapsto\exp_G(Z).
	\end{aligned}
\end{equation}
Using Proposition~\ref{prop:Koszul_algebra:MC_elements} and the fact that $\calR_\Pi\calR_\gamma\gr(Z)=\gr(\exp_G(Z))$ (see eq.~\eqref{eq:grexp}), we also have a bijection
\begin{equation}
	\label{eq:proof:theor:deformation_theory regular_Poisson:bijection1}
	\begin{aligned}
	\operatorname{MC}(\frakX^\bullet(M)[2],\{\frakl_k^G\})\overset{\sim}{\longrightarrow}\{L\subset\bbT M\ \text{Dirac structure}\mid L\pitchfork G\oplus T^\ast\calF\},\ \ Z\longmapsto\gr(\exp_G(Z)).
	\end{aligned}
	\end{equation}
For any $W\in\frakX^2(M)$, we have that $\gr(W)\subset\bbT M$ is Dirac iff $W$ is Poisson, and if $W$ is regular then $\im W^{\sharp}\pitchfork G$ iff $\text{Gr}(W)\pitchfork G\oplus T^\ast\calF$ by Lemma~\ref{lem:nearby_regular_2-vectors:bis}. Hence, the equations~\eqref{eq:proof:theor:deformation_theory regular_Poisson:bijection2} and~\eqref{eq:proof:theor:deformation_theory regular_Poisson:bijection1} lead to the following bijection: 
	\begin{equation*}
	\begin{aligned}
	\Gamma(\calI_{\gamma})\cap\operatorname{MC}(\frakX^\bullet_\good(M)[2],\{\frakl_k^G\})\overset{\sim}{\longrightarrow}\{W\in\RegPoiss^{2k}(M)\mid\im W^\sharp\pitchfork G\},\ \ Z\longmapsto\exp_G(Z).
	\end{aligned}
	\end{equation*}	
Finally, the canonical one-to-one correspondence between rank $2k$ symplectic foliations and rank $2k$ regular Poisson structures on $M$ (as in Proposition~\ref{prop:regular_symplectic_foliation}) completes the proof.	
\end{proof}

\begin{remark}[The corank one case]\label{rem:corank-one}
In case the regular Poisson structure $\Pi$ is of corank one, some simplifications occur. 
On one hand, every complement $G$ to $T\mathcal{F}$ is automatically involutive and $\mathfrak{X}^{\bullet}_\good(M)$ coincides with $\mathfrak{X}^{\bullet}(M)$. Consequently, the deformation problem of the regular Poisson structure $\Pi$ is governed by the dgL[1]a $(\frakX^\bullet(M)[2],\{\frakl^G_k\})$, because of Theorem~\ref{theor:deformation_theory regular_Poisson}.

On the other hand, assuming $M$ is compact, Poisson structures $\mathcal{C}^{0}$-close to $\Pi$ are automatically of corank one, hence it is equivalent to deform $\Pi$ as a Poisson structure (without the rank condition). Consequently, the deformation problem of the regular Poisson structure $\Pi$ is also governed by the usual dgLa  $(\mathfrak{X}^{\bullet}(M)[1],d_{\Pi},[-,-]_{\sf SN})$. The corresponding dgL[1]a is indeed isomorphic with $(\frakX^\bullet(M)[2],\{\frakl^G_k\})$, because of the following. By Lemma~\ref{ex:deformation_theory_Poisson_structures}, it is (strictly) isomorphic with the dgL[1]a $(\Omega^{\bullet}(\gr\Pi)[2],\mu_1^{TM},\mu_2^{TM})$ defined in terms of the splitting $\bbT M=\gr\Pi\oplus TM$. By a result proved in~\cite{GMS}, the latter is $L_{\infty}[1]$-isomorphic with the dgL[1]a $(\Omega^{\bullet}(\gr\Pi)[2],\mu_1^{G\oplus T^{*}\mathcal{F}},\mu_2^{G\oplus T^{*}\mathcal{F}})$, which is defined in terms of the splitting $\bbT M=\gr\Pi\oplus (G\oplus T^{*}\mathcal{F})$. Finally, by Proposition~\ref{prop:Koszul_algebra:isomorphism}, the latter is (strictly) isomorphic with $(\frakX^\bullet(M)[2],\{\frakl^G_k\})$.
\end{remark}

\begin{remark}
{
Given any
$L_\infty[1]$-algebra, there is a natural equivalence relation   on the set of MC elements,  induced by the  elements of degree $-1$.
For the $L_\infty[1]$-algebra
$(\frakX_\good^\bullet(M)[2],\{\frakl^G_k\})$,  
the degree $-1$ elements are the vector fields on $M$. Under the bijection stated in Theorem \ref{theor:deformation_theory regular_Poisson}, the induced equivalence relation on rank $2k$ symplectic foliations transverse to $G$ is given by isotopies. More precisely, and assuming that $M$ is compact: $(\widetilde\calF,\widetilde\omega)$ and $(\calF',\omega')$ are equivalent if{f} there is an isotopy $(\psi_t)_{t\in [0,1]}$ of $M$ mapping the former to the latter, and so that $(\psi_t)_*\widetilde\calF\pitchfork G$ for all $t$. We will show this in \cite{DefSFalg}.}

{
Further, since the 
differential $\frakl_1^G$ in 
the $L_\infty[1]$-algebra
$(\frakX_\good^\bullet(M)[2],\{\frakl^G_k\})$ is the Poisson differential $d_{\Pi}$, 
the degree $-1$ cocycles  are the Poisson vector fields on $(M,\Pi)$, in agreement with the general interpretation of such elements as infinitesimal symmetries. The degree $-1$ coboundaries  are the Hamiltonian vector fields.
}
\end{remark}

\subsection{Deforming by leaf-wise differential forms}
The leaf-wise multivector fields 
$\frakX^\bullet(\calF)$ are contained in the good multivector fields  $\frakX_\good^\bullet(M)$. 
Here we point out that it is easy to describe explicitly the deformation of $(\calF,\omega)$ induced (via Theorem \ref{theor:deformation_theory regular_Poisson})
by a small bivector field in $\frakX^{2}(\calF)$: it is just the gauge transform by the corresponding leaf-wise 2-form.

\begin{lemma}
	\label{lem:leafwise_PL_complex}
	The 
	{cochain complex} 
	$(\frakX^\bullet(\calF)[2],\rmd_\Pi)$ is canonically embedded in  the $L_\infty[1]$-algebra $(\frakX_\good^\bullet(M)[2],\{\frakl_k^G\})$.
\end{lemma}

\begin{proof}
	Since $\frakX^\bullet(\calF)=\bigoplus_{p\in\bbN}\frakX^{(p,0)}(M)$, Lemma~\ref{lem:multibrackets:NxN_grading} implies that $\frakl_1^G=\rmd_\Pi$ preserves $\frakX^\bullet(\calF)$, while $\frakl_2^G$ and $\frakl_3^G$ vanish on $\frakX^\bullet(\calF)$.
	This proves that the natural embedding $\frakX^\bullet(\calF)\lhook\joinrel\rightarrow\frakX_\good^\bullet(M)$ gives a strict morphism of $L_\infty[1]$-algebras $(\frakX^\bullet(\calF)[2],\rmd_\Pi)\longrightarrow(\frakX_\good^\bullet(M)[2],\{\frakl_k^G\})$.
\end{proof}

{ 
Note that the non-degenerate Poisson structure $\Pi$ on the leaves induces an isomorphism of Lie algebroids $\Pi^{\sharp}$ 
between the cotangent Lie algebroid $T^*\calF$ of the   Poisson structure  $\Pi$ and the tangent Lie algebroid $T\calF$.
By pullback, this gives an isomorphism of complexes $\frakX^{\bullet}(\calF)\cong \Omega^{\bullet}(\calF)$.
In particular, an element $Z\in \frakX^2(\calF)$ is $d_{\Pi}$-closed iff the corresponding element $\beta_Z\in\Omega^2(\calF)$ is a closed leafwise form.
\begin{lemma}\label{lem:gaugeB}
The deformation of  $\Pi$  associated to a MC element  $Z\in \frakX^2(\calF)$ lying in $\calI_{\gamma}$, as in Theorem  \ref{theor:deformation_theory regular_Poisson}, is  the gauge transformation of $\Pi$ by  $\tilde{\beta}_{Z}$, where $\tilde{\beta}_{Z}$ is  any extension of  $\beta_{Z}$  to a 2-form on $M$. In particular, it is a regular Poisson structure with same underlying foliation   as $\Pi$.
\end{lemma}
\begin{proof}
We have to show that $\exp_G(Z)=\Pi^{\tilde{\beta}_Z}$, and we do so by checking that their graphs are equal. 
The Dirac structure $\gr(\exp_G(Z))= \calR_\Pi\calR_\gamma\gr(Z)$ is given by
\begin{equation}\label{eq:graphgauge}
  \left\{(\Pi^{\sharp}\xi, \xi{+}\gamma^{\flat}  Z^{\sharp}\xi):\xi \in T^*M\right\},
\end{equation}
  using the fact that $\im(Z^{\sharp})\subset T\calF$ and $\Pi^{\sharp}\circ\gamma^{\flat}=-pr_{ T\calF}$.
Next, we note that  $(\beta_{Z})^{\flat}={-}\omega^{\flat}\circ Z^{\sharp}\circ\omega^{\flat}$. Take the extension of $\beta_{Z}$ to the 2-form $\tilde{\beta}_Z$ on $M$ that annihilates $G$ (its flat-map is ${-}\gamma^{\flat}\circ Z^{\sharp}\circ\gamma^{\flat}$).
Then $\gr(\Pi^{\tilde{\beta}_Z})=\calR_{\tilde{\beta}_Z}(\gr(\Pi))$
 is
equal to~\eqref{eq:graphgauge}, as can be seen using $\gamma^{\flat}\circ\Pi^{\sharp}=-pr_{ T^*\calF}$. Any other extension of $\beta_{Z}$ yields the same gauge transformation of $\Pi$.
\end{proof}
}

\section{\textsf{Relation with Deformations of Foliations}}\label{sec:four}

Forgetting the leafwise symplectic structure and keeping only the underlying foliation, one defines a rank-preserving map from the space of symplectic foliations on $M$ to the space of foliations on $M$, denoted by
\begin{equation}
\label{eq:forgetful_functor}
\textsf{q}:\SymplFol^\bullet(M)\longrightarrow\Fol^\bullet(M),\ \ (\widetilde{\calF},\widetilde{\omega})\longmapsto\widetilde{\calF}.
\end{equation}
This section aims at finding an algebraic interpretation of the latter.
Indeed, in Proposition~\ref{prop:strict_morphism:MC_elements}, we will see that it arises from a strict morphism of $L_\infty[1]$-algebras (cf.~Proposition~\ref{prop:strict_L_infty-algebra_morphism}) going from the $L_\infty[1]$-algebra of the symplectic foliation $(\calF,\omega)$ (cf.~\S\ref{sec:deformations_symplectic_foliations}) to the $L_\infty[1]$-algebra of the foliation $\calF$ (cf.~\S\ref{sec:deformations_foliations}).

\subsection{{Review: Deformations of Foliations}}
\label{sec:deformations_foliations}

We reconstruct here the $L_\infty[1]$-algebra controlling the deformations of a foliation (cf.~\cite{Huebsch,vitagliano2014LieRinehart}),
using the deformation theory of Dirac structures (cf.~Appendix~\ref{app:Deformations_Dirac_Structures}). 

Let $\calF$ be a rank $d$ foliation on a manifold $M$.
Denote the normal bundle to $T\calF$ in $TM$ by $N\calF$ and the quotient VB morphism by $TM\to N\calF:=TM/T\calF,\ X\mapsto\overline{X}$.
Then the Lie algebroid $T\calF\Rightarrow M$ has a natural representation $\nabla$ on $N\calF$, also called the \emph{Bott connection}, which is defined by
\begin{equation*}
\nabla_X\overline{Y}=\overline{[X,Y]},
\end{equation*}
for all $X\in\Gamma(T\calF)$ and $Y\in\frakX(M)$, and one can introduce the de Rham complex $(\Omega^\bullet(\calF;N\calF),\rmd_\nabla)$ of the Lie algebroid $T\calF$ with coefficients in the representation $N\calF$.

Fix a distribution $G$ on $M$ complementary to $T\calF$.
The splitting $TM=G\oplus T\calF$ induces the identifications $T^\circ\calF\simeq G^\ast$, $G^\circ\simeq T^\ast\calF$, $G\simeq N\calF$, and the following splitting of the generalized tangent bundle
\begin{equation*}
\bbT M=(T\calF\oplus T^\circ\calF)\oplus(G\oplus G^\circ),
\end{equation*}
where $T\calF\oplus T^\circ\calF\subset\bbT M$ is a Dirac structure while $G\oplus G^\circ\subset\bbT M$ is an almost Dirac structure.
So, applying Lemma~\ref{lem:2.6first}, one gets the following.
\begin{lemma}
	\label{lem:G_infty_algebra:foliation1}
	There exists a unique $L_\infty[1]$-algebra structure $\{\frakn_k\}$ on $\Gamma(\wedge^\bullet(T^\ast\calF\oplus N\calF))[2]$ that makes $\Gamma(\wedge^\bullet(T^\ast\calF\oplus N\calF))[2]$ into a $G_\infty[1]$-algebra, whose only non-trivial brackets $\frakn_1$, $\frakn_2$, $\frakn_3$ are determined by:
	\begin{itemize}
		\item $\frakn_1$ is the de Rham differential of the Lie algebroid $T\calF\oplus T^\circ\calF\Rightarrow M$, i.e.~it satisfies
		\begin{equation*}
		\frakn_1(f)=\rmd_\calF f,\quad\frakn_1\eta=\rmd_\calF\eta,\quad\frakn_1\overline{X}=\rmd_\nabla\overline{X},
		\end{equation*}
		for all $f\in C^\infty(M)$, $\eta\in\Omega^1(\calF)$ and $X\in\frakX(M)$.
		\item $\frakn_2$ is the almost Gerstenhaber bracket of the almost Lie algebroid $G\oplus G^\circ\to M$ up to a  sign, i.e.~it satisfies  
		\begin{equation*}
		{-}\frakn_2(\eta+\overline{X},f)=X(f),\quad{-}\frakn_2(\eta_1+\overline{X}_1,\eta_2+\overline{X}_2)=\text{pr}_G[X_1,X_2]+\text{pr}_{T^\ast\calF}(\calL_{X_1}\eta_2-\calL_{X_2}\eta_1),
		\end{equation*}
		for all $f\in C^\infty(M)$, $X,X_1,X_2\in\Gamma(G)$ and $\eta,\eta_1,\eta_2\in\Omega^1(\calF)$.
		\item $\frakn_3$ satisfies the following:
		\begin{equation*}
		\frakn_3(\eta_1+\overline{X}_1,\eta_2+\overline{X}_2,\eta_3+\overline{X}_3)={-}\left(\eta_1[X_2,X_3]+\eta_2[X_3,X_1]+\eta_3[X_1,X_2]\right),
		\end{equation*}
		for all $\eta_1,\eta_2,\eta_3\in\Omega^1(\calF)$ and $X_1,X_2,X_3\in\Gamma(G)$.
	\end{itemize}
\end{lemma}

	The graded algebra $\Gamma(\wedge^\bullet(T^\ast\calF\oplus N\calF))$ has a natural $\bbN\times\bbN$ bigrading, where the bi-degree $(p,q)$ component is given by $\Gamma(\wedge^p T^\ast\calF \otimes \wedge^q N\calF)$.
	It is easy to see that the brackets $\frakn_k$ are compatible with this additional $\bbN\times\bbN$ bigrading; specifically we have
	\begin{equation*}
	\text{bi-degree}(\frakn_1)=(1,0),\quad\text{bi-degree}(\frakn_2)=(0,-1),\quad\text{bi-degree}(\frakn_3)=(-1,-2).
	\end{equation*}
	Consequently, one can easily check that, for all $k$,
	\begin{itemize}
		\item the graded $\Omega^\bullet(\calF)$-submodule $\Omega^\bullet(\calF;N\calF)$ of $\Gamma(\wedge^\bullet(T^\ast\calF\oplus N\calF))$ is closed under $\frakn_k$,
		\item the multibracket $\frakn_k:\Gamma(\wedge^\bullet(T^\ast\calF\oplus N\calF))^{\times k}\to\Gamma(\wedge^\bullet(T^\ast\calF\oplus N\calF))$ induces a map
		\begin{equation*}
		\Omega^\bullet(\calF;N\calF)^{k-1}\times\Omega^\bullet(\calF)\longrightarrow\Omega^\bullet(\calF),\ (u_1,\ldots,u_{k-1},\alpha)\longmapsto\frakn_k(u_1,\ldots,u_{k-1},\alpha),
		\end{equation*}
		which is additionally graded $\Omega^\bullet(\calF)$-linear in the first $k-1$ entries.
	\end{itemize}
	This means that $\Omega^\bullet(\calF;N\calF)[1]$ inherits from $\Gamma(\wedge^\bullet(T^\ast\calF\oplus N\calF))[2]$ an $L_\infty[1]$-algebra structure, whose multibrackets we denote by $\{\frakv_k\}$, which is additionally an $LR_\infty[1]$-algebra over the graded algebra $\Omega^\bullet(\calF)$ (cf. Definition \ref{def:LR_infinity_algebra}).
	So we get the following reformulation of a result first obtained in~\cite[Theorem 29]{vitagliano2014LieRinehart}, using different techniques,   {and later in \cite[Theorem 5.5]{GMS} using the same techniques as ours.}
	
\begin{proposition}
	\label{lem:G_infty_algebra:foliation}
	Let $\calF$ be a foliation on a manifold $M$. 
	There exists a unique $L_\infty[1]$-algebra structure $\{\frakv_k\}$ on $\Omega^\bullet(\calF;N\calF)[1]$ that makes $\Omega^\bullet(\calF;N\calF)[1]$ into an $LR_\infty[1]$-algebra over $\Omega^\bullet(\calF)$, whose only non-trivial brackets $\frakv_1$, $\frakv_2$, $\frakv_3$ are given by:
	\begin{itemize}
		\item $\frakv_1$ is the de Rham differential $\rmd_\nabla$ of the Lie algebroid $T\calF\Rightarrow M$ with coefficients in its representation $N\calF$,
		\item $\frakv_2$ is given by the following formula, for all $\alpha,\beta\in\Omega^\bullet(\calF)$, $X,Y\in\Gamma(G)$,
		\begin{equation*}
		\frakv_2(\alpha\otimes \overline{X},\beta\otimes \overline{Y})={-}{(-)^{|\alpha|}}\left(\alpha\wedge{\pr_{T^{*}\mathcal{F}}}\calL_X\beta\otimes \overline{Y}-{\pr_{T^{*}\mathcal{F}}\calL_Y\alpha}\wedge\beta\otimes \overline{X}+\alpha\wedge\beta\otimes\overline{[X,Y]}\right),
		\end{equation*}
		\item $\frakv_3$ is given by the following formula, for all $\alpha,\beta,\gamma\in\Omega^\bullet(\calF)$ and $X,Y,Z\in\Gamma(G)$,
		\begin{align*}
		 \frakv_3(\alpha\otimes\overline{X},\beta\otimes\overline{Y},\gamma\otimes\overline{Z})=&(-)^{|\beta|+1}\Big((\iota_{[Y,Z]}\alpha)\wedge\beta\wedge\gamma\otimes\overline{X}-(-)^{|\alpha||\beta|}(\iota_{[X,Z]}\beta)\wedge\alpha\wedge\gamma\otimes\overline{Y}\\&+(-)^{|\beta||\gamma|+|\alpha||\gamma|}(\iota_{[X,Y]}\gamma)\wedge\alpha\wedge\beta\otimes\overline{Z}\Big).
		\end{align*}
		\end{itemize}
Its anchor maps are given by the multibrackets $\frakn_k$ of the $L_{\infty}[1]$-algebra appearing in Lemma~\ref{lem:G_infty_algebra:foliation1}.
\end{proposition}

Further, if one specializes also Lemma~\ref{lem:2.6second} to this setting, then one obtains the following. 
\begin{lemma}\label{lem:vita}
	The relation $\calL=\gr(Z)$ establishes a one-to-one correspondence between:
	\begin{enumerate}
		\item MC elements $Z$ of the $L_\infty[1]$-algebra $(\Gamma(\wedge^\bullet(T^\ast\calF\oplus N\calF))[2],\{\frakn_k\})$,
		\item Dirac structures $\calL$ on $M$ close to $T\calF\oplus T^\circ\calF$ in the sense that $\calL\pitchfork G\oplus G^\circ$.
	\end{enumerate}
	Here $\gr(Z)$ denotes the graph of $Z$ viewed as a map $T\calF\oplus T^\circ\calF\to T^\ast\calF\oplus N\calF\cong G\oplus G^\circ$.
\end{lemma}

	Recall that the relation $\calL=D\oplus D^\circ$ establishes a one-to-one correspondence between  distributions $D$ on $M$ and those almost Dirac structures $\calL$ on $M$ such that $\calL=\pr_{TM}\calL{{}\oplus{}}{\pr_{T^\ast M}\calL}$.
	If $D$ and $\calL$ correspond to each other in this way, then $D$ is integrable iff $\calL$ is Dirac.
	Further, for any $Z\in\Gamma(\wedge^2(T^\ast\calF\oplus N\calF))$, it is straightforward to see that $\gr(Z)=\pr_{TM}\gr(Z){{}\oplus{}}\pr_{T^\ast M}\gr(Z)$ if and only if $Z\in\Omega^1(\calF;N\calF)$.
	Consequently, Lemma \ref{lem:vita} yields the following reformulation of the result obtained in~\cite[Theorem 6.2.20]{ji2013thesis} in the context of deformation theory of Lie subalgebroids. 
	
\begin{proposition}
	\label{prop:foliations_DC-elements}
	The relation $T\calF^\prime=\gr(\eta)$ establishes a one-to-one correspondence between: 
	\begin{enumerate}
		\item MC elements $\eta$ of the $L_\infty[1]$-algebra $(\Omega^\bullet(\calF;N\calF)[1],\{\frakv_k\})$, and
		\item rank $d$ foliations $\calF^\prime$ on $M$ close to $\calF$ in the sense that $TM=T\calF^\prime\oplus G$.
	\end{enumerate}
\end{proposition}

\subsection{A strict morphism of $L_\infty[1]$-algebras}
\label{sec:strict_morphism}

Let $\Pi$ be a rank $2k$ regular Poisson structure on $M$ with corresponding symplectic foliation $(\calF,\omega)$, and fix a distribution $G$ on $M$ complementary to $T\calF$.
{We construct a strict morphism between the $L_\infty[1]$-algebras governing the deformations of $\Pi$ and of $\mathcal{F}$}, which we introduced in Proposition \ref{prop:good_L_infty_algebra} and Lemma 
	\ref{lem:G_infty_algebra:foliation}
respectively.

First, the map $\omega^\flat:T\calF\to T^\ast\calF$ determines a degree $0$ graded algebra isomorphism 
$\underline{\smash{\phi}}:\frakX^\bullet(\calF)\to \Omega^\bullet(\mathcal{F})$ which acts as follows:
\begin{equation}
\label{eq:rem:expression:underline_phi}
\underline{\smash\phi}(P)(X_1,\ldots,X_{\ell})={(-1)^{\ell}}P(\omega^\flat X_1,\ldots,\omega^\flat X_{\ell}),
\end{equation}
for all $\ell\geq 0$, $P\in\frakX^{\ell}(\calF)$ and $X_1,\ldots,X_{\ell}\in\Gamma(T\calF)$.
We then proceed by defining a degree $0$ graded module morphism $\phi:\frakX^\bullet_\good(M)[2]\to\Omega^\bullet(\calF;N\calF)[1]$, along the graded algebra isomorphism $\underline{\smash\phi}:\frakX^\bullet(\calF)\to \Omega^\bullet(\mathcal{F})$, as the   composition
\begin{equation*}
\frakX^\bullet_\good(M)[2]=\Gamma(\wedge^{\bullet+2} T\calF)\oplus\Gamma(\wedge^{\bullet+1} T\calF\otimes G)\to\Gamma(\wedge^{\bullet+1} T\calF\otimes G) \to\Gamma(\wedge^{\bullet+1} T^\ast\calF\otimes G)\simeq \Omega^\bullet(\calF;N\calF)[1],
\end{equation*}
where the first map is the projection and the second one is $\underline{\smash\phi}\otimes\id_G$. 
Explicitly, $\phi$ acts as follows, for all $\ell\geq 0$, $P\in\frakX^{\ell}_\good(M)$, $X_1,\ldots,X_{\ell-1}\in\Gamma(T\calF)$ and $\beta\in\Gamma(N^\ast\calF)\simeq\Gamma(G^\ast)$:
\begin{equation}
\label{eq:rem:expression:phi}
\beta\big(\phi(P)(X_1,\ldots,X_{\ell-1})\big)={(-1)^{\ell-1}}P(\omega^\flat X_1,\ldots,\omega^\flat X_{\ell-1},\beta).
\end{equation}

\begin{proposition}
	\label{prop:strict_L_infty-algebra_morphism}
	The degree $0$ graded linear map $\phi:\frakX_\good^\bullet(M)[2]\to\Omega^\bullet(\calF;N\calF)[1]$ induces a strict morphism of $L_\infty[1]$-algebras 
	\begin{equation*}
	\begin{tikzcd}
	(\frakX_\good^\bullet(M)[2],\{\frakl^G_k\})\arrow[rr, "\phi"]&&(\Omega^\bullet(\calF;N\calF)[1],\{\frakv_k\}).
	\end{tikzcd}
	\end{equation*}
\end{proposition}
We defer the proof of  Proposition \ref{prop:strict_L_infty-algebra_morphism} to Appendix \ref{app:proofstrict}. There we also
 provide a conceptual argument under the assumption that the 2-form
$\gamma$ is closed,  showing that in that case $\phi$ is obtained from the  pullback by a Courant algebroid isomorphism.

{\begin{remark}\label{rem:ses}
Lemma~\ref{lem:leafwise_PL_complex}
 and Prop.~\ref{prop:strict_L_infty-algebra_morphism}  show that there is a short exact sequence of $L_{\infty}[1]$-algebras and strict morphisms
 \begin{equation}
 \{0\} \to
  (\frakX^\bullet(\calF)[2],\rmd_\Pi) \to
  (\frakX_\good^\bullet(M)[2],\{\frakl^G_k\})\overset{\phi}{\to}(\Omega^\bullet(\calF;N\calF)[1],\{\frakv_k\})
 \to  \{0\}. 
\end{equation}
This short exact sequence reflects the fact that one obtains deformations of regular Poisson structures by deforming both the leaf-wise symplectic form (see Lemma~\ref{lem:gaugeB}) and the underlying foliation (see Prop. \ref{prop:foliations_DC-elements}).
\end{remark}}

In ~\eqref{eq:forgetful_functor} we displayed a natural geometric map from the space of rank $2k$ Poisson structures to the space of rank $2k$ foliations on $M$. The next proposition states that it ``lifts'' to the strict morphism of $L_\infty[1]$-algebras $\phi:(\frakX_\good^\bullet(M)[2],\{\frakl_k^{G}\})\longrightarrow(\Omega^\bullet(\calF;N\calF)[1],\{\frakv_k\})$ we just constructed, i.e.
that it is  exactly the map of MC elements induced by $\phi$.

\begin{proposition}
	\label{prop:strict_morphism:MC_elements}
	The following diagram commutes 
	\begin{equation}
	\label{eq:prop:strict_morphism:MC_elements}
	\begin{tikzcd}
	\Gamma(\calI_{\gamma})\cap\operatorname{MC}\big(\frakX_\good^\bullet(M)[2],\{\frakl_k^{G}\}\big)\arrow[rr, "\phi"]\arrow[d, hook, swap, "Z\mapsto\exp_G(Z)"]&&\operatorname{MC}\big(\Omega^\bullet(\calF;N\calF)[1],\{\frakv_k\}\big)\arrow[d, hook, "\gr"]\\
	\RegPoiss^{2k}(M)\arrow[rr, swap, "\mathsf{q}"]&&\Fol^{2k}(M)
	\end{tikzcd}.
	\end{equation}
\end{proposition}

\begin{proof}
	An arbitrary $Z\in\Gamma(\calI_{\gamma})\cap\frakX^2_\good(M)$ can be uniquely decomposed as $Z=Z_1+Z_2$, where $Z_1\in\Gamma(\wedge^2T\calF)$ and $Z_2\in\Gamma(T\calF\otimes G)$.
	From a straightforward computation, it follows that
	\begin{equation*}
	\gr(\exp_G(Z))=\calR_\Pi\calR_\gamma(\gr (Z))=\{\Pi^\sharp\alpha{+}Z_2^\sharp\alpha+\alpha+\beta{+}\gamma^\flat Z_1^\sharp\alpha{+}\gamma^\flat Z_2^\sharp\beta\mid\alpha+\beta\in T^\ast\calF\oplus G^\ast\}.
	\end{equation*}
	Further, since $\phi(Z)=\phi(Z_2)=(\underline{\smash\phi}\otimes\id_G)Z_2$, one easily gets that
	\begin{equation*}
	\im((\exp_G(Z))^\sharp)=\operatorname{pr}_{TM}(\gr(\exp_G(Z)))=\{\Pi^\sharp\alpha{+}Z_2^\sharp\alpha\mid\alpha\in T^\ast\calF\}=\{V{-}Z_2^\sharp({\gamma}^{\flat}V)\mid V\in T\calF\}=\gr(\phi(Z)),
	\end{equation*}
	and this proves in particular that the diagram~\eqref{eq:prop:strict_morphism:MC_elements} commutes.
\end{proof}

\section{\textsf{Infinitesimal Deformations of Symplectic Foliations and Obstructions}}
\label{sec:obstructedness}
 In this section, we take a closer look at infinitesimal deformations of regular Poisson structures, showing that in general there exist obstructed infinitesimal deformations.
First in \S\ref{sec:regunob} we single out a class of infinitesimal deformations that are always unobstructed.
Then in \S\ref{sec:regob}, using the Kuranishi criterion,  we discuss an example of a regular Poisson structure with obstructed infinitesimal deformations.

Fix a rank $2k$ regular Poisson structure $\Pi$ on a manifold $M$. 

\begin{definition}
	A \emph{smooth deformation} of $\Pi$ is a smooth path  of bivector fields $\Pi_t$ lying in $\RegPoiss^{2k}(M)$ with $\Pi_0=\Pi$.
\end{definition}

Let $\calF$ denote the characteristic foliation of $\Pi$ and fix a distribution $G$ on $M$ complementary to $T\calF$, 
so that one can construct the associated $L_\infty[1]$-algebra $(\frakX^\bullet_\good(M)[2],\{\frakl_k^{G}\}_{k\in\mathbb{N}})$ as in Proposition \ref{prop:good_L_infty_algebra}.
Without loss of generality, one can assume that all smooth deformations of $\Pi$ come, via the map $\exp_G$, from smooth $1$-parameter families $Z_t$ of MC elements of $(\frakX^\bullet_\good(M)[2],\{\frakl_k^{G}\}_{k\in\mathbb{N}})$, with $Z_0=0$.
If $Z_t$ is such a family, then differentiating the MC equation at $t=0$ one obtains
\begin{equation*}
0=\left.\frac{\rmd}{\rmd t}\right|_{t=0}\left(\rmd_\Pi Z_t+\frac{1}{2}\frakl_2^{G}(Z_t,Z_t)+\frac{1}{6}\frakl_3^{G}(Z_t,Z_t,Z_t)\right)=\rmd_\Pi\left(\left.\frac{\rmd}{\rmd t}\right|_{t=0} Z_t\right),
\end{equation*}
Since the tangent map to $\exp_G:\calI_\gamma\to\Pi+\calI_{-\gamma}$ is the identity map along the points of the zero section of $\wedge^2TM$, we recover the content of Lemma~\ref{lem:formal_tangent_space}, namely that a smooth deformation of $\Pi$ gives rise infinitesimally to a $2$-cocycle in $(\frakX^\bullet_\good(M),\rmd_\Pi)$. This justifies the following definition.

\begin{definition}
	An \emph{infinitesimal deformation} (or \emph{first order deformation}) of $\Pi$ is a $2$-cocycle in the complex $(\frakX_\good^\bullet(M),\rmd_\Pi)$
\end{definition}

As seen above, taking the derivative at time $t=0$, each smooth deformation gives rise to an infinitesimal deformation of $\Pi$. 
Infinitesimal deformations of $\Pi$ that do not arise in this way are called \emph{obstructed}. If such an infinitesimal deformation exists, then the deformation problem of the regular Poisson structure $\Pi$ is said to be \emph{obstructed}. We show in Example~\ref{ex:obstructed_example} below that obstructions can occur,
{implying that the space of rank $2k$ regular Poisson structures may fail to be smooth around $\Pi$.}

\subsection{Unobstructed deformations}
\label{sec:regunob}
An arbitrary infinitesimal deformation $Z\in\mathfrak{X}^{2}_\good(M)$
of $\Pi$ decomposes as \[Z=Z_1+Z_2\in\Gamma(\wedge^{2}T\mathcal{F})\oplus\Gamma(T\mathcal{F}\otimes G).\]
Infinitesimal deformations lying in the first summand $\Gamma(\wedge^{2}T\mathcal{F})$ can be turned into closed elements of $\Gamma(\wedge^{2}T^{*}\mathcal{F})$ using the leafwise symplectic form $\omega$; they are the infinitesimal shadows of gauge transformations of the Poisson structure $\Pi$, which leave the underlying foliation unchanged. In particular, any such infinitesimal deformation is unobstructed, as we show in the next Proposition~\ref{prop:gauge_unobs}.

In constrast, an infinitesimal deformation with nonzero component in $\Gamma(T\mathcal{F}\otimes G)$ projects to a nonzero infinitesimal deformation of the underlying foliation $\mathcal{F}$ via the strict $L_{\infty}[1]$-morphism in Prop.~\ref{prop:strict_L_infty-algebra_morphism}. Consequently, a path of regular Poisson structures that prolongs such an infinitesimal deformation necessarily changes the foliation $\calF$.

\begin{proposition}\label{prop:gauge_unobs}
Assume that $M$ is compact and let $\Pi$ be a regular Poisson structure on $M$ with corresponding symplectic foliation $(\calF,\omega)$. Any infinitesimal deformation $Z\in\Gamma(\wedge^{2}T\calF)$ of $\Pi$ is unobstructed.
\end{proposition}
\begin{proof}
Since $d_{\Pi}Z=0$, the foliated two-form $\beta:=\wedge^{2}\omega^{\flat}(Z)\in\Gamma(\wedge^{2}T^{*}\mathcal{F})$ is leafwise closed. Let $\tilde{\beta}\in\Omega^{2}(M)$ be any extension of $\beta$. Compactness of $M$ implies that there exists $\epsilon>0$ small enough such that the bundle map 
\[
\text{Id}+t\left(\tilde{\beta}\right)^{\flat}\circ\Pi^{\sharp}:T^{*}M\rightarrow T^{*}M
\]
is invertible for all $t\in(-\epsilon,\epsilon)$. By gauge transforming $\Pi$ with $t\tilde{\beta}$ for $t\in(-\epsilon,\epsilon)$, we obtain a path of regular Poisson structures $\Pi^{t\tilde{\beta}}$ characterised by
	\[
	\left(\Pi^{t\tilde{\beta}}\right)^{\sharp}=\Pi^{\sharp}\circ\left(\text{Id}+t(\tilde{\beta})^{\flat}\circ\Pi^{\sharp}\right)^{-1}.
	\]
We claim that the path $\Pi^{t\tilde{\beta}}$ is a prolongation of the infinitesimal deformation $Z$. 
To prove this, note that
\[
\left(\Pi^{t\tilde{\beta}}\right)^{\sharp}\circ\left(\text{Id}+t(\tilde{\beta})^{\flat}\circ\Pi^{\sharp}\right)=\Pi^{\sharp},
\]
and differentiating this equality at time $t=0$, we obtain
	\[
	\left.\frac{d}{dt}\right|_{t=0}\left(\Pi^{t\tilde{\beta}}\right)^{\sharp}=-\Pi^{\sharp}\circ(\tilde{\beta})^{\flat}\circ\Pi^{\sharp}=\left(\wedge^{2}\Pi^{\sharp}(\tilde{\beta})\right)^{\sharp}.
	\]
	This shows that
	\[
	\left.\frac{d}{dt}\right|_{t=0}\Pi^{t\tilde{\beta}}=\wedge^{2}\Pi^{\sharp}(\tilde{\beta})=\wedge^{2}\Pi^{\sharp}(\wedge^{2}\omega^{\flat}(Z))=Z.
	\]
\end{proof}
 
\begin{remark}
We can rephrase the proof of Prop.~\ref{prop:gauge_unobs} making use of Lemmas~\ref{lem:leafwise_PL_complex} and~\ref{lem:gaugeB}. Observe that $Z$ is a MC element of the $L_{\infty}[1]$-subalgebra $(\frakX^\bullet(\calF)[2],\rmd_\Pi)\subset(\frakX^\bullet(M)[2],\{\frakl_k^G\})$. Compactness of $M$ ensures that for small enough times $t$, the path of MC elements $t\mapsto tZ$ stays inside the neighborhood $I_{\gamma}$ introduced in~\eqref{eq:def:I_gamma}. 
So Theorem~\ref{theor:deformation_theory regular_Poisson} gives us a curve $t\mapsto\exp_G(tZ)$ of regular Poisson structures which prolongs the infinitesimal deformation $Z$. By Lemma~\ref{lem:gaugeB}, this curve agrees with the one constructed in the proof above, i.e. $\exp_G(tZ)=\Pi^{t\tilde{\beta}}$. 
\end{remark}

\begin{example}
[{An unobstructed deformation}]\label{ex:DGLAunobst}
Consider $(\mathbb{T}^{4},\theta_1,\theta_2,\theta_3,\theta_4)$ with rank $2$ regular Poisson structure $\Pi=\partial_{\theta_1}\wedge\partial_{\theta_2}$. As a complement to the characteristic distribution of $\Pi$, we take $G:=\text{Span}\{\partial_{\theta_3},\partial_{\theta_4}\}$. The bivector field
\[
\xi:=\partial_{\theta_1}\wedge\partial_{\theta_3}+\partial_{\theta_2}\wedge\partial_{\theta_4}\in\mathfrak{X}^{2}_\good(\mathbb{T}^{4})
\]
is an infinitesimal deformation of $\Pi$, and we claim that it is unobstructed. Indeed, it is tangent to the path 
\begin{equation}\label{path}
\Pi_{t}:=\partial_{\theta_1}\wedge\partial_{\theta_2}+t(\partial_{\theta_1}\wedge\partial_{\theta_3}+\partial_{\theta_2}\wedge\partial_{\theta_4})+t^{2}\partial_{\theta_3}\wedge\partial_{\theta_4},
\end{equation}
which consists of rank $2$ regular Poisson structures since $\wedge^{2}\Pi_{t}=0$. Notice that the quadratic term of $\Pi_{t}$ cannot be omitted because the resulting path would consist of rank $4$ Poisson structures for $t\neq0$.

To interpret this example, we mention a more general fact. Say that we are given a regular Poisson manifold $(M,\Pi)$ whose characteristic distribution allows an involutive complement $G$. If $\xi\in\mathfrak{X}^{2}_\good(M)$ is an infinitesimal deformation of $\Pi$ satisfying $[\xi,\xi]_{\gamma}=0$, then $\xi$ is in fact a MC element of $(\frakX^\bullet_\good(M)[2],\{\frakl_k^{G}\}_{k\in\mathbb{N}})$, {which in this case reduces to a dgL[1]a}. So $t\mapsto t\xi$ is a curve of MC elements, which implies that $t\mapsto\exp_{G}(t\xi)$ is a path of regular Poisson structures prolonging $\xi$. One can check by direct computation that this procedure gives exactly the path that we found in~\eqref{path}.
\end{example}

\begin{remark}\label{rem:representative}
Assume $M$ is compact and let $Z\in\mathfrak{X}^{2}_\good(M)$ be an unobstructed infinitesimal deformation of $\Pi$. Then any representative of the Poisson cohomology class $[Z]\in H^{2}_{\Pi}(M)$ is an unobstructed infinitesimal deformation of $\Pi$. Indeed, pick an arbitrary representative $Z+d_{\Pi}Y$ for some $Y\in\mathfrak{X}(M)$, and assume that $\Pi_{t}$ is a path of regular Poisson structures that is a prolongation of $Z$. Let $\phi_{t}$ denote the flow of $Y$; it is globally defined by compactness of $M$. Then the path of regular Poisson structures $(\phi_{t})_{*}\Pi_{t}$ is a prolongation of $Z+d_{\Pi}Y$, since
\[
\left.\frac{d}{dt}\right|_{t=0}(\phi_{t})_{*}\Pi_{t}=-\calL_{Y}\Pi+\left.\frac{d}{dt}\right|_{t=0}\Pi_{t}=d_{\Pi}Y+Z.
\]
\end{remark}

\subsection{Obstructed deformations}
\label{sec:regob}
The $L_{\infty}[1]$-algebra $(\frakX^\bullet_\good(M)[2],\{\frakl_k^{G}\}_{k\in\mathbb{N}})$ controlling the deformation problem of the regular Poisson structure $\Pi$ provides a sufficient criterion for the existence of obstructions. 
Indeed, obstructions can be detected by means of the \emph{Kuranishi map}
\begin{equation}
\mathrm{Kur}\colon H^2(\frakX_\good^\bullet(M),\rmd_\Pi) \to H^3(\frakX_\good^\bullet(M),\rmd_\Pi),\quad [Z]\mapsto [\frakl_2^{G}(Z,Z)].
\end{equation}

\begin{proposition}
	\label{prop:kuranishi_criterion}
	Let $Z$ be an infinitesimal deformation of $\Pi$.
	If $\operatorname{Kur}[Z]\neq 0$, then $Z$ is obstructed. 
\end{proposition}

The proof is a general argument in deformation theory (see, e.g.,~\cite[Theorem 11.4]{OP}) and we skip it.
We now use the Kuranishi map to provide an example of obstructed infinitesimal deformation.

\begin{example}[An obstructed deformation]\label{ex:obstructed_example}
Consider the $3$-torus $\bbT^3=(S^1)^{\times 3}$ with angular cooordinates $\theta_1,\theta_2,\theta_3$.
We equip $\bbT^3$ with the rank $2$ symplectic foliation $(\calF,\omega)$ corresponding to the regular Poisson structure
\begin{equation*}
\Pi:=\partial_{\theta_1}\wedge\partial_{\theta_2}.
\end{equation*}
We consider the splitting of the tangent bundle $T\bbT^3=G\oplus T\calF$, where
\begin{equation*}
T\calF:=\im\Pi^\sharp=\operatorname{Span}\left\{\partial_{\theta_1},\partial_{\theta_2}\right\}\quad\text{and}\quad G:=\operatorname{Span}\left\{\partial_{\theta_3}\right\}.
\end{equation*}
So $G$ is integrable, and the $2$-form $\gamma=\gamma_G\in\Omega^2(\bbT^3)$ is given by $\gamma=\rmd\theta_1\wedge\rmd\theta_2$. 

Further, an arbitrary $P\in\frakX^2(\bbT^3)$ decomposes as $P=\sum_{i<j}f_{ij}\partial_{\theta_i}\wedge\partial_{\theta_j}$
and one can easily compute\footnote{In this example we make use of the dgL[1]a $(\frakX^\bullet(M)[2],\{\frakl^G_k\})$, for the sake of illustration. Since $codim(T\calF)=1$, alternatively we could have used the more familiar dgLa
$(\mathfrak{X}^{\bullet}(M)[1],d_{\Pi},[-,-]_{\sf SN})$, see  Remark~\ref{rem:corank-one}.
}
 its differential $\frakl_1^G(P)\in\frakX^3(\bbT^3)$ as follows
\begin{equation}
\label{eq:example:m_1}
\frakl_1^G(P)=\rmd_\Pi(P)=\left(\partial_{\theta_1}f_{13}+\partial_{\theta_2}f_{23}\right)\partial_{\theta_1}\wedge\partial_{\theta_2}\wedge\partial_{\theta_3}.
\end{equation}
Let $Q\in\frakX^3(\bbT^3)$ be an arbitrary tri-vector field.
For any $\tau_3\in S^1$, define the $2$-torus $\bbT^2(\tau_3)\subset\bbT^3$ as follows
\begin{equation*}
\bbT^2(\tau_3):=\{(\theta_1,\theta_2,\theta_3)\in\bbT^3\mid\theta_3=\tau_3\}.
\end{equation*}
Equation~\eqref{eq:example:m_1} implies immediately that, if $Q$ is $\frakl_1^G$-exact, then the following condition holds for any $\tau_3\in S^1$:
\begin{equation}
	\label{eq:example:constraint}
	\int_{\bbT^2(\tau_3)} Q\left(\rmd\theta_1\wedge\rmd\theta_2\wedge\rmd\theta_3\right)\cdot\rmd\theta_1\wedge\rmd\theta_2=0.
\end{equation}
From now on, let us consider the bivector field $\widetilde{P}$ on $\bbT^3$ given by
\begin{equation*}
\label{eq:example:widetilde_P}
\widetilde{P}=f(\theta_3)\partial_{\theta_1}\wedge\partial_{\theta_3}+g(\theta_3)\partial_{\theta_2}\wedge\partial_{\theta_3}
\end{equation*}
for arbitrary $f,g\in C^\infty(S^1)$.
Equation~\eqref{eq:example:m_1} shows that $\widetilde{P}$ is an infinitesimal deformation, i.e.~$\frakl_1^G\widetilde P=0$. Since
\begin{equation*}
\frakl_2^G(\widetilde{P},\widetilde{P})=2\frakl_2^G{\left(f(\theta_3)\partial_{\theta_1}\wedge\partial_{\theta_3},g(\theta_3)\partial_{\theta_2}\wedge\partial_{\theta_3}\right)}=2\left({-}f^\prime(\theta_3)g(\theta_3){+}f(\theta_3)g^\prime(\theta_3)\right)\partial_{\theta_1}\wedge\partial_{\theta_2}\wedge\partial_{\theta_3},
\end{equation*}
one gets that for all $\tau_3\in S^1$,
\begin{align}
\label{eq:constraint:bis}
\int_{\bbT^2(\tau_3)}(\frakl_2^G(\widetilde{P},\widetilde{P}))(\rmd\theta_1\wedge\rmd\theta_2\wedge\rmd\theta_3)\cdot\rmd\theta_1\wedge\rmd\theta_2=2(2\pi)^2\left({-}f^\prime(\tau_3)g(\tau_3){+}f(\tau_3)g^\prime(\tau_3)\right).
\end{align}

Fix now the functions $f,g\in C^\infty(S^1)$ such that ${-}f^\prime g{+}fg^\prime\neq 0$, for instance set $f(\theta)=\sin\theta$ and $g(\theta)=\cos\theta$.
In this case, Equation~\eqref{eq:constraint:bis} tells us that $Q:=\frakl_2^G(\widetilde{P},\widetilde{P})$ doesn't satisfy Equation~\eqref{eq:example:constraint}.
So $\frakl_2^G(\widetilde{P},\widetilde{P})$ is not $\frakl_1^G$-exact, and $[\widetilde{P}]\in H^2(\frakX^\bullet(\bbT^3),\rmd_\Pi)$ is not killed by the Kuranishi map, i.e.
\begin{equation*}
\operatorname{Kur}([\widetilde{P}])=[\frakl_2^G(\widetilde{P},\widetilde{P})]\neq 0\in H^3(\frakX^\bullet(\bbT^3),\rmd_\Pi).
\end{equation*}
This means that infinitesimal deformation $\widetilde{P}$ of $\Pi$ is obstructed (see Proposition~\ref{prop:kuranishi_criterion}).
\end{example}

\section{Relating Obstructions to Poisson Structures and to Foliations}
\label{sec:relating}

{
In the previous section \S\ref{sec:obstructedness}, we considered infinitesimal deformations $Z$ of a regular Poisson structure $\Pi$, and established some of their properties.
Clearly, by the two forgetful maps in the diagram below,} $Z$ induces both an infinitesimal deformation of $\Pi$ viewed as a Poisson structure (without constraints on the rank), and an infinitesimal deformation of the foliation underlying $\Pi$. In this section, we relate the (un)obstructe{d}ness of $Z$ to that of these two infinitesimal deformations.

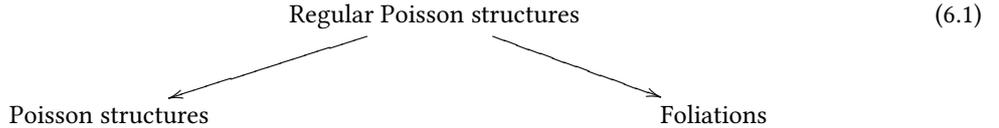
\begin{equation}
\xymatrix{
& \text{Regular Poisson structures}\ar[rd]\ar[ld]&  \\
 \text{Poisson structures}&  &  \text{Foliations}}
\end{equation}

\subsection{Obstructedness as a regular Poisson structure vs. obstructedness as a Poisson structure}

In this subsection, we consider the following problem. Say we are given a regular Poisson structure $(M,\Pi)$, a choice of complement $TM=T\mathcal{F}\oplus G$ and an infinitesimal deformation $\xi\in\mathfrak{X}^{2}_\good(M)$ of $\Pi$. 
\begin{itemize}
    \item Is it possible that $\xi$ is tangent to a path of Poisson structures deforming $\Pi$, but that there exists no such path consisting of \emph{regular} Poisson structures? 
    \item That is, can $\xi$ be obstructed in the deformation problem of $\Pi$ as a \emph{regular} Poisson structure, while it is unobstructed when deforming $\Pi$ just as a Poisson structure?
    \end{itemize}
Assuming $M$ is compact, this phenomenon cannot occur when $\Pi$ has corank one: if $\Pi_t$ is a path of Poisson structures with $\Pi_0=\Pi$, then $\Pi_t$ is also of corank one for small enough $t$. We will show that in higher corank however, the answer is positive. 

In \S \ref{subsubsec:comp} we show that when the extension  $\gamma\in \Omega^2(M)$ of the leaf-wise symplectic form  is   closed, the  Kuranishi criterion (see \S\ref{sec:regob}) is not able to distinguish between the two kinds of obstructedness. Therefore in \S\ref{subsubsec:obex} we focus on the case where $\gamma$ is not closed, and using the Kuranishi criterion we exhibit examples in corank 2 showing that the answer to the above questions is positive.

\subsubsection{{\underline{\smash{Comparing the obstructions}}}}\label{subsubsec:comp}
We first compare at the algebraic level the basic obstructions for extending an infinitesimal deformation of $\Pi$, either to a path of Poisson structures or to a path of \emph{regular} Poisson structures.

\begin{lemma}\label{lem:exact}
Let $(M,\Pi)$ be a regular Poisson manifold, and fix a complement $TM=T\mathcal{F}\oplus G$. If $\xi\in\mathfrak{X}_\good^{2}(M)$ is an infinitesimal deformation of $\Pi$, then $[\xi,\xi]_{\gamma}$ is exact in $\left(\mathfrak{X}^{\bullet}(M),d_{\Pi}\right)$ if and only if $[\xi,\xi]_{SN}$ is exact in $\left(\mathfrak{X}^{\bullet}(M),d_{\Pi}\right)$.
\end{lemma}
\begin{proof}
The two choices of complements $\bbT M=\gr\Pi\oplus TM$ and $\bbT M=\gr\Pi\oplus (G\oplus T^{*}\mathcal{F})$ give rise to $L_{\infty}$-algebras $\left(\Omega^\bullet(\gr\Pi)[1],m_1^{TM},m_2^{TM}\right)$ and $\left(\Omega^\bullet(\gr\Pi)[1],m_1^{G\oplus T^{*}\mathcal{F}},m_2^{G\oplus T^{*}\mathcal{F}},m_3^{G\oplus T^{*}\mathcal{F}}\right)$, as described in Lemma~\ref{lem:2.6first}. As proved in~\cite{GMS}, these are related by a canonical $L_{\infty}$-isomorphism $\{f_n\}_{n\geq 1}$, with $f_1=\text{Id}$. Also using Lemma~\ref{ex:deformation_theory_Poisson_structures} and Proposition~\ref{prop:Koszul_algebra:isomorphism}, we obtain a diagram
\[
\begin{tikzcd}
\left(\Omega^\bullet(\gr\Pi)[1],m_1^{TM},m_2^{TM}\right)\arrow[r,dashed,"Id"]\arrow[d,"\wedge^\bullet(\calR_\Pi|_{T^\ast M})^\ast"] & \left(\Omega^\bullet(\gr\Pi)[1],m_1^{G\oplus T^{*}\mathcal{F}},m_2^{G\oplus T^{*}\mathcal{F}},m_3^{G\oplus T^{*}\mathcal{F}}\right)\arrow[d,"\wedge^\bullet(\calR_\Pi|_{T^\ast M})^\ast"]\\
(\frakX^\bullet(M)[1],\rmd_\Pi,[-,-]_{SN})&  (\frakX^\bullet(M)[1],\{m_k^G\})
\end{tikzcd},
\]
where the vertical maps are strict $L_{\infty}$-isomorphisms and $(\frakX^\bullet(M)[1],\{m_k^G\})$ is the $L_{\infty}$-algebra corresponding with the $L_{\infty}[1]$-algebra $(\frakX^\bullet(M)[2],\{\frakl_k^G\})$. By composition, at the bottom of the above diagram we obtain an $L_{\infty}$-isomorphism from $(\frakX^\bullet(M)[1],\rmd_\Pi,[-,-]_{SN})$ to 
$(\frakX^\bullet(M)[1],\{m_k^G\})$, whose first component is the identity.
As a consequence, the class of $[\xi,\xi]_{SN}$ in $\left(\mathfrak{X}^{\bullet}(M),d_{\Pi}\right)$ coincides with the class of $m_2^G(\xi,\xi)=[\xi,\xi]_{\gamma}$. This proves the lemma.
\end{proof}

However, the obstruction for extending $\xi$ to a path of regular Poisson structures is the class of $[\xi,\xi]_{\gamma}$ in the cohomology of the subcomplex $(\mathfrak{X}_\good^{\bullet}(M),d_{\Pi})\subset(\mathfrak{X}^{\bullet}(M),d_{\Pi})$. Under an additional condition, $[\xi,\xi]_{\gamma}$ is exact in $(\mathfrak{X}_\good^{\bullet}(M),d_{\Pi})$ if and only if it is exact in $(\mathfrak{X}^{\bullet}(M),d_{\Pi})$, as we now show.

\begin{lemma}\label{lem:injective}
Let $(M,\Pi)$ be a regular Poisson manifold, and choose a complement $TM=T\mathcal{F}\oplus G$. If $d\gamma=0$, then the inclusion  $(\mathfrak{X}_\good^{\bullet}(M),d_{\Pi})\hookrightarrow(\mathfrak{X}^{\bullet}(M),d_{\Pi})$ induces an injective map in cohomology.
\end{lemma}
\begin{proof}
We will use the bi-grading $\frakX^\bullet(M)=\bigoplus_{q,p\in\bbN}\frakX^{(p,q)}(M)$ introduced in~\eqref{bigrading}. While in general $d_\Pi\mathfrak{X}^{(p,q)}(M)\subset\mathfrak{X}^{(p+1,q)}(M)\oplus\mathfrak{X}^{(p+2,q-1)}(M)$ (cf. Lemma~\ref{lem:multibrackets:NxN_grading}), we make the following

\emph{Claim: The assumption that $\gamma$ is closed ensures that $d_\Pi\mathfrak{X}^{(p,q)}(M)\subset\mathfrak{X}^{(p+1,q)}(M)$.}  

\noindent
Since $d_{\Pi}(\Gamma(T\mathcal{F}))\subset\Gamma(\wedge^{2}T\mathcal{F})$, this claim follows if we show that $d_{\Pi}(\Gamma(G))\subset\Gamma(T\mathcal{F}\otimes G)$. 
Since $T\mathcal{F}$ is involutive, around any point there exists a local frame $\{Y_1,\ldots,Y_l\}$ for $G$ consisting of infinitesimal automorphisms of $\mathcal{F}$. Notice that $d_{\Pi}Y_i=0$ for such generators $Y_i\in\Gamma(G)$, since  
\[
\calL_{Y_i}\omega=r(\calL_{Y_i}\gamma)=r(d\iota_{Y_i}\gamma+\iota_{Y_i}d\gamma)=0,
\]
{where we used that $d\gamma=0$.}
Here $\omega\in\Gamma(\wedge^{2}T^{*}\mathcal{F})$ is the leafwise symplectic form and $r:\Omega^{\bullet}(M)\rightarrow\Omega^{\bullet}(\mathcal{F})$ is the restriction map. Consequently, for a local section $\sum_{i}f_{i}Y_{i}$ of $G$ we obtain
\[
d_{\Pi}\left(\sum_{i=1}^{l}f_{i}Y_{i}\right)=\sum_{i=1}^{l}(d_{\Pi}f_i)\wedge Y_i\in\Gamma(T\mathcal{F}\otimes G).
\]
This proves that $d_{\Pi}(\Gamma(G))\subset\Gamma(T\mathcal{F}\otimes G)$, which confirms our claim that $d_\Pi\mathfrak{X}^{(p,q)}(M)\subset\mathfrak{X}^{(p+1,q)}(M)$.

We now prove the lemma. Assume that $\xi\in\mathfrak{X}^{n}_\good(M)=\mathfrak{X}^{(n,0)}(M)\oplus\mathfrak{X}^{(n-1,1)}(M)$ is exact, i.e. $\xi=d_{\Pi}\varphi$ for some $\varphi\in\mathfrak{X}^{n-1}(M)$. We have to show that $\xi$ has a primitive in $\mathfrak{X}_\good^{n-1}(M)$. We decompose $\varphi=\varphi_\good+\varphi_{\text{rest}}$ in the direct sum
\[
\mathfrak{X}^{n-1}(M)=\left(\mathfrak{X}^{(n-1,0)}(M)\oplus\mathfrak{X}^{(n-2,1)}(M)\right)\oplus\left(\oplus_{j>1}\mathfrak{X}^{(n-1-j,j)}(M)\right).
\]
Since $\xi=d_{\Pi}\varphi_\good+d_{\Pi}\varphi_{\text{rest}}$, we have $d_{\Pi}\varphi_{\text{rest}}=\xi-d_{\Pi}\varphi_\good\in\mathfrak{X}_\good^{n}(M)$. But by the claim we just proved, we also know that $d_{\Pi}\varphi_{\text{rest}}$ is contained in $\oplus_{j>1}\mathfrak{X}^{(n-j,j)}(M)$, so that
\[
d_{\Pi}\varphi_{\text{rest}}\in\mathfrak{X}_\good^{n}(M)\cap\left(\oplus_{j>1}\mathfrak{X}^{(n-j,j)}(M)\right)=\{0\}.
\]
In conclusion, $\xi=d_{\Pi}\varphi_\good$, which shows that $\xi$ has a primitive in $\mathfrak{X}_\good^{n-1}(M)$.
\end{proof}

Combining Lemmas~\ref{lem:exact} and~\ref{lem:injective}, we obtain the following result. It states that if $\gamma$ is closed, then the primary obstructions for extending an infinitesimal deformation $\xi$ of $\Pi$ to a path of Poisson structures resp. a path of \emph{regular} Poisson structures are equivalent.
So the Kuranishi criterion cannot distinguish obstructedness in the regular Poisson deformation problem from obstructedness in the Poisson deformation problem.

\begin{corollary}\label{cor:Kuranisheq}
Let $(M,\Pi)$ be a regular Poisson manifold, and choose a complement $TM=T\mathcal{F}\oplus G$. Assume moreover that $d\gamma=0$. If $\xi\in\mathfrak{X}_\good^{2}(M)$ is an infinitesimal deformation of $\Pi$, then $[\xi,\xi]_{\gamma}$ is exact in $\left(\mathfrak{X}_\good^{\bullet}(M),d_{\Pi}\right)$ precisely when $[\xi,\xi]_{SN}$ is exact in $\left(\mathfrak{X}^{\bullet}(M),d_{\Pi}\right)$.
\end{corollary}

\begin{remark}\label{rem:variation}
    The conclusion of Lemma~\ref{lem:injective} fails in general if $\gamma$ is not closed. Consider the regular Poisson manifold $\big(\mathbb{T}^{4},\Pi=(\sin\theta_4+2)\partial_{\theta_1}\wedge\partial_{\theta_2}\big)$. In this case, the leafwise symplectic form $\omega$ is given by
    \[
    \omega=\frac{1}{\sin\theta_4+2}d\theta_1\wedge d\theta_2.
    \]
    There exists no closed $2$-form on $\mathbb{T}^{4}$ extending $\omega$, since the leafwise variation of $\omega$ {(see \cite[Def. 1.2.14]{Osorno})}
    \[
    var_{\omega}=\left[\left(\frac{-\cos\theta_4}{(\sin\theta_4+2)^{2}}d\theta_1\wedge d\theta_2\right)\otimes d\theta_{4}\right]\in H^{2}(\mathcal{F};N^{*}\mathcal{F})
    \]
    is nonzero. Now consider the trivector field $$W:=\cos\theta_4\partial_{\theta_1}\wedge\partial_{\theta_2}\wedge\partial_{\theta_3}\in\mathfrak{X}_\good^3(\mathbb{T}^{4}).$$ Its Poisson cohomology class $[W]\in H^{3}_{\Pi}(\mathbb{T}^{4})$ is trivial since $d_{\Pi}(\partial_{\theta_3}\wedge\partial_{\theta_4})=W$. 
    But $W$ is not exact in the complex $\big(\mathfrak{X}_\good^{\bullet}(\mathbb{T}^{4}),d_{\Pi}\big)$. Indeed, if it were exact then it would be of the form
    \begin{align*}
    &d_{\Pi}\left(f_{12}\partial_{\theta_1}\wedge\partial_{\theta_2}+f_{13}\partial_{\theta_1}\wedge\partial_{\theta_3}+f_{14}\partial_{\theta_1}\wedge\partial_{\theta_4}+f_{23}\partial_{\theta_2}\wedge\partial_{\theta_3}+f_{24}\partial_{\theta_2}\wedge\partial_{\theta_4}\right)\\
    &\hspace{1cm}=(\sin\theta_4+2)\left(\left(\frac{\partial f_{13}}{\partial_{\theta_1}}+\frac{\partial f_{23}}{\partial_{\theta_2}}\right)\partial_{\theta_1}\wedge\partial_{\theta_2}\wedge\partial_{\theta_3}+\left(\frac{\partial f_{14}}{\partial_{\theta_1}}+\frac{\partial f_{24}}{\partial_{\theta_2}}\right)\partial_{\theta_1}\wedge\partial_{\theta_2}\wedge\partial_{\theta_4}\right),
    \end{align*}
    which would imply that
    \[
    \cos\theta_4=(\sin\theta_4+2)\left(\frac{\partial f_{13}}{\partial_{\theta_1}}+\frac{\partial f_{23}}{\partial_{\theta_2}}\right)=\partial_{\theta_1}\big((\sin\theta_4+2)f_{13}\big)+\partial_{\theta_2}\big((\sin\theta_4+2)f_{23}\big).
    \]
    But then we reach a contradiction: defining $\mathbb{T}^{2}(\tau_3,\tau_4)$ to be the $2$-subtorus obtained fixing the coordinates $\theta_3,\theta_4$ to arbitrary constants $\tau_3,\tau_4\in S^{1}$, we get
    \begin{align*}
    4\pi^{2}\cos\tau_4&=\int_{\mathbb{T}^{2}(\tau_3,\tau_4)}\cos\theta_4 d_{\theta_1}\wedge d\theta_2\\
    &=\int_{\mathbb{T}^{2}(\tau_3,\tau_4)}\big(\partial_{\theta_1}\big((\sin\theta_4+2)f_{13}\big)+\partial_{\theta_2}\big((\sin\theta_4+2)f_{23}\big)\big)d_{\theta_1}\wedge d\theta_2=0
    \end{align*}
    This shows that the inclusion $(\mathfrak{X}_\good^{\bullet}(M),d_{\Pi})\hookrightarrow(\mathfrak{X}^{\bullet}(M),d_{\Pi})$ is not injective in cohomology.
\end{remark}

\subsubsection{{\underline{\smash{Obstructed examples which are unobstructed as Poisson structures}}}}\label{subsubsec:obex}
{By Corollary~\ref{cor:Kuranisheq} and Remark~\ref{rem:variation}}, in order to find a regular Poisson structure $\Pi$ and an infinitesimal deformation of $\Pi$ that is obstructed for the regular Poisson deformation problem, but unobstructed for the Poisson deformation problem, we should look at examples where the foliated symplectic form $\omega$ is not tamed by a closed $2$-form. That is why we will look at examples where $\omega$ has non-zero leafwise variation.

\begin{example}\label{ex:StephaneobstructedReg}
Consider again the regular Poisson manifold $\big(\mathbb{T}^{4},\Pi=(\sin\theta_4+2)\partial_{\theta_1}\wedge\partial_{\theta_2}\big)$, as in Remark~\ref{rem:variation} above. As a complement to the characteristic distribution, we take $G=\text{Span}\{\partial_{\theta_3},\partial_{\theta_4}\}$, so that $\gamma\in\Omega^{2}(\mathbb{T}^{4})$ is given by
    \[
    \gamma=\frac{1}{\sin\theta_4+2}d\theta_1\wedge d\theta_2.
    \]
Consider the infinitesimal deformation $\xi:=\sin\theta_4\partial_{\theta_1}\wedge\partial_{\theta_3}+\partial_{\theta_2}\wedge\partial_{\theta_4}\in\mathfrak{X}_\good^{2}(\mathbb{T}^{4})$ of $\Pi$. If we deform $\Pi$ as a Poisson structure, then $\xi$ is unobstructed since it is tangent to the path of Poisson structures 
\[
\Pi_{t}:=(\sin\theta_4+2)\partial_{\theta_1}\wedge\partial_{\theta_2}+t\left(\sin\theta_4\partial_{\theta_1}\wedge\partial_{\theta_3}+\partial_{\theta_2}\wedge\partial_{\theta_4}\right)+t^{2}\partial_{\theta_3}\wedge\partial_{\theta_4}.
\]
Notice that the path $\Pi_t$ does not deform $\Pi$ as a regular Poisson structure. Indeed, $\wedge^{2}\Pi_{t}=4t^{2}\partial_{\theta_1}\wedge\partial_{\theta_2}\wedge\partial_{\theta_3}\wedge\partial_{\theta_4}$, so that $\Pi_t$ has rank $4$ for $t\neq0$. In fact, we claim that $\xi$ is not tangent to a path of rank $2$ Poisson structures. To see why, we compute
\[
[\xi,\xi]_{\gamma}=-4\frac{\cos\theta_4}{\sin\theta_4+2}\partial_{\theta_1}\wedge\partial_{\theta_2}\wedge\partial_{\theta_3},
\]
which is not exact in $(\mathfrak{X}_\good^{\bullet}(M),d_{\Pi})$. Indeed, if it were exact with primitive $\varphi\in\mathfrak{X}_\good^{2}(\mathbb{T}^{4})$, then we would have that
\[
W:=\cos\theta_4\partial_{\theta_1}\wedge\partial_{\theta_2}\wedge\partial_{\theta_3}=d_{\Pi}\left(-\frac{\sin\theta_4+2}{4}\varphi\right)
\]
is exact in $(\mathfrak{X}_\good^{\bullet}(M),d_{\Pi})$, which violates Remark~\ref{rem:variation}. So $\xi$ is an infinitesimal deformation of $\Pi$, which is tangent to a path of Poisson structures but not to a path of rank $2$ Poisson structures deforming $\Pi$.
\end{example}

We present a family of examples that {generalize} Ex.~\ref{ex:StephaneobstructedReg}.

\begin{example}\label{ex:genexStephane}
Consider   the regular Poisson manifold $\big(\mathbb{T}^{4},\Pi=h(\theta_3,\theta_4)\partial_{\theta_1}\wedge\partial_{\theta_2}\big)$,
where $h$ is nowhere-vanishing. As a complement to the characteristic distribution, we take $G=\text{Span}\{\partial_{\theta_3},\partial_{\theta_4}\}$.
An element of the form
$$\xi:=f(\theta_3,\theta_4)\partial_{\theta_1}\wedge\partial_{\theta_3}+g(\theta_3,\theta_4)\partial_{\theta_2}\wedge\partial_{\theta_4}\in\mathfrak{X}_\good^{2}(\mathbb{T}^{4})$$
is an infinitesimal deformation  of $\Pi$, i.e. $[\Pi,\xi]=0$.

 \emph{Claim: If the infinitesimal deformation $\xi$ can be extended to a smooth path of rank-2 Poisson structures, then   
\begin{equation}\label{eq:extorusrank2}
 g \frac{\partial (f/h)}{\partial \theta_4}=f \frac{\partial (g/h)}{\partial \theta_3}=0.
 \end{equation}}
To prove the claim, we relate $\xi$ to foliations as in Prop.~\ref{prop:strict_morphism:obstructions} below. Recall from Prop.~\ref{prop:strict_L_infty-algebra_morphism} that we have a strict morphism of $L_{\infty}[1]$-algebras
$\phi:(\frakX_\good^\bullet(M)[2],\{\frakl^G_k\})\rightarrow(\Omega^\bullet(\calF;N\calF)[1],\{\frakv_k\})$. 
We have
$$\phi(\xi)={-}\left(d\theta_1\otimes \frac{g}{h} \partial_{\theta_4}-d\theta_2\otimes \frac{f}{h}\partial_{\theta_3}\right),$$  
and we compute
$$\frakv_2(\phi(\xi),\phi(\xi))={-}2d\theta_1\wedge d\theta_2\otimes 
\left(
\frac{g}{h} \frac{\partial (f/h)}{\partial \theta_4}\partial_{\theta_3}
- \frac{f}{h} \frac{\partial (g/h)}{\partial \theta_3}\partial_{\theta_4}
\right).$$
Notice that the two coefficients in the round bracket  are functions of $\theta_3$ and $\theta_4$ only. 
In~\cite[\S6.2]{SZPre} it was shown that for all $\alpha\in \Omega^1(\calF;G)$, we have $\int_{\mathbb{T}^{2}(\tau_3,\tau_4)} \frakv_1 \alpha=0$, where $\mathbb{T}^{2}(\tau_3,\tau_4)$ denotes the 2-subtorus obtained fixing the last two coordinates to arbitrary constants $\tau_3\in S^1$ and $\tau_4\in S^1$. Hence $\frakv_2(\phi(\xi),\phi(\xi))$ is $\frakv_1$-exact  only when it vanishes identically. In other words, $\mathrm{Kur}[\phi(\xi)]= 0$ if{f} Eq.~\eqref{eq:extorusrank2} holds. By Prop.~\ref{prop:strict_morphism:obstructions}, this proves the claim.

We now provide instances in which $\xi$ can be extended to a smooth path of  Poisson structures,
albeit not of constant rank 2.
Notice that $$[\xi,\xi]_{SN}=2\left(f\frac{\partial g}{\partial \theta_3}\partial_{\theta_1}\wedge\partial_{\theta_2}\wedge\partial_{\theta_4}-
g\frac{\partial f}{\partial \theta_4}\partial_{\theta_1}\wedge\partial_{\theta_2}\wedge\partial_{\theta_3}\right).
$$
We consider two cases.

\begin{itemize}
\item[a)] 
Suppose $f=f(\theta_3)$ and $g=g(\theta_4)$. Then we have $[\xi,\xi]_{SN}=0$, so $$\Pi+t\xi$$ is a Poisson structure.  
Further $\Pi+t\xi$ has ``square'' $-2t^2 fg \;\partial_{\theta_1}\wedge\partial_{\theta_2}\wedge\partial_{\theta_3}\wedge\partial_{\theta_4}$, so for $t\neq0$ it has constant rank two  only if $fg=0$ vanishes.

 Notice that Eq.~\eqref{eq:extorusrank2} boils down to $fg\frac{\partial h}{\partial \theta_4}=fg\frac{\partial h}{\partial \theta_3}=0$. So, for instance, for 
any non-constant function $h(\theta_3,\theta_4)$ without zeros
 and  $f=g=1$,  we see that the infinitesimal deformation $\xi$ is unobstructed in the realm of Poisson structures but is obstructed in the realm of regular Poisson structures.

\item[b)]  Suppose $f=f(\theta_4)$, $g=g(\theta_3)$, i.e.~the opposite dependence than in case a).
Set $h:=\frac{1}{a} f g +C$ where $C\in \mathbb{R}$ and $a\in \mathbb{R}\setminus\{0\}$ are constants (we assume that $h$ is nowhere vanishing). Then 
$$\Pi_{t}:=\Pi+t\xi+t^2 X$$
is Poisson, for $X:=a\partial_{\theta_3}\wedge \partial_{\theta_4}$, since one checks easily that $[\xi,X]_{SN}=0$ and $2[\Pi,X]_{SN}=-[\xi,\xi]_{SN}$.
Further $\Pi_t$ has ``square'' $2t^2 Ca\;\partial_{\theta_1}\wedge\partial_{\theta_2}\wedge\partial_{\theta_3}\wedge\partial_{\theta_4}$, so for $t\neq0$  the Poisson structure $\Pi_t$  has rank 2 only if $C=0$.

Now Eq.~\eqref{eq:extorusrank2} boils down to $g\frac{\partial f}{\partial \theta_4}C=f\frac{\partial g}{\partial \theta_3}C=0$. So, for instance, for   $f=\sin(\theta_4)$, $g=1$, $a=1$ and $C=2$ (so that $h=\sin(\theta_4)+2$ ),
 we see that the infinitesimal deformation $\xi$ is unobstructed in the realm of Poisson structures but obstructed in the realm of regular Poisson structures. 
This is exactly Ex.~\ref{ex:StephaneobstructedReg}. 
\end{itemize}
\end{example}

\begin{remark}
For the regular Poisson manifold $\big(\mathbb{T}^{4},\partial_{\theta_1}\wedge\partial_{\theta_2}\big)$,
the leafwise symplectic form clearly admits a closed extension $\gamma\in \Omega^2(\mathbb{T}^{4})$. Therefore, because of  Corollary~\ref{cor:Kuranisheq}, the Kuranishi criterium is not able to establish if an infinitesimal deformation is obstructed in the realm of regular Poisson manifolds while being unobstructed in the realm of  Poisson manifolds.
Hence we are led to modify this regular Poisson structure multiplying it with a nowhere vanishing Casimir function $h$, i.e. to consider 
$\Pi=h(\theta_3,\theta_4)\partial_{\theta_1}\wedge\partial_{\theta_2}$: this is what we do in Ex.~\ref{ex:genexStephane}.

Notice that if the leafwise symplectic form $h^{-1}(\theta_3,\theta_4)d\theta_1\wedge d\theta_2$ admits a closed extension $\gamma\in \Omega^2(\mathbb{T}^{4})$, then $h$ has to be constant. To see this, for all pairs of points $(\tau_3,\tau_4)$ and $(\tau'_3,\tau'_4)\in \mathbb{T}^{2}$,
apply Stokes' theorem to $\gamma$ on a 3-dimensional submanifold of the form 
$\mathbb{T}^{2}\times P$, where $P\subset \mathbb{T}^{2}$ is an arc joining $(\tau_3,\tau_4)$ to 
$(\tau'_3,\tau'_4)$. {Alternatively, one can use the leafwise variation of $h^{-1}(\theta_3,\theta_4)d\theta_1\wedge d\theta_2$, which is given by
\begin{equation}\label{eq:variation}
\left[d\theta_1\wedge d\theta_2\otimes\left(\frac{\partial h^{-1}}{\partial\theta_3}d\theta_3+\frac{\partial h^{-1}}{\partial\theta_4}d\theta_4\right)\right]\in H^{2}(\mathcal{F};N^{*}\mathcal{F}).
\end{equation}
If the leafwise symplectic form $h^{-1}(\theta_3,\theta_4)d\theta_1\wedge d\theta_2$ admits a closed extension, then its variation needs to be zero. 
Since the foliation $\mathcal{F}$ is given by the fibers of $\mathbb{T}^{4}\rightarrow(\mathbb{T}^{2},\theta_3,\theta_4)$
, one has $H^{2}(\mathcal{F};N^{*}\mathcal{F})\cong\Omega^{1}(\mathbb{T}^{2};H^{2}(\mathbb{T}^{2}))$, where the first copy of $\mathbb{T}^{2}$ is the base with coordinates $(\theta_3,\theta_4)$ and the second one is the fiber with coordinates $(\theta_1,\theta_2)$. Hence, the leafwise variation~\eqref{eq:variation} being zero implies that both $(\partial h^{-1}/\partial\theta_3)d\theta_1\wedge d\theta_2$ and $(\partial h^{-1}/\partial\theta_4)d\theta_1\wedge d\theta_2$ are leafwise exact, so that necessarily $h$ is constant.
}
\end{remark}

\subsection{Relating obstructedness with the underlying foliation}\label{subsec:obsfol}

We discuss, and illustrate with examples, how obstructedness of a regular Poisson structure $\Pi$ relates with its characteristic foliation $\calF$.
{The main tool is the strict $L_{\infty}[1]$-morphism $\phi$ introduced in Prop.~\ref{prop:strict_L_infty-algebra_morphism}: if an infinitesimal deformation of $\Pi$ is unobstructed, then the corresponding infinitesimal deformation of $\calF$ is also unobstructed (see Proposition~\ref{prop:strict_morphism:obstructions} below).}

We ask whether the converse  holds:
\begin{itemize}
    \item  Given an infinitesimal deformation $Z$ of $\Pi$ such that  $\phi(Z)$ is tangent to a path of foliations deforming $\calF$, is $Z$ itself tangent to a path of regular Poisson structures?
        \item In other words, does the unobstructedness of   $\phi(Z)$ imply the unobstructedness of  $Z$?
\end{itemize}
In  \S\ref{subsubsec:Unobstructedness results} we 
 present a few situations in which the infinitesimal deformation  $Z$ is unobstructed.
However, as we show in \S\ref{subsubsec:obstructedness results},  in general the answer to the above questions is negative, i.e.  the converse of Proposition~\ref{prop:strict_morphism:obstructions} does not hold. This means that the obstructedness of an infinitesimal deformation of $\Pi$  is not ``due'' exclusively to the obstructedness of the corresponding infinitesimal deformation of $\calF$. Finally, in \S\ref{subsubsec:stab} we present some remarks about a related question, namely 
the stability of symplectic foliations.

\bigskip
Fix a rank $d$ foliation $\calF$ on a manifold $M$.
{We first recall some terminology, in analogy to \S\ref{sec:obstructedness}.}

\begin{definition}
A \emph{smooth deformation} of the rank $d$ foliation $\calF$ is a smooth path $\calF_t$ in $\Fol^d(M)$ with $\calF_0=\calF$.
\end{definition}
Upon choosing a complement $G$ to the distribution $T\calF$, we can construct the $L_{\infty}[1]$-algebra $(\Omega^\bullet(\calF;N\calF)[1],\{\frakv_k\})$ introduced in Prop.~\ref{lem:G_infty_algebra:foliation}, 
which governs the deformation problem of $\calF$.
Being interested in small deformations of the foliation $\calF$, one can assume that all $\calF_t$'s in a smooth deformation of $\calF$ are transverse to $G$. Consequently, in view of Prop.~\ref{prop:foliations_DC-elements}, there is a unique smooth $1$-parameter family $\eta_t$, with $\eta_0=0$, of MC elements of $(\Omega^\bullet(\calF;N\calF)[1],\{\frakv_k\})$, such that $T\calF_t=\gr(\eta_t)$.
Differentiating the MC equation for $\eta_t$ at $t=0$, one gets that $\rmd_\nabla\dot\eta_0=0$.
{Additionally, it is easy to see that the $1$-cocycle $\dot\eta_0$ in $(\Omega^\bullet(\calF;N\calF),\rmd_\nabla)$ coincides exactly with the \emph{infinitesimal deformation} associated by Heitsch to the smooth deformation $\calF_t$ of the foliation $\calF$ (cf.~\cite[Def.~2.7 and Cor.~2.11]{heitsch1975cohomology}).}
This justifies the following definition.

\begin{definition}
An \emph{infinitesimal deformation} of $\calF$ is a $1$-cocycle in the complex $(\Omega^\bullet(\calF;N\calF),\rmd_\nabla)$ .
\end{definition}

So each smooth deformation gives rise, as its derivative at $t=0$, to an infinitesimal deformation.
The converse is generally false: there may exist \emph{obstructed infinitesimal deformations}, i.e.~infinitesimal deformations of $\calF$ which do not arise from smooth deformations. This reflects the fact that the space of rank $d$ foliations may fail to be smooth around $\calF$.
The $L_{\infty}[1]$-algebra $(\Omega^\bullet(\calF;N\calF)[1],\{\frakv_k\})$ controlling the deformation problem of $\calF$ gives a criterion for the existence of obstructions. 
Indeed, obstructions can be detected by means of the \emph{Kuranishi map}
\begin{equation}
\mathrm{Kur}\colon H^1(\Omega^\bullet(\calF;N\calF),\rmd_\nabla) \to H^2(\Omega^\bullet(\calF;N\calF),\rmd_\nabla),\quad [\eta]\mapsto [\frakv_2(\eta,\eta)].
\end{equation}
\begin{proposition}
	\label{prop:kuranishi_criterion:foliations}
	Let $\eta$ be an infinitesimal deformation of $\calF$.
	If $\operatorname{Kur}[\eta]\neq 0$, then $\eta$ is obstructed.
\end{proposition}

From now on, let $\calF$ be a foliation on $M$ that comes from a regular Poisson structure $\Pi\in\mathfrak{X}^{2}(M)$. Recall that the $L_{\infty}[1]$-algebra $(\frakX_\good^\bullet(M)[2],\{\frakl^G_k\})$ governing the deformations of $\Pi$ and the $L_{\infty}[1]$-algebra $(\Omega^\bullet(\calF;N\calF)[1],\{\frakv_k\})$ governing those of $\calF$ are related by a strict $L_{\infty}[1]$-morphism (see Prop.~\ref{prop:strict_L_infty-algebra_morphism})
\begin{equation}\label{eq:morphism}
\phi:(\frakX_\good^\bullet(M)[2],\{\frakl^G_k\})\rightarrow(\Omega^\bullet(\calF;N\calF)[1],\{\frakv_k\}).
\end{equation}
Via this morphism, one can detect obstructed infinitesimal deformations of the regular Poisson structure $\Pi$.

\begin{proposition}
	\label{prop:strict_morphism:obstructions}
	Let $Z$ be an infinitesimal deformation of the regular Poisson structure $\Pi$, i.e.~$Z\in\frakX^2_\good(M)$ and $\rmd_\Pi Z=0$.
	Then $\phi(Z)$ is an infinitesimal deformation of the characteristic foliation $\calF$, i.e.~$\phi(Z)\in\Omega^1(\calF;N\calF)$ and $\rmd_\nabla\phi(Z)=0$.
	Further,
\begin{itemize}
\item $\mathrm{Kur}[\phi(Z)]\neq 0\Longrightarrow\mathrm{Kur}[Z]\neq 0$
\item if  $\phi(Z)$ is obstructed, then $Z$ is obstructed as well.
\end{itemize}	
	\end{proposition}

\begin{proof}
	The first part follows immediately from $\phi$ being a strict $L_\infty[1]$-algebra morphism (Prop.~\ref{prop:strict_L_infty-algebra_morphism}) and the consequent identity $[\phi]\circ\mathrm{Kur}=\mathrm{Kur}\circ[\phi]$, where $[\phi]$ denotes the map induced by $\phi$ in cohomology.

	To prove the last statement of the proposition, assume that $Z$ is an unobstructed infinitesimal deformation of $\Pi$, which is tangent to a path of deformations $\Pi_t$.  We can assume that the path $\Pi_t$ comes from a smooth family of MC elements $Z_t$ of $(\frakX_\good^\bullet(M)[2],\{\frakl_k^{G}\})$, via $\Pi_t=\exp_G(Z_t)$. 
	Then $\phi(Z_t)$ is a smooth family of MC elements of $(\Omega^\bullet(\calF;N\calF)[1],\{\frakv_k\})$, 
and satisfies
	\begin{equation*}
	\left.\frac{\rmd}{\rmd t}\right|_{t=0}\phi(Z_t)=\phi\left(\left.\frac{\rmd}{\rmd t}\right|_{t=0}Z_t\right)=\phi(Z).
	\end{equation*}
Here in the first equality we used the linearity of $\phi$, and in the second that
	by Remark~\ref{rem:expmap} we have
	\[
	Z=\left.\frac{d}{dt}\right|_{t=0}\Pi_t=\left.\frac{d}{dt}\right|_{t=0}\exp_G(Z_t)=\left.\frac{d}{dt}\right|_{t=0}Z_t.
	\]
This shows that, if $Z$ is unobstructed, then $\phi(Z)$ is unobstructed as well, finishing the proof.

We remark that, as one might expect,  $\phi(Z)$ is tangent to the path of foliations 
	underlying the regular Poisson structures $\Pi_t$. Indeed
the smooth family  $\phi(Z_t)$   corresponds with the path of foliations $\gr(\phi(Z_t))=q(\Pi_t)$. Here we used Prop.~\ref{prop:strict_morphism:MC_elements}, denoting by $\textsf{q}:\operatorname{\RegPoiss}^{2k}(M)\rightarrow\Fol^{2k}(M)$ the obvious map (see eq.~\ref{eq:forgetful_functor}).
\end{proof}

\begin{remark}\label{rem:folobstr}
	Proposition~\ref{prop:strict_morphism:obstructions} sheds a new light on Example~\ref{ex:obstructed_example}.
	In that example we constructed an obstructed infinitesimal deformation $w$ for a regular Poisson structure $\Pi$ on $\bbT^3$.
	Actually, one can check that $\mathrm{Kur}[\phi(w)]\neq 0$ and so $\phi(w)$ is an obstructed infinitesimal deformation of the characteristic foliation $\calF$ of $\Pi$ on $\bbT^3$ (cf.~Proposition~\ref{prop:kuranishi_criterion:foliations}).
	This and Proposition~\ref{prop:strict_morphism:obstructions} gives another proof of the fact that $\mathrm{Kur}[w]\neq 0$ and so the infinitesimal deformation $w$ is obstructed.
\end{remark}

\subsubsection{\underline{\smash{Unobstructedness results}}}
\label{subsubsec:Unobstructedness results}

 We present a few conditions guaranteeing that if an infinitesimal deformation $Z$ of a  regular Poisson structure is such that the corresponding deformation of the foliation is unobstructed, then $Z$ itself is unobstructed.
 As earlier, $\phi$ denotes the strict $L_{\infty}[1]$-morphism 
of equation~\eqref{eq:morphism}.

\begin{proposition}\label{dim2}
Let $(M,\Pi)$ be a compact regular Poisson manifold with $2$-dimensional symplectic leaves. {If an infinitesimal deformation $Z\in \frakX_\good^2(M)$ of $\Pi$ is such that $\phi(Z)$ is unobstructed, then $Z$ itself is unobstructed.} In particular, if the deformation problem of the underlying foliation $\mathcal{F}$ is unobstructed, then also the deformation problem of $\Pi$ as a rank-$2$ Poisson structure is unobstructed.
\end{proposition}
\begin{proof}
Fix a complement $G$ to $T\mathcal{F}$ and assume that $Z=Z_1+Z_2\in\Gamma(\wedge^{2}T\mathcal{F})\oplus\Gamma(T\mathcal{F}\otimes G)$ is an infinitesimal deformation of $\Pi$ so that $\phi(Z)$ is unobstructed. As before, denote by $\gamma\in\Omega^{2}(M)$ the extension of the leafwise symplectic form $\omega$ by zero on $G$, i.e.
$
\gamma|_{T\mathcal{F}\times T\mathcal{F}}=\omega$ and $ \iota_{Y}\gamma=0\ \ \forall Y\in\Gamma(G).$

By assumption, the infinitesimal deformation $\phi(Z_2)$ of $\mathcal{F}$ is tangent to a path of foliations $T\mathcal{F}_{t}=\text{Graph}(\Phi_t)$, for some $\Phi_t\in\Gamma(T^{*}\mathcal{F}\otimes G)$. For small enough $t$,  since $M$ is compact,  we have that $\gamma$ is still non-degenerate on the leaves of $\mathcal{F}_{t}$, and it is automatically leafwise closed for dimension reasons. 
So we obtain Poisson structures
 $\Pi_t:=\mathcal{F}_{t}^{{-}\gamma}$, defined by gauge transforming $\mathcal{F}_{t}$ by ${-}\gamma$. We also have a foliated two-form $\alpha:=\wedge^{2}\omega^{\flat}(Z_1)$ which is automatically leafwise closed; denote by $\tilde{\alpha}$ any extension of $\alpha$. Gauge transforming $\Pi_t$ with $t\tilde{\alpha}$ for small enough $t$, we obtain a path of 
Poisson structures $$\Pi_{t}^{t\tilde{\alpha}}=\mathcal{F}_{t}^{{-}\gamma+t\tilde{\alpha}}.$$ We claim that this path is a prolongation of the infinitesimal deformation $Z_1+Z_2$.

First, since the Dirac structure $\mathcal{F}_{t}^{{-}\gamma}$ given by
\[
\mathcal{F}_{t}^{{-}\gamma}=\left\{v+\Phi_t(v)+\beta{-}\gamma^{\flat}\left(v+\Phi_t(v)\right):v\in T\mathcal{F}, \beta\in T\mathcal{F}_{t}^{0}\right\}
\]
corresponds with the Poisson structure $\Pi_t$, we have that
\[
{-}\Pi_{t}^{\sharp}\circ\gamma^{\flat}\circ(\text{Id}_{T\mathcal{F}}+\Phi_{t})=\text{Id}_{T\mathcal{F}}+\Phi_t.
\]
{Notice that the L.H.S. is just ${-}\Pi_{t}^{\sharp}\circ\gamma^{\flat}|_{T\mathcal{F}}$. Differentiating at time $t=0$, we get
$$
{-}\left.\frac{d}{dt}\right|_{t=0}\Pi_{t}^{\sharp}\circ\gamma^{\flat}
=\left.\frac{d}{dt}\right|_{t=0}\Phi_{t}.
$$
}

{As we know that $\frac{d}{dt}|_{t=0}\Pi_t$ lies in $\frakX^2_\good(M)$ by Lemma~\ref{lem:formal_tangent_space},
} this proves that 
$\frac{d}{dt}|_{t=0}\Pi_t\in\Gamma(T\mathcal{F}\otimes G)$ and that $$\phi\left(\left.\frac{d}{dt}\right|_{t=0}\Pi_t\right)=\left.\frac{d}{dt}\right|_{t=0}\Phi_{t}=\phi(Z_2).$$ Hence, $\frac{d}{dt}|_{t=0}\Pi_t=Z_2$.
Next, we have by definition that
\[
\left(\Pi_{t}^{t\tilde{\alpha}}\right)^{\sharp}\circ\left(\text{Id}+t\tilde{\alpha}^{\flat}\circ\Pi_{t}^{\sharp}\right)=\Pi_{t}^{\sharp}.
\]
Again differentiating at time $t=0$, we obtain
\begin{align*}
\left.\frac{d}{dt}\right|_{t=0}\left(\Pi_{t}^{t\tilde{\alpha}}\right)^{\sharp}&=\left.\frac{d}{dt}\right|_{t=0}\Pi_{t}^{\sharp}-\Pi_{0}^{\sharp}\circ \left(\wedge^{2}\omega^{\flat}(Z_1)\right)^{\flat}\circ\Pi_{0}^{\sharp}\\
&=\left.\frac{d}{dt}\right|_{t=0}\Pi_{t}^{\sharp}+\left(\wedge^{2}\Pi_{0}^{\sharp}\left(\wedge^{2}\omega^{\flat}(Z_1)\right)\right)^{\sharp}\\
&=Z_{2}^{\sharp}+Z_{1}^{\sharp}.
\end{align*}
This confirms that $\Pi_{t}^{t\tilde{\alpha}}$ is a prolongation of the infinitesimal deformation $Z_1+Z_2$, so the proof is finished.
\end{proof}

{
We present a simple example for the construction in Proposition~\ref{dim2}.
\begin{example}
 On $M=\mathbb{R}^3$, take the rank-2  Poisson structure $\Pi=\partial_x\wedge\partial_y$ and the infinitesimal deformation $Z:=Z_2:=\partial_x\wedge \partial_z$, and extend the leaf-wise symplectic form to $\gamma:=dx\wedge dy\in \Omega^2(M)$.  The image $\phi(Z)=dy\otimes \partial_z$ can be prolonged to the path of MC elements $\{t dy\otimes \partial_z\}$, corresponding to the foliations $\calF_t=\gr(tdy\otimes \partial_z)=\text{Span}\{\partial_x,\partial_y+t\partial_z\}.$
 Then  $\mathcal{F}_{t}^{-\gamma}=\gr(\Pi_t)$ for $\Pi_t:=\partial_x\wedge (\partial_y+t\partial_z)$, which indeed constitutes a path of rank-2  Poisson structure which prolongs $Z$.  
 
 {Notice that, even though $M$ is not compact, $\gamma\in \Omega^2(M)$ happens to be non-degenerate on the leaves of $\calF_t$, for all $t\in \mathbb{R}$.
A compact example can be obtained replacing $\mathbb{R}^3$ by the 3-torus $\mathbb{R}^3/\mathbb{Z}^3$. 
 }
\end{example}
}

The following is a variation of Proposition~\ref{dim2} in which leaves are allowed to have arbitrary dimension, but the hypotheses are more stringent.

\begin{proposition}
\label{prop:extendclosed}
Let $(M,\Pi)$ be a compact regular Poisson manifold. Consider an infinitesimal deformation $Z=Z_1+Z_2\in \frakX_\good^2(M)$ of $\Pi$.
Assume that   $\phi(Z)$ is unobstructed, that $\gamma$ is \emph{closed}, and that the leafwise two-form $\wedge^{2}\omega^{\flat}(Z_1)\in \Omega^2_{closed}(\mathcal{F})$ extends to a \emph{closed} 2-form on $M$.
Then $Z$ itself is unobstructed.
\end{proposition}

\begin{proof}
{We first argue that $\wedge^{2}\omega^{\flat}(Z_1)$ is indeed leafwise closed. As shown in the proof of Lemma~\ref{lem:injective}, the assumption that $d\gamma=0$ ensures that $d_{\Pi}\mathfrak{X}^{(p,q)}(M)\subset\mathfrak{X}^{(p+1,q)}(M)$. So if $Z_1+Z_2\in\Gamma(\wedge^{2}T\mathcal{F})\oplus\Gamma(T\mathcal{F}\otimes G)$ is an infinitesimal deformation of $\Pi$, then
\[
0=d_\Pi Z_1+d_\Pi Z_2\in\Gamma(\wedge^{3}T\mathcal{F})\oplus\Gamma(\wedge^{2}T\mathcal{F}\otimes G),
\]
implying in particular that $d_\Pi Z_1=0$. As a consequence, $\wedge^{2}\omega^{\flat}(Z_1)$ is leafwise closed.}
The proof of
Prop.~\ref{dim2} shows that the path of bivector fields $\mathcal{F}_{t}^{{-}\gamma+t\tilde{\alpha}}$
is a prolongation of the infinitesimal deformation $Z$, for any extension $\tilde{\alpha}$ of $\wedge^{2}\omega^{\flat}(Z_1)$. By assumption we can choose $\tilde{\alpha}$ to be closed. The form ${-}\gamma+t\tilde{\alpha}$ is closed for every $t$, so it restricts to a closed foliated form on $\mathcal{F}_{t}$, showing that $\mathcal{F}_{t}^{{-}\gamma+t\tilde{\alpha}}$ is Poisson.
\end{proof}

\begin{remark}
{Let $(M,\Pi)$ be a compact regular Poisson manifold. Take an infinitesimal deformation $Z\in \frakX_\good^2(M)$ of $\Pi$ such that $\phi(Z)$ is unobstructed. Then $Z$ can be prolonged to a smooth path of regular bivector fields (not necessarily Poisson), each of which spans an
involutive distribution.
Indeed, such a path is provided by $\mathcal{F}_{t}^{{-}\gamma+t\tilde{\alpha}}$ 
as in the proof of
Prop.~\ref{dim2}.
}    
\end{remark}

We now show that all infinitesimal deformations of the regular Poisson structure $\Pi$ are unobstructed if the underlying foliation is infinitesimally rigid.

\begin{proposition}\label{prop:infrig}
	Let $(M,\Pi)$ be a compact regular Poisson manifold with characteristic distribution $T\mathcal{F}$. {Then every infinitesimal deformation $Z\in \frakX_\good^2(M)$ of $\Pi$ such that $[\phi(Z)]=0\in H^{1}(\mathcal{F};N\mathcal{F})$ is unobstructed.} In particular, if $H^{1}(\mathcal{F};N\mathcal{F})=0$, then the deformation problem of $\Pi$ is unobstructed.
\end{proposition}
\begin{proof}
	Fix a complement $G$ to $T\mathcal{F}$, and let $Z=Z_1+Z_2\in\Gamma(\wedge^{2}T\mathcal{F})\oplus\Gamma(T\mathcal{F}\otimes G)$ be an infinitesimal deformation  of $\Pi$ satisfying
	\[
	\varphi(Z_1+Z_2)=\varphi(Z_2)=d_{\nabla}Y
	\]
	for some $Y\in\Gamma(G)$. By Remark~\ref{rem:representative} and Prop.~\ref{prop:gauge_unobs}, the infinitesimal deformation $Z$ is unobstructed as soon as we find a representative $\tilde{Z}$ of the Poisson cohomology class $[Z]\in H^{2}_{\Pi}(M)$ satisfying $\tilde{Z}\in\Gamma(\wedge^{2}T\calF)$. We have
	\[
	\varphi(Z_2{-}d_{\Pi}Y)=d_{\nabla}Y-d_{\nabla}(\varphi(Y))=d_{\nabla}Y-d_{\nabla}Y=0.
	\]
	So we can write $Z_1+Z_2=(Z_1+Z_2{-}d_{\Pi}Y){+}d_{\Pi}Y:=\tilde{Z}{+}d_{\Pi}Y$, where $\tilde{Z}\in\Gamma(\wedge^{2}T\mathcal{F})$. This finishes the proof.
\end{proof}

\begin{remark}
A restatement of the proof of Prop.~\ref{prop:infrig} is as follows.
The short exact sequence (of cochain complexes) given in Remark~\ref{rem:ses} induces a long exact sequence in cohomology. The assumption that $[\phi(Z)]=0$ in $H^{1}(\mathcal{F};N\mathcal{F})$ implies that
 $[Z]\in H^2(\frakX_\good^\bullet(M))$ is represented by some cocycle 
 $W$ lying in $\frakX^2(\calF)$. We can work with $W$ instead of $Z$, by  Remark~\ref{rem:representative}.
Now $W$ is unobstructed by Prop.~\ref{prop:gauge_unobs}.
\end{remark}

The following example is an illustration of the proposition above.

\begin{example}[{All infinitesimal deformations are unobstructed}]
	Consider the manifold $S^{1}\times S^{2}$, and let $\psi$ denote the coordinate on $S^{1}$. Let $\Pi$ be the Poisson structure $\Pi$ on $S^{1}\times S^{2}$ for which the symplectic leaves are given by copies $\{\psi\}\times S^{2}$ endowed with the standard symplectic structure $\omega_{S^{2}}$ on $S^{2}$. Since the characteristic foliation $\mathcal{F}$ is a fibration defined by the closed one-form $d\psi$, we have $H^1(\mathcal{F};N\mathcal{F})\cong H^{1}(\mathcal{F})\cong C^{\infty}\left(S^{1},H^{1}(S^{2})\right)=0.$ So the above proposition ensures that the deformation problem of $\Pi$ is unobstructed. We double-check that this is indeed the case.
	
	Since any Poisson structure close enough to $\Pi$ is also regular of corank one, it is enough to show that any infinitesimal deformation of $\Pi$ is tangent to a path of Poisson structures. First notice that, since $\Pi$ is induced by a cosymplectic structure and $H^{1}(\mathcal{F})=0$, we have an isomorphism~\cite[Thm.~3.2.17]{Osorno}
	\begin{equation}\label{iso}
	\wedge^{2}\Pi^{\sharp}:H^{2}(\mathcal{F})\rightarrow H^{2}_{\Pi}(S^{1}\times S^{2}):[f(\psi)\omega_{S^{2}}]\mapsto[f(\psi)\Pi].
	\end{equation}
	
	By Remark~\ref{rem:representative}, we only need to check that infinitesimal deformations of the form $f(\psi)\Pi$ are unobstructed. This is clearly the case, since a prolongation is given by the path $\Pi+tf(\psi)\Pi=(1+tf(\psi))\Pi$ for small enough $t$. Note indeed that this path consists of Poisson structures since $1+tf(\psi)$ is a Casimir of $\Pi$ for each value of $t$. This confirms that the deformation problem of $\Pi$ is unobstructed.
\end{example}

\subsubsection{\underline{\smash{Obstructed deformations with unobstructed underlying foliations}}}
\label{subsubsec:obstructedness results}

{In Proposition~\ref{prop:strict_morphism:obstructions} we saw that given an
  infinitesimal deformation $Z$ of a regular Poisson structure, if $\varphi(Z)$ is obstructed then $Z$ also is.}
One can wonder if all obstructed infinitesimal deformations $Z$ of a regular Poisson structure   arise
{in this way, i.e. whether the obstructedness of $Z$ is ``due'' exclusively to the obstructedness of the corresponding infinitesimal deformation of the  underlying foliation}. We display some examples, showing that the answer is negative.  Even more, in Example \ref{ex:obs-unobs} we 
display a regular Poisson structure $\Pi$  such that \emph{all} infinitesimal deformations of the characteristic foliation    are unobstructed, whereas the deformation problem of $\Pi$ is obstructed.

The examples concern corank-one Poisson structures of cosymplectic type  on compact manifolds of the form $S^{1}\times N$, and we first prove some statements about obstructedness in this setting. Recall from Remark~\ref{rem:corank-one} that in the corank-one case, the usual dgLa $(\mathfrak{X}^{\bullet}(S^{1}\times N)[1],d_{\Pi},[-,-]_{\sf SN})$ can be used to study deformations of the regular Poisson structure $\Pi$.

\begin{lemma}\label{lem:cosymplectic}
Let $(N,\omega)$ be a compact symplectic manifold. Consider the manifold $S^{1}\times N$ with cosymplectic structure $(d\psi,\omega)$, where $\psi$ denotes the coordinate on $S^{1}$. Correspondingly, there is a corank-one Poisson structure $\Pi$ on $S^{1}\times N$ whose symplectic leaves are $\left(\{\psi\}\times N,\omega\right)$ and for which $\partial_{\psi}$ is a transverse Poisson vector field. Take $G:=\text{Span}\{\partial_{\psi}\}$ as a complement to the characteristic distribution $T\calF=\ker(d\psi)$.
\begin{enumerate}
    \item Denote by $\text{Kur}_{\calF}$ the Kuranishi map of the $L_{\infty}[1]$-algebra $(\Omega^\bullet(\calF;G)[1],\{\frakv_k\})$ governing the deformations of $\calF$. Then $\text{Kur}_{\calF}$ is trivial if{f} $\wedge:H^{1}(N)\times H^{1}(N)\rightarrow H^{2}(N)$ is trivial. 
    \item If $H^{1}(N)$ is one-dimensional, then the deformation problem of $\calF$ is unobstructed.
    \item Denote by $\text{Kur}_{\Pi}$ the Kuranishi map of the dgLa $(\mathfrak{X}^{\bullet}(S^{1}\times N)[1],d_{\Pi},[-,-]_{\sf SN})$ governing the deformations of $\Pi$. Then $\text{Kur}_{\Pi}$ is trivial if{f} both $\wedge:H^{1}(N)\times H^{1}(N)\rightarrow H^{2}(N)$ and $\wedge:H^{1}(N)\times H^{2}(N)\rightarrow H^{3}(N)$ are trivial.
\end{enumerate}
\end{lemma}
\begin{proof}
{Throughout the proof, we will use that $H^\bullet(\calF;G)\cong H^\bullet(\calF)\otimes \mathbb{R}\partial_{\psi}\cong 
C^{\infty}(S^1,H^\bullet(N))\otimes \mathbb{R}\partial_{\psi}$, where the first isomorphism holds  since $\partial_{\psi}$ provides a trivialization of $G$
for which the Bott connection is trivial. 
}
\begin{enumerate}
    \item Fix a basis $\{[\alpha_1],\ldots,[\alpha_k]\}$ of $H^{1}(N)$, so that
    \[
    H^{1}(\calF;G)=\left\{\left[\sum_{i=1}^{k}f_i(\psi)\alpha_i\otimes\partial_{\psi}\right]:f_i\in C^{\infty}(S^{1})\right\}.
    \]
    Looking at Prop.~\ref{lem:G_infty_algebra:foliation}, we see that the Kuranishi map $\text{Kur}_{\calF}$ is given by the formula
    \begin{align*}
    \text{Kur}_{\calF}\left[\sum_{i=1}^{k}f_i(\psi)\alpha_i\otimes\partial_{\psi}\right]&=\left[\frakv_{2}\left(\sum_{i=1}^{k}f_i(\psi)\alpha_i\otimes\partial_{\psi},\sum_{i=1}^{k}f_i(\psi)\alpha_i\otimes\partial_{\psi}\right)\right]\\
    &={-}\left[\sum_{i,j=1}^{k}(f'_i(\psi)f_{j}(\psi)-f_i(\psi)f'_j(\psi))\alpha_i\wedge\alpha_j\otimes\partial_{\psi}\right]
    \end{align*}
    Consequently, if $\wedge:H^{1}(N)\times H^{1}(N)\rightarrow H^{2}(N)$ is trivial, then also $\text{Kur}_{\calF}$ is trivial. Conversely, assume that there exist $i,j\in\{1,\ldots,k\}$ such that $\alpha_{i}\wedge\alpha_{j}$ is not exact. Fix two functions $f_{i},f_{j}\in C^{\infty}(S^{1})$ such that $f'_i(\psi)f_{j}(\psi)-f_i(\psi)f_{j}'(\psi)$ is not identically zero. Then $\text{Kur}_{\calF}[f_i(\psi)\alpha_{i}\otimes\partial_{\psi}+f_j(\psi)\alpha_j\otimes\partial_{\psi}]$ is nonzero.
    {We remark that using the
    isomorphism 
    $H^\bullet(\calF;G)\cong 
C^{\infty}(S^1,H^\bullet(N))$, the Kuranishi map simply reads $[a]\mapsto -2[ \partial_{\psi}a\wedge a]$.}    
    \item Fixing a generator $[\alpha]$ of $H^{1}(N)$, we have
\[
H^{1}(\mathcal{F},G)=\left\{[f(\psi)\alpha\otimes\partial_{\psi}]:f\in C^{\infty}(S^{1})\right\}.
\]
We first check that an infinitesimal deformation of the form $f(\psi)\alpha\otimes\partial_{\psi}$ is unobstructed. If $f\equiv 0$, there is nothing to prove, so assume that $f$ is not identically zero. Notice that 
\begin{equation}\label{eq:graph}
\text{Graph}(f(\psi)\alpha\otimes\partial_{\psi})=\{v+f(\psi)\alpha(v)\partial_{\psi}:v\in T\mathcal{F}\}=\text{Ker}(f(\psi)\alpha-d\psi).
\end{equation}
Consider the path of one-forms $\beta_t:=(1-t)d\psi+tf(\psi)\alpha.$ They all give rise to a foliation on $S^{1}\times N$ since $\beta_t$ is nowhere zero and 
\[
\beta_t\wedge d\beta_t=\left((1-t)d\psi+tf(\psi)\alpha\right)\wedge t f'(\psi)d\psi\wedge\alpha=0.
\]
The path $\text{Ker}(\beta_t)$ is a prolongation of $f(\psi)\alpha\otimes\partial_{\psi}$ since $\text{Ker}(\beta_0)=\text{Ker}(d\psi)=T\mathcal{F}$ and
\[
\text{Ker}\left(\left.\frac{d}{dt}\right|_{t=0}\beta_t\right)=\text{Ker}(f(\psi)\alpha-d\psi)=\text{Graph}(f(\psi)\alpha\otimes\partial_{\psi}),
\]
using~\eqref{eq:graph} in the last equality. In general, an arbitrary infinitesimal deformation of $\mathcal{F}$ is of the form $f(\psi)\alpha\otimes\partial_{\psi}+d_{\nabla}Y$ for some $Y\in\Gamma(G)$. Denote by $(\varphi_t)$ the flow of $Y$, which is globally defined because $S^{1}\times N$ is compact. As a consequence of~\cite[Thm. 2.6]{SZequivalences}, we know that the path of involutive distributions $(\varphi_t)_{*}(T\mathcal{F})$ is a prolongation of the second summand $d_{\nabla}Y$. By what we showed above, there exists a prolongation $T\mathcal{F}_{t}$ of the first summand $f(\psi)\alpha\otimes\partial_{\psi}$. Then $(\varphi_{t})_{*}(T\mathcal{F}_{t})$ is a prolongation of $f(\psi)\alpha\otimes\partial_{\psi}+d_{\nabla}Y$, since
\[
\left.\frac{d}{dt}\right|_{0}(\varphi_{t})_{*}(T\mathcal{F}_{t})=\left.\frac{d}{dt}\right|_{0}(\varphi_{t})_{*}(T\mathcal{F})+\left.\frac{d}{dt}\right|_{0}T\mathcal{F}_{t}=\text{Graph}(d_{\nabla}Y+f(\psi)\alpha\otimes\partial_{\psi}).
\]
\item Because the Poisson structure $\Pi$ is induced by a cosymplectic structure, we have short exact sequences in cohomology~\cite[Thm. 3.2.17]{Osorno} 
\[
0\rightarrow H^{\bullet}(\mathcal{F})\overset{\wedge^{\bullet}\Pi^{\sharp}}{\longrightarrow} H_{\Pi}^{\bullet}(S^{1}\times N)\rightarrow H^{\bullet-1}(\mathcal{F})\rightarrow 0.
\]
Consequently, choosing bases $\{[\alpha_1],\ldots,[\alpha_k]\}$ of $H^{1}(N)$ and $\{[\beta_1],\ldots,[\beta_l]\}$ for $H^{2}(N)$, we get
\[
H_{\Pi}^{2}(S^{1}\times N)=\left\{\left[\sum_{i=1}^{l}f_{i}(\psi)\wedge^{2}\Pi^{\sharp}(\beta_{i})\right]:f_i\in C^{\infty}(S^{1})\right\}\oplus\left\{\left[\sum_{j=1}^{k}g_{j}(\psi)\Pi^{\sharp}(\alpha_j)\wedge\partial_{\psi}\right]:g_j\in C^{\infty}(S^{1})\right\}.
\]
As a result, the Kuranishi map $\text{Kur}_{\Pi}$ is given by
\begin{align}\label{eq:sch}
&\text{Kur}_{\Pi}\left[\sum_{i=1}^{l}f_{i}(\psi)\wedge^{2}\Pi^{\sharp}(\beta_{i})+\sum_{j=1}^{k}g_{j}(\psi)\Pi^{\sharp}(\alpha_j)\wedge\partial_{\psi}\right]\nonumber\\
&\hspace{0.5cm}=\Big[\sum_{i,j}f_{i}f_{j}\left[\wedge^{2}\Pi^{\sharp}(\beta_i),\wedge^{2}\Pi^{\sharp}(\beta_j)\right]_{SN}+2\sum_{i,j}\left(g_{j}f_{i}\left[\wedge^{2}\Pi^{\sharp}(\beta_i),\Pi^{\sharp}(\alpha_j)\right]_{SN}\wedge\partial_{\psi}+g_{j}f'_{i}\wedge^{3}\Pi^{\sharp}(\alpha_{j}\wedge\beta_{i})\right)\nonumber\\
&\hspace{1.5cm}+\sum_{i,j}(g_{i}g'_{j}-g'_{i}g_{j})\wedge^{2}\Pi^{\sharp}(\alpha_{i}\wedge\alpha_j)\wedge\partial_{\psi}\Big]\nonumber\\
&\hspace{0.5cm}=\Big[\sum_{i,j}f_{i}f_{j}\wedge^{3}\Pi^{\sharp}([\beta_i,\beta_j]_{\Pi})+2\sum_{i,j}\left(g_{j}f_{i}\wedge^{2}\Pi^{\sharp}([\beta_i,\alpha_j]_{\Pi})\wedge\partial_{\psi}+g_{j}f'_{i}\wedge^{3}\Pi^{\sharp}(\alpha_{j}\wedge\beta_{i})\right)\nonumber\\
&\hspace{1.5cm}+\sum_{i,j}(g_{i}g'_{j}-g'_{i}g_{j})\wedge^{2}\Pi^{\sharp}(\alpha_{i}\wedge\alpha_j)\wedge\partial_{\psi}\Big].
\end{align}
Here $[\cdot,\cdot]_{\Pi}$ denotes the Koszul bracket~\cite{SZPre}, which is defined by the rules
\begin{align*}
&[\eta,\xi]_{\Pi}=(-1)^{|\eta|+1}\left(\calL_{\Pi}(\eta\wedge\xi)-(\calL_{\Pi}\eta)\wedge\xi-(-1)^{|\eta|}\eta\wedge\calL_{\Pi}\xi\right),\\
&\calL_{\Pi}=[\iota_{\Pi},d]=\iota_{\Pi}\circ d - d\circ\iota_{\Pi},
\end{align*}
and we used~\cite[Lemma 2.11]{SZPre}, which implies that $\wedge^{\bullet}\Pi^{\sharp}$ intertwines the Koszul bracket and the Schouten bracket. It is clear that the Koszul bracket of closed forms is exact, and therefore the summands in~\eqref{eq:sch} involving $[\cdot,\cdot]_{\Pi}$ are trivial in Poisson cohomology. If $\wedge:H^{1}(N)\times H^{1}(N)\rightarrow H^{2}(N)$ and $\wedge:H^{1}(N)\times H^{2}(N)\rightarrow H^{3}(N)$ are trivial, then also the remaining summands in~\eqref{eq:sch} are zero in cohomology. Conversely, if either $\wedge:H^{1}(N)\times H^{1}(N)\rightarrow H^{2}(N)$ or $\wedge:H^{1}(N)\times H^{2}(N)\rightarrow H^{3}(N)$ is not trivial, then~\eqref{eq:sch} can be arranged to be nonzero in cohomology since $\wedge^{\bullet}\Pi^{\sharp}:H^{\bullet}(\calF)\rightarrow H^{\bullet}_{\Pi}(S^{1}\times N)$ is injective. This finishes the proof.
\end{enumerate}
\end{proof}

The next example features a regular Poisson structure $\Pi$ with obstructed infinitesimal deformations that project to unobstructed infinitesimal deformations of the foliation $\mathcal{F}$ under the strict $L_{\infty}[1]$-morphism~\eqref{eq:morphism}. This shows that Prop.~\ref{prop:strict_morphism:obstructions} does not cover all obstructed infinitesimal deformations of the regular Poisson structure.  Lemma \ref{lem:cosymplectic} is not used in this example, but it will be needed for a variation   just below. 

\begin{example}
[{An obstructed deformation with  unobstructed deformation of the foliation}]
\label{projobs}
	Consider $S^{1}\times\mathbb{T}^{4}$ with coordinates $(\psi,\theta_1,\theta_2,\theta_3,\theta_4)$ and corank-one Poisson structure $\Pi$ given by $\Pi=\partial_{\theta_1}\wedge\partial_{\theta_2}+\partial_{\theta_3}\wedge\partial_{\theta_4}$. As a complement to the characteristic distribution $T\mathcal{F}$, we take $G=\text{Span}\{ \partial_{\psi}\}$. We claim that any bivector field of the form
	\begin{equation}\label{bivector}
		\xi:=f(\psi)\partial_{\theta_1}\wedge\partial_{\theta_2}+\partial_{\theta_3}\wedge\partial_{\psi}
	\end{equation}
	with {$f(\psi)$ non-constant}
	is an obstructed infinitesimal deformation of $\Pi$ that projects to an unobstructed infinitesimal deformation of $\mathcal{F}$. Denoting by $\phi$ the strict $L_{\infty}[1]$-morphism~\eqref{eq:morphism}, one checks that
	$
	\phi(\xi)=d\theta_4\otimes\partial_{\psi}.
	$
	\begin{itemize}
		\item {We first look at the case where $f=c$ is constant. Then the path $\Pi_t$ given by
		\[
		\Pi_t=(1+ct)\partial_{\theta_1}\wedge\partial_{\theta_2}+\partial_{\theta_3}\wedge(\partial_{\theta_4}	+t\partial_{\psi})
		\]
		is a path of Poisson structures with the desired velocity~\eqref{bivector} at time zero. Since 
		\[
		\wedge^{2}\Pi_t=2(1+ct)\partial_{\theta_1}\wedge\partial_{\theta_2}\wedge\partial_{\theta_3}\wedge\partial_{\theta_4}+2t(1+ct)\partial_{\theta_1}\wedge\partial_{\theta_2}\wedge\partial_{\theta_3}\wedge\partial_{\psi},
		\]
		it follows that this path consists of corank-one Poisson structures for small enough $t\in(-\epsilon,\epsilon)$. Moreover, the path $(\Pi_t)_{t\in(-\epsilon,\epsilon)}$ gives rise to a path of foliations
		\[
		\im\Pi_t^{\sharp}=\text{Span}\{\partial_{\theta_1},\partial_{\theta_2},\partial_{\theta_3},\partial_{\theta_4}+t\partial_{\psi}\}=\text{Graph}(td\theta_4\otimes\partial_{\psi})
		\]
		that prolongs the infinitesimal deformation $d\theta_4\otimes\partial_{\psi}$. So if $f$ is constant, this example confirms Proposition~\ref{prop:strict_morphism:obstructions}, showing that unobstructedness of $\xi\in\mathfrak{X}^{2}_{\good}(S^{1}\times\mathbb{T}^{4})$ implies unobstructedness of $\phi(\xi)\in\Omega^{1}(\mathcal{F};G)$.}

		\item Now assume that $f$ is not constant. We show that the bivector field $\xi$ cannot be realized as the velocity at time $0$ of a path of Poisson structures $\Pi_t$ with $\Pi_0=\Pi$. To do so, we compute the Schouten bracket
		\begin{equation}\label{trivector}
			[f(\psi)\partial_{\theta_1}\wedge\partial_{\theta_2}+\partial_{\theta_3}\wedge\partial_{\psi},f(\psi)\partial_{\theta_1}\wedge\partial_{\theta_2}+\partial_{\theta_3}\wedge\partial_{\psi}]_{SN}=2f'(\psi)\partial_{\theta_1}\wedge\partial_{\theta_2}\wedge\partial_{\theta_3},
		\end{equation}
		and by the Kuranishi criteron, it suffices to show that this trivector field defines a non-trivial class in $H^{3}_{\Pi}(S^{1}\times\mathbb{T}^{4})$. Since the Poisson structure $\Pi$ is induced by the cosymplectic structure $(d\psi,d\theta_1\wedge d\theta_2+d\theta_3\wedge d\theta_4)$, it gives an injective map in cohomology
		\[
		\wedge^{3}\Pi^{\sharp}:H^{3}(\mathcal{F})\rightarrow H^{3}_{\Pi}(S^{1}\times\mathbb{T}^{4}),
		\]
		which follows from~\cite[Proposition 1.4.7]{Osorno} and~\cite[Theorem 3.2.17]{Osorno}. Since the class of the trivector field~\eqref{trivector} is the image under this map of the class
		\[
		\left[-2f'(\psi)d\theta_1\wedge d\theta_2\wedge d\theta_4\right]\in H^{3}(\mathcal{F}),
		\]
		it suffices to note that the latter class is non-trivial. This is indeed the case, for if it was trivial then the restriction of that $3$-form to each level set $\mathbb{T}^{4}$ of $\psi$ would be exact, which would imply that $f'(\psi)\equiv 0$. So the infinitesimal deformation $\xi$ of $\Pi$ is obstructed. However, the infinitesimal deformation $\phi(\xi)=d\theta_4\otimes\partial_{\psi}$ of $\mathcal{F}$ is unobstructed: it is tangent to the path $\Phi_{t}=td\theta_{4}\otimes\partial_{\psi}\in\Omega^{1}(\mathcal{F};G)$, which indeed deforms $\mathcal{F}$ because $\text{Graph}(\Phi_{t})=\text{Span}\{\partial_{\theta_1},\partial_{\theta_2},\partial_{\theta_3},\partial_{\theta_4}+t\partial_\psi\}$ is involutive.
	\end{itemize}
\end{example}

In the example above, the foliation $\calF$ still has obstructed infinitesimal deformations, because of Lemma~\ref{lem:cosymplectic} (1). We now alter this example by changing the fiber suitably, which yields an example of a regular Poisson structure $\Pi$ with characteristic foliation $\mathcal{F}$ such that \emph{all} infinitesimal deformations of $\mathcal{F}$ are unobstructed, whereas the deformation problem of $\Pi$ is obstructed. Keeping in mind Lemma~\ref{lem:cosymplectic}, we want the fiber to possess the following properties.

\begin{lemma}\label{mfld}
There exists a compact symplectic manifold $(N,\omega)$ that satisfies the following:
\begin{itemize}
    \item $H^{1}(N)$ is one-dimensional, with generator $[\alpha]$.
    \item There exists a closed $2$-form $\beta\in\Omega^{2}(N)$ such that $\alpha\wedge\beta$ is not exact. That is, the map  $\wedge:H^{1}(N)\times H^{2}(N)\rightarrow H^{3}(N)$ is not trivial.
\end{itemize}
\end{lemma}
\begin{proof}
A theorem by Gompf~\cite{gompf} ensures that every finitely presented group arises as the fundamental group of a compact symplectic $4$-manifold. In particular, there exists a compact symplectic manifold $(M,\omega_M)$ such that $\pi_1(M)=\mathbb{Z}$ and therefore $H^{1}(M)\cong \text{Hom}\left(\pi_1(M),\mathbb{R}\right)\cong\mathbb{R}$. Fix a generator $[\alpha]\in H^{1}(M)$.

Define $N:=M\times S^{2}$ and endow it with the product symplectic form $p_{1}^{*}\omega_M+p_{2}^{*}\omega_{S^{2}}$. By the Künneth formula, we have $H^{1}(M\times S^{2})=\mathbb{R}[p_{1}^{*}\alpha]$ and the closed $2$-form $p_{2}^{*}\omega_{S^{2}}$ is such that $p_{1}^{*}\alpha\wedge p_{2}^{*}\omega_{S^{2}}$ is not exact.
\end{proof}

\begin{example}[{An obstructed deformation with unobstructed foliation}]\label{ex:obs-unobs}
Let $(N,\omega)$ be a compact symplectic manifold as in Lemma~\ref{mfld}. We consider the manifold $S^{1}\times N$ with cosymplectic structure $(d\psi,\omega)$, where $\psi$ denotes the coordinate on $S^{1}$. Correspondingly, there is a Poisson structure $\Pi$ on $S^{1}\times N$ whose symplectic leaves are $\left(\{\psi\}\times N,\omega\right)$ and for which $\partial_{\psi}$ is a transverse Poisson vector field. As a complement to the characteristic distribution $T\mathcal{F}$, we take $G:=\text{Span}(\partial_{\psi})$. 

Lemma~\ref{lem:cosymplectic} (2) implies that the deformation problem of the underlying foliation $\mathcal{F}$ is unobstructed. But the deformation problem of the corank-one Poisson structure $\Pi$ is obstructed, by Lemma~\ref{lem:cosymplectic} (3).
\end{example}

Implicitly, we made sure that the dimension of the fiber $N$ of $S^{1}\times N$ is at least $4$, by imposing the second condition in Lemma~\ref{mfld}. This was indeed necessary in order to obtain an example with the desired properties, because {of} Proposition~\ref{dim2}.

\subsubsection{\underline{\smash{Stability of symplectic foliations}}}  \label{subsubsec:stab}

{
Various kinds of stability questions in Poisson geometry have been addressed in the literature, for instance, in~\cite{CrFeStab} and~\cite{DufourWade}. Here we mention the stability question for symplectic foliations: fix a regular Poisson structure $(M,\Pi)$, i.e. a symplectic foliation $(\mathcal{F},\omega)$. The stability question reads: \emph{does any foliation $\mathcal{F}'$ nearby $\mathcal{F}$ carry a leaf-wise symplectic structure?}
From a geometric perspective it is clear that stability holds in each of the following cases:
\begin{itemize}
\item if $\omega$  extends to a closed 2-form $\Omega$ on $M$, because then one can pull back $\Omega$ to the leaves of  $\mathcal{F}'$
\item if $rank(\mathcal{F})=2$, {since one} can proceed as above with any extension of $\omega$.
\end{itemize}
}
 
{ 
The smooth 1-parameter version of the question reads: 
\emph{for any smooth family of foliations $\{\mathcal{F}_t\}$  with $\mathcal{F}_0=\mathcal{F}$, does there exist a smooth family of leaf-wise symplectic structures $\omega_t$ on  $\mathcal{F}_t$ with $\omega_0=\omega$, for $t$ lying in some open interval around zero?}
}

{
If the answer is positive, then a surjectivity statement for unobstructed infinitesimal deformations holds, namely:
 the map   $\phi:\frakX_\good^\bullet(M)[2]\to\Omega^\bullet(\calF;N\calF)[1]$
of equation~\eqref{eq:morphism}
induces a map of cocycles  
\begin{equation}\label{eq:infstab}
  Z^0(\frakX_\good^\bullet(M)[2])\to Z^0(\Omega^\bullet(\calF;N\calF)[1])
\end{equation} 
with the property that every \emph{unobstructed} cocyle $\alpha$ on the r.h.s. has an \emph{unobstructed} preimage.
This surjectivity statement is related to the unobstructedness of infinitesimal deformations  we addressed in  
\S\ref{subsubsec:Unobstructedness results}, but  
it features a more flexible behaviour. Indeed,
given an unobstructed cocycle on the codomain of the map~\eqref{eq:infstab}, the  surjectivity statement is the \emph{existence} of an unobstructed preimage, while in \S\ref{subsubsec:Unobstructedness results} we provided conditions under which \emph{all} preimages are unobstructed.
 }

\appendix


\section{Courant Algebroids and Dirac Structures}
\label{app:Courant_algebroids}

We review some background material concerning Courant algebroids and Dirac structures. {The latter were introduced by Courant in~\cite{Courant1990Dirac}.}

\begin{definition}
A \emph{Courant algebroid} consists of a vector bundle $E\to M$ equipped with a non-degenerate symmetric pairing $\ldab-,-\rdab:E\otimes E\to\bbR$, a bilinear bracket $\ldsb-,-\rdsb:\Gamma(E)\times\Gamma(E)\to\Gamma(E)$ called the \emph{Dorfman bracket}, and a VB morphism $\rho:E\to M$ over $\id_M$ called the \emph{anchor map}, satisfying the compatibility conditions:
\begin{align*}
	&\ldsb e,\ldsb e^\prime,e^{\prime\prime}\rdsb\rdsb=\ldsb\ldsb e,e^\prime\rdsb,e^{\prime\prime}\rdsb+\ldsb e^\prime,\ldsb e,e^{\prime\prime}\rdsb\rdsb,\\
	&\ldab\ldsb e,e^\prime\rdsb,e^{\prime\prime}\rdab=\ldab e,\ldsb e^\prime,e^{\prime\prime}\rdsb\rdab,\\
	&\rho(e)\ldab e^\prime,e^\prime\rdab=2\ldab\ldsb e,e^\prime\rdsb,e^\prime\rdab,
\end{align*}
for any $e,e^\prime,e^{\prime\prime}\in\Gamma(E)$.
Consequently, the Dorfman bracket and the anchor are related by the Leibniz rule
\begin{equation*}
\ldsb e,fe^\prime\rdsb=(\rho(e)f)e^\prime+f\ldsb e,e^\prime\rdsb,
\end{equation*}
for any $e,e^\prime\in\Gamma(E)$ and $f\in C^\infty(M)$.
See, e.g.,~\cite{Roytenberg1999PhDthesis} for alternative descriptions of Courant algebroids.
\end{definition}

\begin{example}
The \emph{generalized tangent bundle} $\bbT M:=TM\oplus T^\ast M$ is the prototypical example of a Courant algebroid. It is equipped with the pairing $\ldab-,-\rdab$, Dorfman bracket $\ldsb-,-\rdsb$ and anchor $\rho$, which are defined on sections $X+\alpha,Y+\beta\in\Gamma(\bbT M)$ as follows:
\begin{equation*}
\ldab X+\alpha,Y+\beta\rdab=\alpha(Y)+\beta(X),\quad\rho(X+\alpha)=X,\quad\ldsb X+\alpha,Y+\beta\rdsb=[X,Y]+\calL_X\beta-\iota_Y\rmd\alpha.
\end{equation*}
\end{example}

\begin{definition}
Given a Courant algebroid $E$, a subbundle $L\subset E$ is called an \emph{almost Dirac structure} if $L=L^\perp$, where $L^\perp\subset E$ denotes the orthogonal of $L$ w.r.t. the pairing $\ldab-,-\rdab$.
A \emph{Dirac structure} is an almost Dirac structure $L\subset E$ that is additionally involutive w.r.t. the Dorfman bracket $\ldsb-,-\rdsb$, i.e. $\ldsb\Gamma(L),\Gamma(L)\rdsb\subset\Gamma(L)$.
\end{definition}

\begin{remark}
\label{rem:Courant_tensor}
For each almost Dirac structure $L\subset E$, its \emph{Courant tensor} $\Upsilon_L\in\Gamma(\wedge^3 L^{*})$ is defined by
		\begin{equation*}
		\Upsilon_L(\xi_1,\xi_2,\xi_3)=\ldab\xi_1,\ldsb\xi_2,\xi_3\rdsb\rdab,
		\end{equation*}
	for all $\xi_1,\xi_2,\xi_3\in\Gamma(L)$. It is easy to see that $L$ is Dirac if and only if $\Upsilon_L=0$.
\end{remark}

\begin{example}
We show some special classes of (almost) Dirac structures for the generalized tangent bundle.
\begin{itemize}
    \item $2$-forms on $M$ correspond with almost Dirac structures that are transverse to $T^{*}M$, via the identification
    \begin{equation*}
\begin{aligned}
\Omega^2(M)&\overset{\sim}{\longrightarrow}\{\calL\subset\bbT M\ \text{almost Dirac structure}\mid\calL\pitchfork T^{*}M\},\\
\omega&\longmapsto\gr(\omega):=\{X+\iota_X\omega \mid X\in TM\}.
\end{aligned}
\end{equation*}
Moreover, $\omega\in\Omega^{2}(M)$ is closed if and only if $\gr(\omega)\subset\bbT M$ is Dirac.
\item Bivector fields\footnote{{If $Z$ is a non-degenerate bivector field and $\omega$  the corresponding non-degenerate 2-form,
determined by $Z^{\sharp}=-(\omega^{\flat})^{-1}$, then the graphs satisfy $\gr(Z)=\gr(-\omega)$.}}  on $M$ correspond with almost Dirac structures that are transverse to $TM$ as follows:
\begin{equation*}
\begin{aligned}
\frakX^2(M)&\overset{\sim}{\longrightarrow}\{\calL\subset\bbT M\ \text{almost Dirac structure}\mid\calL\pitchfork TM\},\\
Z&\longmapsto\gr(Z):=\{\iota_\alpha Z+\alpha\mid\alpha\in T^\ast M\}.
\end{aligned}
\end{equation*}
Moreover, $Z\in\frakX^2(M)$ is Poisson if and only if $\gr(Z)\subset\bbT M$ is Dirac.
\item Distributions on $M$ correspond with almost Dirac structures $\calL\subset\bbT M$ s.~t.~$\calL=\pr_{TM}\calL{{}\oplus{}}{\pr_{T^\ast M}\calL}$,
\begin{equation*}
\begin{aligned}
\{\text{distributions on $M$}\}&\overset{\sim}{\longrightarrow}\{\calL\subset\bbT M\ \text{almost Dirac structure}\mid\calL=\pr_{TM}\calL{{}\oplus{}}{\pr_{T^\ast M}\calL}\},\\
D&\longmapsto D\oplus D^{0}.
\end{aligned}
\end{equation*}
Moreover, $D$ is involutive if and only if $D\oplus D^{0}\subset\bbT M$ is Dirac.
\end{itemize}
\end{example}

\begin{remark}
	\label{rem:almost_Dirac_structure}
	Let $L\subset E$ be an almost Dirac structure, and assume that $R$ is a complementary subbundle of $E$. Using the projection $\pr_{L}$ with kernel $R$, one can restrict the Dorfman bracket $\ldsb-,-\rdsb$ to an almost Lie bracket $\pr_{L}\circ\ldsb-,-\rdsb$ on $\Gamma(L)$. Along with the restriction of the anchor map $\rho$ to $L$, one obtains an almost Lie algebroid structure $(\rho|_{L},\pr_{L}\circ\ldsb-,-\rdsb)$ on $L$.
The latter allows us to further construct:
\begin{itemize}
	\item a differential $\rmd_L$ on $\Omega^\bullet(L)\!:=\!\Gamma(\wedge^\bullet L^\ast)$, i.e.~the degree $1$ graded algebra derivation $\rmd_L$ of $\Omega^\bullet(L)$ such that
	\begin{align*}
	&\rmd_Lf(\xi)=\rho(\xi)f,\qquad \rmd_L\eta(\xi,\zeta)=\calL_{\rho(\xi)}\iota_\zeta\eta-\calL_{\rho(\zeta)}\iota_\xi\eta-\iota_{\pr_{L}\ldsb\xi,\zeta\rdsb}\eta,
	\end{align*}
	for all $f\in C^\infty(M)=\Omega^0(L)$, $\eta\in\Gamma(L^\ast)=\Omega^1(L)$, and $\xi,\zeta\in\Gamma(L)$.
	\item an almost Gerstenhaber bracket $[-,-]_L$ on $\Gamma(\wedge^\bullet L)$, i.e.~the degree $0$ graded almost Lie bracket on $\Gamma(\wedge^\bullet L)[1]$ determined by
	\begin{equation*}
	[\xi,f]_L=\rho(\xi)f,\quad [\xi,\zeta]_L=\pr_L\ldsb\xi,\zeta\rdsb,\quad[P,Q\wedge R]_L=[P,Q]_L\wedge R+(-)^{|Q|(|P|-1)}Q\wedge[P,R]_L,
	\end{equation*}
	for all $f\in C^\infty(M)$, $\xi,\zeta\in\Gamma(L)$ and homogeneous $P,Q,R\in\Gamma(\wedge^\bullet L)$.
\end{itemize}
If $L$ is Dirac, then it becomes a Lie algebroid (independent of $R$). In that case, $\rmd_L$ is an honest differential, i.e. $\rmd_L^2=0$, and $[-,-]_L$ satisfies the graded Jacobi identity.
\end{remark}

\section{Deformation Theory of Dirac Structures}
\label{app:Deformations_Dirac_Structures}

We will briefly review the classical constructions and results from deformation theory of Dirac structures that are relevant for the main goal of this paper,
{following~\cite{liu1995manin} and~\cite{FZgeo}.}
{The relevant algebraic structure here is the one of
$L_{\infty}$-algebras~\cite{LadaMarkl}, which contain differential graded Lie algebras (dgLa's) as special cases. We will mostly work with the equivalent notion of $L_{\infty}[1]$-algebra (see for instance~\cite[\S2]{fiorenza2007cones}), 
in which all the structure maps are  graded symmetric and of degree $1$.}

Let $E\to M$ be a Courant algebroid and $A\subset E$ be a Dirac structure. {Upon choosing an almost Dirac structure complementary to $A$, the graded vector space $\Omega^{\bullet}(A)[2]$ becomes endowed with $L_{\infty}[1]$-algebra structure, as stated in Lemma~\ref{lem:2.6first} below.}

We first introduce some notation.
Let $A\to M$ be a vector bundle.
For any $\alpha\in\Omega^k(A):=\Gamma(\wedge^{k}A^*)$, we define the VB morphism $\alpha^\sharp:A\to\wedge^{k-1}A^\ast$ by setting
\begin{equation*}
(\alpha^\sharp\xi_1)(\xi_2,\ldots\xi_k)=\alpha(\xi_1,\ldots,\xi_k),
\end{equation*}
for all $x\in M$ and $\xi_1,\ldots,\xi_k\in A_x$.
Clearly, if $k=0$, then $\alpha^\sharp=0$.
Given $\alpha_1\in\Omega^{k_1}(A),\ldots,\alpha_n\in\Omega^{k_n}(A)$, one can define the VB morphism $\alpha_1^\sharp\wedge\ldots\wedge \alpha_n^\sharp:\wedge^n A\to\wedge^{k_1+\ldots+k_n-n}A^\ast$ so that
\begin{equation*}
\label{eq:sharp:property1}
(\alpha_1^\sharp\wedge\ldots\wedge \alpha_n^\sharp)(\xi_1\wedge\ldots\wedge\xi_n)=\sum_{\sigma\in S_n}(-)^\sigma \alpha_1^\sharp(\xi_{\sigma(1)})\wedge\ldots\wedge \alpha_n^\sharp(\xi_{\sigma(n)}),
\end{equation*}
for all $x\in M$ and $\xi_1,\ldots,\xi_n\in A_x$.
Further, for any $\alpha_1\in\Omega^{k_1}(A)$ and $\alpha_2\in\Omega^{k_2}(A)$, we define the VB morphisms $\alpha_1^\sharp\wedge\alpha_2,\ \alpha_1\wedge \alpha_2^\sharp:A\to\wedge^{k_1+k_2-1}A^\ast$ by setting
\begin{equation*}
(\alpha_1^\sharp\wedge\alpha_2)\xi=(\alpha_1^\sharp\xi)\wedge\alpha_2\quad\text{and}\quad(\alpha_1\wedge\alpha_2^\sharp)\xi=\alpha_1\wedge(\alpha_2^\sharp\xi),
\end{equation*}
for any $\xi\in A$.
Consequently, for all homogeneous $\alpha\in\Omega^\bullet(A)$ and $\beta\in\Omega^\bullet(A)$, the following identity holds:
\begin{equation*}
\label{eq:sharp:property2}
(\alpha\wedge\beta)^\sharp=\alpha^\sharp\wedge\beta+(-)^{|\alpha|}\alpha\wedge\beta^\sharp.
\end{equation*}

Using these morphisms along with the notation introduced in Remarks~\ref{rem:Courant_tensor} and~\ref{rem:almost_Dirac_structure}, one can describe the induced $L_{\infty}[1]$-algebra structure on $\Omega^{\bullet}(A)[2]$ as follows. {The following is an adaptation of~\cite[Lemma 2.6]{FZgeo} (see also~\cite[Proposition 3.11]{SZDirac}), in which we take the negative\footnote{An advantage of this convention is that no minus signs appear in Lemma~\ref{lem:2.6second} below.}
 of the binary bracket appearing there.}

\begin{lemma}	\label{lem:2.6first}
	Given a Courant algebroid $E\to M$, let $E=A\oplus B$ be a splitting where $A$ is a Dirac structure and $B$ is a complementary almost Dirac structure. Then
	the graded vector space $\Omega^\bullet(A)[2]$ has an induced $L_\infty[1]$-algebra structure $\{\mu_k^B\}$, whose only non-trivial multibrackets $\mu_1^B,\ \mu_2^B,\ \mu_3^B$ are defined as follows:
	\begin{itemize}
		\item $\mu_1^B$ coincides with the de Rham differential of Lie algebroid $A\Rightarrow M$, i.e.~for all $\alpha\in\Omega^\bullet(A)$:
		\begin{equation}
		\label{eq:lem:L_infty-algebra_Dirac:unary_bracket}
		\mu_1^B(\alpha{[2]})=(\rmd_A\alpha){[2]}.
		\end{equation}
		\item $\mu_2^B$ acts as follows, for all homogeneous $\alpha,\beta\in\Omega^\bullet(A)$:
	
		\begin{equation}
		\label{eq:lem:L_infty-algebra_Dirac:binary_bracket}
		\mu_2^B(\alpha{[2]},\beta{[2]})=(-1)^{|\alpha|}[\alpha,\beta]_B{[2]}.
		\end{equation}
		\item $\mu_3^B$ acts as follows, for all homogeneous $\alpha,\beta,\gamma\in\Omega^\bullet(A)$:
		\begin{equation}
		\label{eq:lem:L_infty-algebra_Dirac:ternary_bracket}
		\mu_3^B(\alpha{[2]}, \beta{[2]}, \gamma{[2]})=(-1)^{|\beta|} (\alpha^\sharp\wedge \beta^\sharp \wedge \gamma^\sharp) \Upsilon_B{[2]}.
		\end{equation}
	\end{itemize}
	In the RHS of equations~\eqref{eq:lem:L_infty-algebra_Dirac:binary_bracket} and~\eqref{eq:lem:L_infty-algebra_Dirac:ternary_bracket}, we use the identification $B\overset{\simeq}{\longrightarrow} A^\ast,\ u\longmapsto\ldab u,-\rdab|_A$.
\end{lemma}

\begin{remark}
	\label{rem:decalage}
	We list some remarks concerning the $L_{\infty}[1]$-algebra $\big(\Omega^\bullet(A)[2],\mu_1^{B},\mu_2^{B},\mu_3^{B}\big)$ introduced above.
	\begin{enumerate}
		\item  If $V$ is a graded vector space, then $L_\infty$-algebra structures $\{m_k\}$ on $V$ correspond bijectively to $L_\infty[1]$-algebra structures $\{\mu_k\}$ on $V[1]$ via the following relation (we follow the convention of~\cite[Remark 1.1]{fiorenza2007cones}, {up to a global minus sign}): 
		\begin{equation*}
		m_k(v_1,\ldots,v_k)={-}(-)^k(-)^{\sum_i(k-i)|v_i|}\mu_k(v_1[1],\ldots,v_k[1]),
		\end{equation*}
		for all $k\in\bbN$ and homogeneous $v_1,\ldots,v_k\in V$.
		In particular, the $L_\infty[1]$-algebra structure $\{\mu_k^B\}$ on $\Omega^\bullet(A)[2]$ described in Lemma~\ref{lem:2.6first} corresponds to an $L_\infty$-algebra structure $\{m_k^B\}$ on $\Omega^\bullet(A)[1]$,  
		whose only non-trivial multibrackets $m_1^B,m_2^B,m_3^B$ are defined as follows:
		\begin{align*}
		&m_1^B(a)=\mu_1^B(a[1]),\\
		&m_2^B(a,b)={-}(-)^{|a|}\mu_2^B(a[1],b[1]),\\
		&m_3^B(a,b,c)=(-)^{|b|}\mu_3^B(a[1],b[1],c[1]),
		\end{align*}
		for homogeneous elements ${a,b,c}\in\Omega^\bullet(A)[1]$. 
		
		\item For any almost Dirac structure $B\subset E$ complementary to the Dirac structure $A\subset E$, the $L_\infty[1]$-algebra structure $\{\mu_k^B\}$ reduces to a dgL[1]a, i.e.~$\mu_k^B=0$ for all $k\geq 3$, if and only if $B$ is Dirac. In this case, the $L_\infty$-algebra $(\Omega^\bullet(A)[1],\{m_k^B\})$ reduces to the dgLa {$(\Omega^\bullet(A)[1], d_A,[-,-]_{A^*})$} used in~\cite{liu1995manin}.
		\item The $L_\infty[1]$-algebra $(\Omega^\bullet(A)[2],\{\mu_k^B\})$ (resp.~$L_\infty$-algebra $(\Omega^\bullet(A)[1],\{m_k^B\})$) is actually a \emph{$G_\infty[1]$-algebra} (resp.~\emph{$G_\infty$-algebra}) since its multibrackets are compatible with the graded algebra structure of $\Omega^\bullet(A)$. That is, the following graded Leibniz rule holds:
		\begin{equation}
		\mu_k^{B}(\alpha_1,\ldots,\alpha_{k-1},\alpha_k\wedge\beta)=\mu_k^{B}(\alpha_1,\ldots,\alpha_k)\wedge\beta+(-)^{|\alpha_k|(1+|\alpha_1|+\ldots+|\alpha_{k-1}|)}\alpha_k\wedge\mu_k^{B}(\alpha_1,\ldots,\alpha_{k-1},\beta),
		\end{equation}
		for all homogeneous $\alpha_1,\ldots,\alpha_k,\beta\in\Omega^\bullet(A)[2]$ (a similar Leibniz rule holds for the graded skew-symmetric multibrackets $m_k^B$).
		For the general definition of \emph{$G_\infty$-algebras}, also called \emph{homotopy Gerstenhaber algebras}
	{or \emph{$P_\infty$-algebras,}}
		 we refer the reader to~\cite{gerstenhaber1995homotopy,tamarkin1998another}.
	\end{enumerate}
\end{remark}

Given a Dirac structure $A\subset E$ and a complementary almost Dirac structure $B\subset E$, we now turn to the geometric information encoded by the Maurer-Cartan elements of the associated $L_\infty[1]$-algebra.
 \begin{definition}
	\label{def:MC_equation:L_infty_algebra:Dirac}
	A \emph{Maurer--Cartan (MC) element} of the $L_\infty[1]$-algebra $(\Omega^\bullet(A)[2],\{\mu_k^B\})$ is a degree $0$ element $\eta$ of $\Omega^\bullet(A)[2]$, i.e.~a $2$-form $\eta\in\Omega^2(A)$, satisfying the \emph{MC equation}
	\begin{equation*}
	\mu_1^B\eta+\frac{1}{2}\mu_2^B(\eta,\eta)+\frac{1}{6}\mu_3^B(\eta,\eta,\eta)=0.
	\end{equation*}
\end{definition}

The pairing $\ldab-,-\rdab$ induces VB isomorphisms $B\overset{\sim}{\longrightarrow}A^\ast$ and $E=A\oplus B\overset{\sim}{\longrightarrow}A\oplus A^\ast$. 
Under these identifications, it is easy to see that the relation 
\begin{equation*}
L=\gr(\eta)=\{\xi{+}\iota_\xi\eta\mid\xi\in A\}\subset A\oplus A^\ast\simeq A\oplus B=E.
\end{equation*}
establishes a one-to-one correspondence between degree $0$ elements $\eta$ of $\Omega^\bullet(A)[2]$, i.e.~$2$-forms $\eta\in\Omega^2(A)$, and almost Dirac structures $L\subset E$ that are \emph{close to $A$ w.r.t. $B$}, in the sense that they are still transverse to $B$.

As proven in~\cite[Theorem 6.1]{liu1995manin} (when the complement $B$ is Dirac) and~\cite[Lemma 2.6]{FZgeo} (in the general case), the MC elements of the $L_\infty[1]$-algebra $(\Omega^\bullet(A)[2],\{\mu_k^B\})$ are exactly those $2$-forms $\eta\in\Omega^2(A)$ for which the corresponding almost Dirac structure $L=\gr(\eta)$  is involutive, i.e.~Dirac. 
\begin{lemma}[{\cite[Theorem 6.1]{liu1995manin} and~\cite[Lemma 2.6]{FZgeo}}]
	\label{lem:2.6second}
	Let $E$ be a Courant algebroid and $A\subset E$ a Dirac structure. Fix an almost Dirac structure $B\subset E$ complementary to $A$.
	Then a canonical one-to-one correspondence between
	\begin{enumerate}
		\item MC elements $\eta$ of the associated $L_\infty[1]$-algebra $(\Omega^\bullet(A)[2],\{\mu^B_k\})$,
		\item Dirac structures $L\subset E$ that are transverse to $B$,
	\end{enumerate}
	is established by the following relation: 
	\begin{equation*}
	L=\gr(\eta):=\{\xi{+}\iota_\xi\eta\mid \xi\in A\}\subset A\oplus A^\ast\simeq A\oplus B=E.
	\end{equation*}
\end{lemma}

This means that, for each complementary almost Dirac structure $B\subset E$, the associated $L_\infty[1]$-algebra $(\Omega^\bullet(A)[2],\{\mu_k^B\})$ controls the ``corresponding'' deformation problem of the given Dirac structure $A\subset E$.

\section{Proof of Proposition \ref{prop:strict_L_infty-algebra_morphism}}
\label{app:proofstrict}

Proposition \ref{prop:strict_L_infty-algebra_morphism} states that the map  
 	\begin{equation*}
	\begin{tikzcd}
	(\frakX_\good^\bullet(M)[2],\{\frakl^G_k\})\arrow[rr, "\phi"]&&(\Omega^\bullet(\calF;N\calF)[1],\{\frakv_k\})
	\end{tikzcd}
	\end{equation*}
  is a  strict morphism of $L_\infty[1]$-algebras.
  We first prove this proposition by a direct computation. Then in Proposition  \ref{rem:conceptualclosed} 
  we provide a conceptual argument under the assumption that the 2-form
$\gamma$ is closed.

\begin{proof}[Proof of Prop.~\ref{prop:strict_L_infty-algebra_morphism}]
Recall that both $(\frakX_\good^\bullet(M)[2],\{\frakl^G_k\})$ and $(\Omega^\bullet(\calF;N\calF)[1],\{\frakv_k\})$ are in fact $LR_{\infty}[1]$-algebras over $\mathfrak{X}^{\bullet}(\calF)$ and $\Omega^{\bullet}(\calF)$, with anchor maps $\{\rho_k\}$ and $\{\frakn_k\}$, respectively (see Proposition \ref{prop:good_L_infty_algebra} and Lemma
	\ref{lem:G_infty_algebra:foliation}).
Moreover, these anchor maps are derivations in their last entries, and the module morphism $\phi$ is defined along the algebra morphism $\underline{\smash\phi}$. Therefore, we can conclude that $\phi$ is a strict $L_{\infty}[1]$-morphism if we check that the following equalities hold:
\begin{align}
&\phi\big(\frakl^G_k(P_1,\ldots,P_k)\big)=\frakv_k\big(\phi(P_1),\ldots,\phi(P_k)\big),\label{toshow1}\\
&\underline{\smash\phi}\big(\rho_k(P_1,\ldots,P_{k-1}|X)\big)=\frakn_k\big(\phi(P_1),\ldots,\phi(P_{k-1})|\underline{\smash\phi}(X)\big),\label{toshow2}
\end{align}
where $P_1,\ldots,P_k\in\frakX^0_\good(M)\oplus\frakX^1_\good(M)=C^\infty(M)\oplus\Gamma(T\calF)\oplus\Gamma(G)$ and $X\in\frakX^0(\calF)\oplus\frakX^1(\calF)=C^\infty(M)\oplus\Gamma(T\calF).$

	 We now carry out three computations to check that $\phi$ preserves the unary, binary, and ternary brackets, i.e. that equation \eqref{toshow1} is satisfied.
	 
{i)}	Let us start proving that $\phi$ preserves the unary bracket.
	In view of the above, we just have to check that 
	\begin{equation*}
	\label{eq:proof:prop:strict_L_infty-algebra_morphism:unary_brackets}
	\phi(\rmd_\Pi(f))=\rmd_\nabla(\phi(f)),\qquad\phi(\rmd_\Pi(X))=\rmd_\nabla(\phi(X)),\qquad\phi(\rmd_\Pi(Y))=\rmd_\nabla(\phi(Y)),
	\end{equation*}
	for all $f\in C^\infty(M)$, $X\in\Gamma(T\calF)$ and $Y\in\Gamma(G)$.
	Actually, the first two identities are obvious: both the LHS and RHS vanish because $\phi$ kills $\frakX^\bullet(\calF)$. To check the third identity, we compute for $\alpha\in\Gamma(T^{*}\calF)$ and $\beta\in\Gamma(G^{*})$:
	\begin{align*}
	\big\langle \phi(d_{\Pi}Y)(\Pi^{\sharp}\alpha),\beta\big\rangle&={-}d_{\Pi}Y\big(\omega^{\flat}(\Pi^{\sharp}\alpha),\beta\big)\\
	&=d_{\Pi}Y(\alpha,\beta)\\
	&={-}(\calL_{Y}\Pi)(\alpha,\beta)\\
	&={-}\calL_{Y}\big(\Pi(\alpha,\beta)\big){+}\Pi\big(\calL_{Y}\alpha,\beta\big){+}\Pi\big(\alpha,\calL_{Y}\beta\big)\\
	&=\Pi\big(\alpha,\calL_{Y}\beta\big),
	\end{align*}
	using in the last equality that $\beta$ annihilates $T\mathcal{F}=\im\Pi^{\sharp}$. On the other hand, we have
			\begin{align*}
	\big\langle\rmd_\nabla\phi(Y)(\Pi^{\sharp}\alpha),\beta \big\rangle &=\big\langle\nabla_{\Pi^{\sharp}\alpha}\phi(Y),\beta \big\rangle\\
	&=\big\langle[\Pi^{\sharp}\alpha,Y],\beta \big\rangle\\
	&=-\big\langle\calL_{Y}\Pi^{\sharp}\alpha,\beta \big\rangle\\
	&=\big\langle\Pi^{\sharp}\alpha,\calL_{Y}\beta \big\rangle-\calL_{Y}\langle \Pi^{\sharp}\alpha,\beta\rangle\\
	&=\Pi(\alpha,\calL_{Y}\beta).
	\end{align*}

	{ii)} Let us now prove that $\phi$ preserves the binary bracket.
	By degree reasons, we only have to show that
	\begin{gather*}
	\label{eq:proof:prop:strict_L_infty-algebra_morphism:binary_brackets}
	\phi(\frakl_2^{G}(f,X))=\frakv_2(\phi(f),\phi(X)),\qquad\phi(\frakl_2^{G}(f,Y))=\frakv_2(\phi(f),\phi(Y)),\qquad\\
	\phi(\frakl_2^{G}(X_1,X_2))=\frakv_2(\phi(X_1),\phi(X_2)),\quad\phi(\frakl_2^{G}(X,Y))=\frakv_2(\phi(X),\phi(Y)),\quad\phi(\frakl_2^{G}(Y_1,Y_2))=\frakv_2(\phi(Y_1),\phi(Y_2)),
	\end{gather*}
	for all $f\in C^\infty(M)$, $X,X_1,X_2\in\Gamma(T\calF)$ and $Y,Y_1,Y_2\in\Gamma(G)$.
	Except the last one, all the identities above are obvious since both the LHS and RHS vanish.
	The last one can be proved as follows:
	\begin{equation*}
	\phi(\frakl_2^{G}(Y_1,Y_2))=\phi({-}[Y_1,Y_2]_{\gamma})=\phi({-}[Y_1,Y_2])={-}\pr_{G}[Y_1,Y_2]=\frakv_2(Y_1,Y_2)=\frakv_2(\phi(Y_1),\phi(Y_2)).
	\end{equation*}
	 	
{iii)}	We now prove that $\phi$ preserves the ternary bracket.
	By degree reasons, we only have to show that 
	\begin{gather*}
	\label{eq:proof:prop:strict_L_infty-algebra_morphism:ternary_brackets}
	\phi(\frakl_3^{G}(V_1,V_2,V_3))=\frakv_3(\phi(V_1),\phi(V_2),\phi(V_3)),
	\end{gather*}
	for all $V_1,V_2,V_3\in\Gamma(T\calF\oplus G)$.
	Clearly, this identity holds because both sides are zero. This proves~\eqref{toshow1}.
	
	We now proceed by checking compatibility with the anchor maps, as required by~\eqref{toshow2}. 
	
	i) For the first anchor, we need to check that
	\begin{gather*}
	 \underline{\smash\phi}(\rho_1(f))=\frakn_1(\underline{\smash\phi}(f)),\qquad \underline{\smash\phi}(\rho_1(X))=\frakn_1(\underline{\smash\phi}(X)),
	\end{gather*}
	for all $f\in C^{\infty}(M), X\in\Gamma(T\calF)$. The first equality holds because  
	\[
	\underline{\smash\phi}(\rho_1(f))=\underline{\smash\phi}(d_{\Pi}f)=-\underline{\smash\phi}(\Pi^{\sharp}df)={-}\omega^{\flat}(\Pi^{\sharp}df)=d_{\calF}f=\frakn_1(f)=\frakn_1(\underline{\smash\phi}(f)).
	\]
	The second equality holds because one one hand
	\[
	\underline{\smash\phi}(\rho_1(X))=\underline{\smash\phi}(d_{\Pi}X)=-\underline{\smash\phi}\big(d_{\Pi}(\Pi^{\sharp}(\omega^{\flat}X))\big)=\underline{\smash\phi}\big(\wedge^{2}\Pi^{\sharp}(d_{\calF}(\omega^{\flat}X))\big)=\wedge^{2}\omega^{\flat}\big(\wedge^{2}\Pi^{\sharp}(d_{\calF}(\omega^{\flat}X))\big)=d_{\calF}(\omega^{\flat}X)
	\]
	whereas on the other hand
	\[
	\frakn_1(\underline{\smash\phi}(X))=\frakn_1(\omega^{\flat}X)=d_{\calF}(\omega^{\flat}X).
	\]
{Above we made use of the fact that 	$\underline{\smash\phi}(X)=\omega^{\flat}X$ for all $X\in\Gamma(T\calF)$.}
	
	ii) For compatibility with the second anchor, the only non-trivial equalities to check are 
	\begin{gather*}
	\underline{\smash\phi}(\rho_2(Y|f))=\frakn_2(\phi(Y)|\underline{\smash\phi}(f)),\qquad \underline{\smash\phi}(\rho_2(Y|X))=\frakn_2(\phi(Y)|\underline{\smash\phi}(X))
	\end{gather*}
	for $f\in C^{\infty}(M), X\in\Gamma(T\calF), Y\in\Gamma(G)$. Here the first equality holds since
	\[
	\underline{\smash\phi}(\rho_2(Y|f))=\underline{\smash\phi}({-}[Y,f]_{\gamma})={-}[Y,f]_{\gamma}={-}Y(f)=\frakn_2(Y,f)=\frakn_2(\phi(Y)|\underline{\smash\phi}(f)),
	\]
	and also the second equality is true because
	\begin{align*}
	&\underline{\smash\phi}(\rho_2(Y|X))=\underline{\smash\phi}({-}[Y,X]_{\gamma})=\underline{\smash\phi}\big(\Pi^{\sharp}(\calL_{Y}\gamma^{\flat}(X))\big)=\omega^{\flat}\big(\Pi^{\sharp}(\calL_{Y}\gamma^{\flat}(X))\big)=-\pr_{T^{*}\calF}\calL_{Y}\gamma^{\flat}(X),\\
	&\frakn_2(\phi(Y)|\underline{\smash\phi}(X))=\frakn_2(Y,\omega^{\flat}(X))=-\pr_{T^{*}\calF}\calL_{Y}\gamma^{\flat}(X).
	\end{align*}
	
	iii) At last, to check compatibility with the third anchor, the only non-trivial requirement is 
	\[
	\underline{\smash\phi}\big(\rho_3(Y_1,Y_2|X)\big)=\frakn_3\big(\phi(Y_1),\phi(Y_2)|\underline{\smash\phi}(X)\big)
	\]
	for $Y_1,Y_2\in\Gamma(G)$ and $X\in\Gamma(T\calF)$. This equality holds because
	\begin{align*}
	&\underline{\smash\phi}\big(\rho_3(Y_1,Y_2|X)\big)=-\Upsilon^G_{TM}(Y_1,Y_2,X)=-\gamma(X,[Y_1,Y_2]),\\
	&\frakn_3\big(\phi(Y_1),\phi(Y_2)|\underline{\smash\phi}(X)\big)=-\underline{\smash\phi}(X)([Y_1,Y_2])={-}\gamma^{\flat}(X)([Y_1,Y_2])={-}\gamma(X,[Y_1,Y_2]).
	\end{align*}
	We now proved that the relations~\eqref{toshow2} hold, so the proof is finished.
\end{proof}

When $\gamma$ is closed, the strict $L_{\infty}[1]$-morphism $\phi$  from Proposition \ref{prop:strict_L_infty-algebra_morphism} is obtained from the pullback by a Courant algebroid isomorphism, as we now explain.

\begin{proposition}\label{rem:conceptualclosed}
Assume that $\gamma$ is closed. Then $\phi$ is induced by an isomorphism of Courant algebroids $$F:(\bbT M,\ldab-,-\rdab,\ldsb-,-\rdsb_G,\rho_G)\overset{\sim}{\rightarrow}(\bbT M,\ldab-,-\rdab,\ldsb-,-\rdsb,\pr_{TM})$$ which transforms the splitting $\bbT M= T^{*}M\oplus TM$ into $\bbT M=(T\calF\oplus T^\circ\calF)\oplus(G\oplus G^\circ)$.    
\end{proposition}

 \begin{proof}
 Recall that 
\begin{itemize}
    \item $(\mathfrak{X}_{\good}^\bullet(M)[2],\{\frakl^G_k\})$ is an $L_{\infty}[1]$-subalgebra of $(\mathfrak{X}^\bullet(M)[2],\{\frakl^G_k\})$, and the latter arises {as in Lemma \ref{lem:2.6first}} from the splitting $\bbT M=T^{*}M\oplus TM$ of the Courant algebroid $(\bbT M,\ldab-,-\rdab,\ldsb-,-\rdsb_G,\rho_G)$
(see Lemma~\ref{lem:Koszul_algebra:Courant_iso}),
\item $(\Omega^\bullet(\calF;G)[1],\{\frakv_k\})$ is an $L_{\infty}[1]$-subalgebra of $(\Gamma(\wedge^\bullet(T^\ast\calF\oplus G))[2],\{\frakn_k\})$, and the latter arises from the splitting $\bbT M=(T\calF\oplus T^\circ\calF)\oplus(G\oplus G^\circ)$ of the standard Courant algebroid $(\bbT M,\ldab-,-\rdab,\ldsb-,-\rdsb,\pr_{TM})$ (see Lemma \ref{lem:G_infty_algebra:foliation1}). 
\end{itemize}   

The map $F$ in question is given by the following formula:
\[
F:\bbT M\rightarrow\bbT M:X+\alpha\mapsto (\Pi^{\sharp}\alpha+\pr_GX)+(\iota_{X}\gamma+\pr_{G^{*}}\alpha).
\]

 \emph{Claim: $F$ is a Courant algebroid isomorphism.} 
 
One readily checks that $F$ is an orthogonal transformation of $(\bbT M,\ldab-,-\rdab)$, regardless of $\gamma$ being closed or not. Clearly $F$ matches the anchors $\rho_G$ and $\pr_{TM}$, it takes $T^{*}M$ to $(T\calF\oplus T^\circ\calF)$ and $TM$ to $(G\oplus G^\circ)$. It remains to show that $F$ takes $\ldsb-,-\rdsb_G$ to $\ldsb-,-\rdsb$, and this is where the condition $d\gamma=0$ comes into play. To facilitate this computation, recall from Lemma~\ref{lem:Koszul_algebra:Courant_iso} that $\calR_\Pi\circ\calR_\gamma$ intertwines $\ldsb-,-\rdsb_G$ and $\ldsb-,-\rdsb$, so it is enough to show that $F\circ \calR_{-\gamma}\circ\calR_{-\Pi}$ preserves $\ldsb-,-\rdsb$. This map is easily described:
\[
F\circ \calR_{-\gamma}\circ\calR_{-\Pi}:\bbT M\rightarrow \bbT M:X+\alpha\mapsto X+\alpha+\iota_{X}\gamma,
\]
{i.e. it agrees with the gauge transformation $\calR_{\gamma}$. It is well-known that $\calR_{\gamma}$ preserves the Courant bracket $\ldsb-,-\rdsb$ if{f} $\gamma$ is closed.}
This 
confirms our claim that $F:(\bbT M,\ldab-,-\rdab,\ldsb-,-\rdsb_G,\rho_G)\overset{\sim}{\rightarrow}(\bbT M,\ldab-,-\rdab,\ldsb-,-\rdsb,\pr_{TM})$ is a Courant algebroid isomorphism.
\vspace{0.2cm}

 It follows that we obtain a strict $L_{\infty}[1]$-isomorphism
\begin{equation}\label{eq:F}
\wedge^\bullet F|_{T^\ast M}^\ast:(\Gamma(\wedge^\bullet(T^\ast\calF\oplus G))[2],\{\frakn_k\})\overset{\sim}{\longrightarrow}(\mathfrak{X}^\bullet(M)[2],\{\frakl^G_k\}).
\end{equation}
We will now invert this map and restrict to the suitable $L_{\infty}[1]$-subalgebras. First note that 
\[
F^{-1}|_{T\calF\oplus G^{*}}:T\calF\oplus G^{*}\rightarrow T^{*}M:X+\alpha\mapsto\alpha-\iota_{X}\gamma
\]
and consequently  the dual map reads
\[
F^{-1}|_{T\calF\oplus G^{*}}^\ast:TM\rightarrow T^\ast\calF\oplus G:X\mapsto \ldab X+\iota_{X}\gamma,-\rdab|_{T\calF\oplus G^{*}}. 
\]
Clearly this map takes $T\calF$ to $T^{*}\calF$ and it takes $G$ to itself. So by Remark \ref{rem:good} there is an induced map
\[
\wedge^\bullet F^{-1}|_{T\calF\oplus G^{*}}^\ast:\frakX^{\bullet}_{\good}(M)[2]\rightarrow \Gamma(\wedge^{\bullet}T^{*}\calF)[2]\oplus\Gamma(\wedge^{\bullet}T^{*}\calF\otimes G)[1],
\]
which is a strict $L_{\infty}[1]$-isomorphism since it is a restriction of the inverse of the map  \eqref{eq:F}.

\vspace{0.2cm} 
 
\emph{Claim: The composition of the above map with the projection, 
\begin{equation}\label{eq:strictmorph}
P\circ\wedge^\bullet F^{-1}|_{T\calF\oplus G^{*}}^\ast:(\frakX_{\good}^\bullet(M)[2],\{\frakl^G_k\})\rightarrow(\Gamma(\wedge^{\bullet}T^{*}\calF\otimes G)[1],\{\frakv_k\}),
\end{equation}
is a strict $L_{\infty}[1]$-morphism. }

It suffices to show that the projection
$P: (\Gamma(\wedge^{\bullet}T^{*}\calF)[2]\oplus\Gamma(\wedge^{\bullet}T^{*}\calF\otimes G)[1]\rightarrow\Gamma(\wedge^{\bullet}T^{*}\calF\otimes G)[1]$ intertwines\footnote{We mention as an aside that the full projection 
$\Gamma(\wedge^\bullet(T^\ast\calF\oplus G))[2]\rightarrow\Gamma(\wedge^{\bullet}T^{*}\calF\otimes G)[1]$ does not intertwine the brackets.}
the multibrackets  $\{\frakn_k\}$ and $\{\frakv_k\}$.
  To show this, we make the following remarks:
\begin{itemize}
    \item Because $\frakn_1$ has bi-degree $(1,0)$, it follows that $P(\frakn_1\xi)=\frakv_1(P(\xi))$ for all $\xi$.
    \item Because
    {$\frakn_2$ has bi-degree $(0,-1)$, it follows that 
    \begin{equation}\label{eq:n2v2} P(\frakn_2(\xi,\eta))=\frakv_2(P(\xi),P(\eta)) 
    \end{equation}
 for all
    $\xi,\eta$ in $\Gamma(\wedge^{\bullet}T^{*}\calF)[2]\oplus\Gamma(\wedge^{\bullet}T^{*}\calF\otimes G)[1]$. Explicitly, since}     $\frakn_2(\Gamma(G),\Gamma(T^{*}\calF))\subset\Gamma(T^{*}\calF)$ and $\frakn_2(\Gamma(T^{*}\calF),\Gamma(T^{*}\calF))=0$, we have
    {that both terms in eq. \eqref{eq:n2v2} vanish}
    whenever both $\xi,\eta$ have bi-degree $(\bullet,0)$, or when $\xi$ has bi-degree $(\bullet,0)$ and $\eta$ has bi-degree $(\bullet,1)$. If {both $\xi$ and $\eta$}
    have bi-degree $(\bullet,1)$, then
    \[
    P(\frakn_2(\xi,\eta))=\frakn_2(\xi,\eta)=\frakv_2(\xi,\eta)=\frakv_2(P(\xi),P(\eta)).
    \]
\end{itemize}
These remarks prove the claim that the map~\eqref{eq:strictmorph} is a strict $L_{\infty}[1]$-morphism.

\vspace{0.2cm}

 \emph{Claim: the map $P\circ\wedge^\bullet F^{-1}|_{T\calF\oplus G^{*}}^\ast$ defined in \eqref{eq:strictmorph} coincides with the map $\phi$ defined in   \eqref{eq:rem:expression:phi}.} 
 
 Let $Y_1\wedge\cdots\wedge Y_k\in\mathfrak{X}_{\good}^{k}(M)$ and choose $X_1,\ldots,X_{k-1}\in\Gamma(T\calF)$ and $\beta\in\Gamma(G^{*})$. Then
\begin{align*}
\left\langle (\wedge^k F^{-1}|_{T\calF\oplus G^{*}}^\ast Y_1\wedge\cdots\wedge Y_k)(X_1,\ldots,X_{k-1}),\beta\right\rangle&=\langle (F^{-1})^{*}Y_1\wedge\cdots\wedge(F^{-1})^{*}Y_k,X_1\wedge\cdots\wedge X_{k-1}\otimes\beta\rangle\\
&=\begin{vmatrix}
\langle(F^{-1})^{*}Y_1,X_1\rangle & \cdots & \langle(F^{-1})^{*}Y_1,X_{k-1}\rangle & \langle(F^{-1})^{*}Y_1,\beta\rangle \\
\vdots & & \vdots & \vdots\\
\langle(F^{-1})^{*}Y_k,X_1\rangle & \cdots & \langle(F^{-1})^{*}Y_k,X_{k-1}\rangle & \langle(F^{-1})^{*}Y_k,\beta\rangle 
\end{vmatrix}\\
&=\begin{vmatrix}
\langle\gamma^{\flat}Y_1,X_1\rangle & \cdots & \langle\gamma^{\flat}Y_1,X_{k-1}\rangle & \beta(Y_1)\\
\vdots & & \vdots & \vdots\\
\langle\gamma^{\flat}Y_k,X_1\rangle & \cdots & \langle\gamma^{\flat}Y_k,X_{k-1}\rangle & \beta(Y_k)
\end{vmatrix}\\
&=(-1)^{k-1}\begin{vmatrix}
\langle Y_1,\gamma^{\flat}X_1\rangle & \cdots & \langle Y_1,\gamma^{\flat}X_{k-1}\rangle & \beta(Y_1)\\
\vdots & & \vdots &\vdots\\
\langle Y_k,\gamma^{\flat}X_1\rangle & \cdots & \langle Y_k,\gamma^{\flat}X_{k-1}\rangle & \beta(Y_k)
\end{vmatrix}\\
&=(-1)^{k-1}Y_1\wedge\cdots\wedge Y_k(\gamma^{\flat}X_1,\ldots,\gamma^{\flat}X_{k-1},\beta)\\
&=\left\langle\phi(Y_1\wedge\cdots\wedge Y_k)(X_1,\ldots,X_{k-1}),\beta\right\rangle.
\end{align*}
This shows that $P\circ\wedge^\bullet F^{-1}|_{T\calF\oplus G^{*}}^\ast$ agrees with $\phi$, as claimed. In particular, the latter is a strict $L_{\infty}[1]$-morphism.
 \end{proof}

\bibliographystyle{plain}

\end{document}